\documentclass[10pt]{amsart}

\usepackage{amsthm, amsmath, amssymb, mathrsfs, enumitem, svn, amsfonts, stmaryrd, mathtools}

 \usepackage[a4paper, left=1.5cm, right=1.5cm, top=0.6cm,bottom=0.7cm]{geometry}
 
 \usepackage[linktocpage=true, hidelinks, colorlinks, linkcolor=blue, citecolor=magenta, bookmarksopen=true, urlcolor=brown]{hyperref}
 
\usepackage{mathscinet}
\usepackage{tikz-cd} 
\usepackage{tabu}  
\usepackage[all, cmtip]{xy}

\newtheorem{theorem}[subsubsection]{Theorem}
\newtheorem{thm}[subsubsection]{Theorem}
\newtheorem{lemma}[subsubsection]{Lemma}
\newtheorem{lem}[subsubsection]{Lemma}
\newtheorem{cor}[subsubsection]{Corollary}

\newtheorem{prop}[subsubsection]{Proposition}

\theoremstyle{definition}
\newtheorem{defn}[subsubsection]{Definition}

\newtheorem{definition}[subsubsection]{Definition}

\newtheorem{notation}[subsubsection]{Notation}
\newtheorem{Notation}[subsubsection]{Notation}
\newtheorem{convention}[subsubsection]{Convention}
\newtheorem{remark}[subsubsection]{Remark}
\newtheorem{rem}[subsubsection]{Remark}

\def \inj {\hookrightarrow }
\def \to {\rightarrow}

\renewcommand{\projlim}{\varprojlim}

\def\Mat{\mathrm{Mat}}
\def\mat{\mathrm{Mat}}

\DeclareMathOperator{\Mod}{Mod}

\DeclareMathOperator{\Rep}{Rep}

\DeclareMathOperator{\Hom}{Hom}

\DeclareMathOperator{\Ker}{Ker}

\DeclareMathOperator{\Gal}{Gal}
\DeclareMathOperator{\gal}{Gal}

\def\inf{\mathrm{inf}}
\def\sup{\mathrm{sup}}

\DeclareMathOperator{\Fr}{Fr}

\DeclareMathOperator{\Fil}{Fil}

\def\et{\mathrm{\acute{e}t}}

\def\cris{\mathrm{cris}}
\def\crys{\mathrm{crys}}

\newcommand{\st}{\mathrm{st}}

\def\dR{\mathrm{dR}}

\def\rig{\mathrm{rig}}

\newcommand{\ur}{\mathrm{ur}}

\newcommand{\Kpinfty}{K_{p^\infty}}
\newcommand{\Kinfty}{K_{\infty}}

\newcommand{\la}{\mathrm{la}}
\newcommand{\dan}{\text{$\mbox{-}\mathrm{an}$}}

\newcommand{\dla}{\text{$\mbox{-}\mathrm{la}$}}
\newcommand{\dpa}{\text{$\mbox{-}\mathrm{pa}$}}

\newcommand{\MF}{\mathrm{MF}^{\varphi, N}_{K_0}}
\newcommand{\MFwa}{\mathrm{MF}^{\varphi, N, \mathrm{wa}}_{K_0}}

\newcommand{\Ainf}{\mathbf{A}_{\mathrm{inf}}}
\newcommand{\Acris}{\mathbf{A}_{\mathrm{cris}}}
\newcommand{\Bcris}{\mathbf{B}_{\mathrm{cris}}}

\newcommand{\bcris}{\mathbf{B}_{\mathrm{cris}}}
\newcommand{\Bst}{\mathbf{B}_{\mathrm{st}}}

\newcommand{\bdr}{\mathbf{B}_{\mathrm{dR}}}

\renewcommand{\OE}{{\mathcal{O}_{\mathcal{E}}}}
\newcommand{\oedagger}{{\mathcal{O}^\dagger_{\mathcal{E}}}}

\newcommand*{\wt}[1]{\widetilde{#1}}
\newcommand*{\wh}[1]{\widehat{#1}}

\newcommand{\Z}{\mathbb{Z}}

\newcommand{\Zp}{\mathbb{Z}_p}
\newcommand{\Qp}{\mathbb{Q}_p}

\def\upi{\underline \pi}
\newcommand{\bigM}{\mathcal{M}}

\newcommand{\cm}{\mathcal{M}}

\newcommand{\gs}{\mathfrak{S}}
\newcommand{\huaM}{\mathfrak{M}}
\newcommand{\gM}{\mathfrak{M}}
\newcommand{\gm}{\mathfrak{M}}
\def \hR {{\widehat{\mathcal R}} }
\def \hM {{\widehat{\mathfrak{M}}} }

\newcommand{\A}{ {\mathbf{A}}   }
\newcommand{\B}{  {\mathbf{B}}  }
\newcommand{\wtA}{   {\widetilde{{\mathbf{A}}}}  }
\newcommand{\wtE}{   {\widetilde{{\mathbf{E}}}}  }
\newcommand{\wta}{   {\widetilde{{\mathbf{A}}}}  }
\newcommand{\wtb}{   {\widetilde{{\mathbf{B}}}}  }
\newcommand{\wtB}{   {\widetilde{{\mathbf{B}}}}  }

\newcommand{\vece}{\vec{e}}
\newcommand{\LHS}{\text{LHS}}

\newcommand{\bst}{\mathbf{B}_{\mathrm{st}}}

\renewcommand{\phi}{\varphi}

\renewcommand{\projlim}{\varprojlim}

\newcommand{\cbf}{\mathbf{c}}

\newcommand{\kbf}{\mathbf{k}}

\def \hR {\widehat{\mathcal{R}} }
\def \Kinfty {K_{\infty}}
\def \Kpinfty {K_{p^\infty}}

\def \O {\mathcal{O}}

\usepackage{relsize}
\usepackage[bbgreekl]{mathbbol}
\usepackage{amsfonts}

\DeclareSymbolFontAlphabet{\mathbb}{AMSb}
\DeclareSymbolFontAlphabet{\mathbbl}{bbold}
\newcommand{\prism}{{\mathlarger{\mathbbl{\Delta}}}}

\numberwithin{equation}{subsection}
\begin{document}

\title[]{Breuil-Kisin modules and integral $p$-adic Hodge theory}


\date{\today}

\author[]{Hui Gao\\
with  Appendix A by Yoshiyasu Ozeki, \\
and joint Appendix B with Tong Liu}
 
\address{Department of Mathematics, Southern University of Science and Technology, Shenzhen 518055, Guangdong, China}
\email{gaoh@sustech.edu.cn}

\address{Department of Mathematics, Purdue University, West Lafayette, IN 47907}
\email{tongliu@math.purdue.edu}

\address{Department of Mathematics and Physics, Faculty of Science, Kanagawa University,
  2946 Tsuchiya, Hiratsuka-shi, Kanagawa 259--1293, Japan}
\email{ozeki@kanagawa-u.ac.jp}

\begin{abstract}
We construct a category of \emph{Breuil-Kisin $G_K$-modules} to classify  integral semi-stable Galois representations. Our theory uses Breuil-Kisin modules and  Breuil-Kisin-Fargues modules with Galois actions, and can be regarded as the algebraic avatar of the integral $p$-adic cohomology theories of Bhatt-Morrow-Scholze and Bhatt-Scholze. As a key ingredient, we classify Galois representations that are of finite $E(u)$-height.
\end{abstract}

\maketitle
\setcounter{tocdepth}{1}
\tableofcontents

\section{Introduction}\label{secintro}
\subsection{Overview and main theorems}
In this paper, we construct certain ``algebraic avatar" of some \emph{integral} $p$-adic cohomology theories. Let us first fix some notations.
\begin{notation}\label{notabegin}
Let  $p$ be a prime. Let $k$ be a perfect  field of characteristic $p$, let $W(k)$ be the ring of Witt vectors, and let $K_0 :=W(k)[1/p]$.
Let $K$ be a totally ramified finite extension of $K_0$, let $\mathcal O_K$ be the ring of integers, and let $e := [K: K_0]$.
Fix an algebraic closure $\overline {K}$ of $K$ and set $G_K:=\Gal(\overline{K}/K)$.
Let $C_p$ be the $p$-adic completion of $\overline{K}$, and let $\mathcal{O}_{C_p}$ be the ring of integers. Let $v_p$ be the valuation on $C_p$ such that $v_p(p)=1$.
\end{notation}

 The study of $p$-adic Hodge theory is roughly divided into two closely related directions: in the geometric direction, one studies $p$-adic cohomology theories and their comparisons; while in the algebraic direction, one studies (semi-)linear algebra categories and uses them to classify $p$-adic Galois representations. Let us first recall some  foundational theorems.

\begin{theorem} \cite{HK94, Tsu99, CN17}
\label{thmTsuji}
Let $\mathfrak{X}$ be a proper formal scheme over $\mathcal{O}_K$ with semi-stable reduction. Let  $R\Gamma_{\mathrm{log-crys}}$ (resp. $R\Gamma_{\dR}$,  $R\Gamma_{\et}$) denote the log-crystalline (resp. de Rham,  \'etale) cohomology theory. 
Let $\bst$ be Fontaine's semi-stable period ring.
We have
\begin{eqnarray}\label{eq01}
R\Gamma_{\mathrm{log-crys}}(\mathfrak{X}_k/W(k))[1/p]  \otimes_{K_0} K  & \simeq &  R\Gamma_{\dR}(\mathfrak{X}/\mathcal O_K)[1/p];\\
\label{eq02} R\Gamma_{\mathrm{log-crys}}(\mathfrak{X}_k/W(k))[1/p]  \otimes_{K_0} \Bst & \simeq   &
R\Gamma_{\et}(\mathfrak{X}_{\overline K}, \Zp)[1/p]  \otimes_{\Qp} \Bst .
\end{eqnarray} 
\end{theorem}
Here, when $\mathfrak{X}$ is a scheme,   \eqref{eq01} is proved by Hyodo-Kato, and   \eqref{eq02} (the $\mathrm{C}_{\st}$-conjecture of Fontaine-Jannsen) is proved by Tsuji; the general case when $\mathfrak{X}$ is a formal scheme  is proved by Colmez-Nizio\l.
These comparisons imply that the log-crystalline cohomology groups equipped with their natural $(\varphi, N)$- and filtration structures give rise to \emph{admissible}  filtered $(\varphi, N)$-modules (recall admissible means ``coming from Galois representations"). The following theorem of Colmez-Fontaine, which says that ``weakly admissible implies admissible", can  be regarded as the algebraic avatar of Thm. \ref{thmTsuji}.

\begin{theorem}\cite{CF00} \label{thmcf}
The category of \emph{weakly admissible} filtered $(\varphi, N)$-modules is equivalent to the category of semi-stable representations of $G_K$. 
\end{theorem}

Note that  in Thm. \ref{thmTsuji}, the integral and torsion information in all the cohomology theories are lost in the comparison theorems. Recently, Bhatt-Morrow-Scholze \cite{BMS1, BMS2} and Bhatt-Scholze \cite{BS22} defined some \emph{integral}  $p$-adic cohomology theories (in the good reduction case) which specialize to all these classical $p$-adic cohomology theories (and keep  all the integral  and torsion information). Let us first introduce some notations and definitions.

\begin{notation}\label{114}
Let $\pi \in   K$ be a \emph{fixed} uniformizer, and let $E(u) \in W(k)[u]$ be the minimal polynomial of $\pi$ over $K_0$. \emph{Fix} a sequence of elements $\pi _n \in \overline K$ inductively such that $\pi_0 = \pi$ and $(\pi_{n+1})^p = \pi_n$. Let $K_{\infty} : = \cup _{n = 1} ^{\infty} K(\pi_n)$ and let $G_\infty:= \gal(\overline K/\Kinfty)$.
Let $\gs  : =  W(k)[\![u]\!]$.

Let $R:= \mathcal{O}_{C_p}^\flat$ be the tilt of the perfectoid ring $\mathcal{O}_{C_p}$, let $\mathfrak{m}_R$ be its maximal ideal, and let $W(R)$ (also denoted as $\Ainf$) be the ring of Witt vectors. Let $\theta: W(R) \to  \mathcal{O}_{C_p}$ be the usual map; let $\xi$  be a generator of the principal ideal $\Ker \theta$.
Let $\Fr R:=C_p^\flat$ be the fractional field of $R$, and let $W(\Fr R)$ be the ring of Witt vectors.

Note that $\upi:=(\pi_n)_{n \geq 0} $ defines an element in $R$; let $[\underline \pi ]\in W(R)$ be the Teichm\"uller lift of $\upi$. Then we can define a $W(k)$-linear  embedding $\gs \inj W(R)$ via $u\mapsto [\underline \pi]$; hence $E(u)$ maps to a generator of $\Ker \theta$. The Frobenius $\varphi: R \to R$ where $x \mapsto x^p$ induces Frobenii (always denoted as $\varphi$) on $W(R)$ and $\gs$.

Fix $\mu _n \in \overline K$ inductively such that $\mu_0=1$, $\mu_1$ is a primitive $p$-th root of unity and $(\mu_{n+1})^p = \mu_n$. Let $\Kpinfty: = \cup _{n = 1} ^{\infty} K(\mu_n)$. 
Let $\underline{\varepsilon} = (1, \mu_1, \mu_{2}, \ldots) \in R$,  and let $[\underline{\varepsilon}] \in W(R)$ be the Teichm\"{u}ller lift.

\end{notation}

\begin{defn} \label{d115}
\hfill
\begin{enumerate} 
\item A \emph{Breuil-Kisin module} is a finitely generated $\gs$-module  $\gm $  equipped with an $\gs[ 1/E(u)]$-linear isomorphism 
\[\gm \otimes_{\varphi, \gs}\gs[ 1/E(u)] \to \gm[ 1/E(u)] ;\] it is called \emph{of (non-negative) finite $E(u)$-height} if under the isomorphism, the image of $\gm\otimes_{\varphi, \gs}\gs$ is contained in $\gm$.

\item Let $\wh\gm $ be a  finitely presented $W(R)$-module such that $\wh\gm [1/p]$ is finite free over $W(R)[1/p]$.
\begin{enumerate}
\item It is called a (non-$\varphi$-twisted) \emph{Breuil-Kisin-Fargues module} if it is equipped with a $W(R)[1/\xi]$-linear isomorphism 
\[\wh\gm \otimes_{\varphi, W(R)}W(R)[1/\xi]  \to \wh\gm[1/\xi].\]
\item It is called a \emph{$\varphi$-twisted Breuil-Kisin-Fargues module} if it is equipped with a $W(R)[1/\varphi(\xi)]$-linear isomorphism 
\[\wh\gm \otimes_{\varphi, W(R)}W(R)[1/\varphi(\xi)]  \to \wh\gm[1/\varphi(\xi)].\]
\end{enumerate}
\end{enumerate}
\end{defn}

\begin{rem}\label{rfirst}
\begin{enumerate}
\item Given a (non-$\varphi$-twisted) Breuil-Kisin-Fargues module $\wh\gm$, then $\wh\gm\otimes_{\varphi, W(R)} W(R)$ is a $\varphi$-twisted Breuil-Kisin-Fargues module. As $\varphi$ is an automorphism on $W(R)$, this induces an equivalence between the relevant categories. 

\item  It seems to us that both versions of ``Breuil-Kisin-Fargues modules" deserve their merits. Indeed, the $\varphi$-twisted version (which is precisely \cite[Def. 4.22]{BMS1}) is perhaps the (geometrically)  natural  version as it naturally appears in cohomology and comparison theorems, cf. Thm. \ref{bmsintro} in the following. 
However, the non-$\varphi$-twisted version (which, e.g., is also used in \cite[Def. 4.2.1]{EG23}) has the technical convenience that it is more ``parallel" to the Breuil-Kisin modules: for example, one can choose $\xi=E(u)$. In addition, in our \emph{algebraic} study of Breuil-Kisin modules, we also need to embed $\gs$ into various other rings (without $\varphi$-twisting), hence the process is more uniform if we use the non-$\varphi$-twisted version throughout: this is indeed what we do in this paper, cf. Def. \ref{defwr11} and Rem. \ref{remscholze}.
\end{enumerate} 
\end{rem}

\begin{theorem}\cite{BMS1, BMS2, BS22} \label{bmsintro}
Let $\mathfrak{X}$ be a proper smooth formal scheme over $\mathcal{O}_K$.
There exist cohomology theories $R\Gamma_{\gs}(\mathfrak{X})$ and $R\Gamma_{\Ainf}(\mathfrak{X}_{\mathcal{O}_{C_p}})$ which are equipped with morphisms $\varphi$, such that the cohomology groups are Breuil-Kisin modules and \emph{$\varphi$-twisted} Breuil-Kisin-Fargues modules respectively, and such that we have the following comparisons
\begin{eqnarray}
\label{compawr} R\Gamma_{\gs}(\mathfrak{X}) \otimes_{\varphi, \gs} \Ainf & \simeq &  R\Gamma_{\Ainf}(\mathfrak{X}_{\mathcal{O}_{C_p}}) ;\\
\label{compacris} R\Gamma_{\gs}(\mathfrak{X}) \otimes_{\varphi, \gs}^{\mathbb L} W(k) & \simeq &  R\Gamma_{\crys}(\mathfrak{X}_k/W(k));\\
\label{compadr} R\Gamma_{\gs}(\mathfrak{X}) \otimes_{\varphi, \gs}^{\mathbb L} \mathcal O_K  & \simeq &  R\Gamma_{\dR}(\mathfrak{X}/\mathcal O_K);\\
\label{compaet}R\Gamma_{\Ainf}(\mathfrak{X}_{\mathcal{O}_{C_p}}) \otimes_{\Ainf} \Ainf[  1/\mu]
& \simeq & R\Gamma_{\et}(\mathfrak{X}_{\overline K}, \Zp) \otimes_{\Zp} \Ainf[ 1/\mu].
\end{eqnarray}
Here, the (derived) tensor product in \eqref{compawr} is via the flat (cf. \cite[Lem. 4.30]{BMS1}) morphism $\gs \xrightarrow{\varphi} \gs \hookrightarrow \Ainf$ where the second map is the $W(k)$-linear embedding sending $u$ to $[\underline{\pi}]$;
the derived tensor product in \eqref{compacris} is via $\gs \xrightarrow{\varphi} \gs \to W(k)$
where the second map is the   $W(k)$-linear map sending  $u$ to $0$; the  derived  tensor product in \eqref{compadr} is via $\gs \xrightarrow{\varphi} \gs \to \mathcal{O}_K$ where the second map is the $W(k)$-linear map sending $u$ to $\pi$; the element $\mu$ in \eqref{compaet} is $[\underline{\varepsilon}]-1$.
 \end{theorem}

All the above results are first proved by Bhatt-Morrow-Scholze in \cite{BMS1, BMS2}.
Recently, Bhatt-Scholze \cite{BS22} developed the  prismatic cohomologies and reproved all these results; in particular, \eqref{compawr} can now be regarded as a  prismatic base-change theorem.
Let us mention that it is natural  to expect the existence of a ``log"-version of the  prismatic site, and hence the  log-prismatic cohomologies (e.g., the semi-stable version of $R\Gamma_{\Ainf}$ is already constructed by \v{C}esnavi\v{c}ius-Koshikawa \cite{CK19}).

The main goal of this paper is to construct the algebraic avatar of these integral $p$-adic cohomologies (\emph{modulo $p$-power torsion}) and the comparisons amongst them, which is the following:

\begin{defn}\label{defwr11}
Let $\textnormal{Mod}_{\gs, W(R)}^{\varphi, G_K , [-\infty, +\infty]}$
be the category consisting of triples $(\mathfrak{M}, \varphi_{\gm}, G_K)$, which we call the \emph{Breuil-Kisin $G_K$-modules}, where
\begin{enumerate}
\item $(\mathfrak{M}, \varphi_\mathfrak{M})$ is a \emph{finite free} Breuil-Kisin module;
\item  $G_K$ is a continuous $\varphi_{\hM}$-commuting $W(R)$-semi-linear   $G_K$-action on the (non-$\varphi$-twisted) Breuil-Kisin-Fargues module  $\hM:= W(R)
\otimes_{ \gs} \mathfrak{M}$, such that
\begin{enumerate}
\item $\gm \subset \hM^{G_\infty}$ via the embedding $\gm \hookrightarrow \hM$;
\item  $\gm/u\gm \subset (\hM/W(\mathfrak{m}_R)\hM)^{G_K}$ via the embedding $\gm/u\gm \hookrightarrow \hM/W(\mathfrak{m}_R)\hM$.
\end{enumerate} 
\end{enumerate}
  \end{defn}

  \begin{rem}
After choosing a  basis of $\gm$, the data in the definition above can be expressed using two matrices (always so if $p>2$ and can be made so if $p=2$), one for $\varphi$ and one for $\tau$ (see Notation \ref{nota hatG} for $\tau$ and discussions about $p=2$ case).
  \end{rem}
  

 \begin{rem}\label{remscholze} 
 \begin{enumerate}
 \item The module $\wh\gm$ together with its   $G_K$-action is called a ``Breuil-Kisin-Fargues $G_K$-module" in \cite[Def. 4.2.3]{EG23}. Here, we choose the terminology ``Breuil-Kisin  $G_K$-modules"  to emphasize the role played by Breuil-Kisin modules in our theory.

\item If we change the expression ``$\hM:= W(R)
\otimes_{ \gs} \mathfrak{M}$" in Def. \ref{defwr11} to ``$\hM:= W(R)
\otimes_{\varphi, \gs} \mathfrak{M}$", then by Rem. \ref{rfirst}, we get an equivalent ``$\varphi$-twisted category". 
 
 \item (I thank Peter Scholze for useful discussions on this remark.)  
One observes that the cohomology theories $R\Gamma_{\gs}(\mathfrak{X})$ and $R\Gamma_{A_{\mathrm{inf}}}(\mathfrak{X}_{\mathcal{O}_{C_p}})$ in Thm. \ref{bmsintro} ``satisfy"   conditions in  the \emph{$\varphi$-twisted version} of Def. \ref{defwr11} (although the cohomology groups are not necessarily finite free).
Namely, there is a  $\varphi$-commuting $G_K$-action on $R\Gamma_{A_{\mathrm{inf}}}(\mathfrak{X}_{\mathcal{O}_{C_p}})$; via the isomorphism \eqref{compawr}, the image of $R\Gamma_{\gs}(\mathfrak{X})$ in $R\Gamma_{A_{\mathrm{inf}}}(\mathfrak{X}_{\mathcal{O}_{C_p}})$ is fixed by $G_\infty$, and the image of $R\Gamma_{\gs}(\mathfrak{X})\otimes_{\varphi, \gs}^{\mathbb L} W(k)$ is furthermore fixed by $G_K$. 
All these follow from the  constructions in \cite{BMS1, BMS2}, but they look most natural using the \emph{base change theorem} of prismatic cohomologies in \cite{BS22}. Indeed, we have (cf. \emph{loc. cit.} for notations about prisms and prismatic cohomology),
\begin{align} 
 R\Gamma_{A_{\mathrm{inf}}}(\mathfrak{X}_{\mathcal{O}_{C_p}}) &\simeq   \varphi^\ast_{\Ainf} R\Gamma_\prism(\mathfrak{X}_{\mathcal{O}_{C_p}}/\Ainf), \quad \text{ by \cite[\S 17]{BS22};  }  \\
R\Gamma_{\gs}(\mathfrak{X})&\simeq   R\Gamma_\prism(X/\gs), \quad \text{ by \cite[\S 15.2]{BS22}.  } 
\end{align}
(In particular, note that $R\Gamma_\prism(\mathfrak{X}_{\mathcal{O}_{C_p}}/\Ainf)$ in \cite{BS22} gives \emph{non-$\varphi$-twisted} Breuil-Kisin-Fargues modules.)
Now, the $\varphi$-commuting action of $G_K$ on $R\Gamma_\prism(\mathfrak{X}_{\mathcal{O}_{C_p}}/\Ainf)$ is induced by $G_K$-action on the prism $(\Ainf, (\xi)) \in (\mathcal{O}_{C_p})_\prism$. 
The morphism $\varphi: \gs  \to \Ainf$ induces a morphism of prisms in $(\O_K)_\prism$:
\begin{equation} \label{egsinto}
(\gs, (E(u))) \to (\Ainf, (\varphi(\xi))); 
\end{equation}
this morphism, by  base change theorem of prismatic cohomologies, induces an isomorphism 
\begin{equation}\label{egsintomod}
R\Gamma_\prism(X/\gs)\otimes_{\varphi, \gs} \Ainf \simeq 
 \varphi^\ast_{\Ainf} R\Gamma_\prism(\mathfrak{X}_{\mathcal{O}_{C_p}}/\Ainf),
\end{equation}
which then induces a $G_K$-action on the left hand side. 
Given $g \in G_\infty$,  the composite $\gs \xrightarrow{\varphi} \Ainf \xrightarrow{g} \Ainf$ induces   exactly the same morphism  in \eqref{egsinto} (since $\gs \subset (\Ainf)^{G_\infty}$) and hence exactly the same isomorphism in \eqref{egsintomod}; this implies that $R\Gamma_\prism(X/\gs)$ is fixed by $G_\infty$. 
Finally, by reduction modulo $W(\mathfrak{m}_R)$, \eqref{egsinto}  induces  a morphism of prisms in $(\O_K)_\prism$:
\begin{equation} \label{egsinto2}
(W(k), (p)) \to (W(\overline k), (p)).
\end{equation}
Since $W(k) \subset (W(\overline k))^{G_K}$, we deduce  again by prismatic base change theorem that $R\Gamma_\prism(X/\gs) \otimes^{\mathbb L}_{\varphi, \gs} W(k)$ is fixed by $G_K$.
 \end{enumerate}
\end{rem}


The following is our main theorem, which, per Rem. \ref{remscholze}, can be regarded as the algebraic avatar of   Thm. \ref{bmsintro}.

\begin{theorem}\label{thm113}
(cf. Thm. \ref{thmsimple} \footnote{In Thm. \ref{thmsimple},   we will stick with  representations of non-negative Hodge-Tate weights, as well as Breuil-Kisin modules of non-negative  $E(u)$-heights; this makes the writing easier. The general case (as stated in Thm. \ref{thm113}) can be easily deduced by twisting.})
The category of Breuil-Kisin $G_K$-modules is equivalent to the category of $G_K$-stable $\Zp$-lattices in semi-stable representations of $G_K$.
\end{theorem}

\begin{rem}\label{rem1111}
\begin{enumerate}
\item We have a crystallinity criterion to tell when a Breuil-Kisin $G_K$-module comes from a crystalline representation, cf. Prop. \ref{propcryscrit}.

\item Loosely speaking, Thm. \ref{thm113} ``implies" that the information in \eqref{compawr} and \eqref{compacris} are already \emph{enough} to ``recover" $R\Gamma_{\et}(\mathfrak{X}_{\overline K}, \Zp)$ (modulo torsion). 
Given a Breuil-Kisin $G_K$-module  as in Def. \ref{defwr11}, let $\varphi^\ast \gm:=\gs \otimes_{\varphi, \gs} \gm$ and let  $\varphi^\ast \hM: =W(R) \otimes_{\varphi, W(R)} \hM$, then we can show (cf. Prop. \ref{cor710})
\begin{equation*}
\varphi^\ast \gm/E(u)\varphi^\ast \gm \subset (\varphi^\ast\hM/E(u)\varphi^\ast\hM)^{G_K};
\end{equation*} 
this can be regarded as the algebraic avatar of the de Rham comparison \eqref{compadr}.
\end{enumerate}
\end{rem}

We now give some historical remarks about the study of  (algebraic) integral $p$-adic Hodge theory, and compare some of the theories.

\begin{rem} \label{1110}
\begin{enumerate}
\item In  algebraic  integral $p$-adic Hodge theory, we use various (semi-)linear objects to classify $\Zp$-lattices in semi-stable Galois representations. 
For example, we have   Fontaine-Laffaille  theory \cite{FL82}, 
the theory of Wach modules \cite{Wach-smf, Wach-compo} (then refined by \cite{Col99, Ber04}),
and Breuil's theory of strongly divisible $S$-lattices (conjectured in \cite{Bre02}, fully proved in \cite{Liu08, Gao17} using input from \cite{Kis06}).
However, these theories are valid only with certain restrictions on ramification of the base field or on Hodge-Tate weights, or  are valid only for certain crystalline representations. Liu's theory of $(\varphi, \hat{G})$-modules \cite{Liu10} (with input from Kisin \cite{Kis06}) is so far the only theory that works for \emph{all} integral  semi-stable representations. However, unlike our Breuil-Kisin $G_K$-modules which can be regarded as the algebraic avatar of some cohomology theories, the ring ``$\hR$" (cf. Appendix \ref{secappB}) in Liu's theory is too \emph{implicit}, and it seems hopeless to construct some cohomology theory over it. 

 
\item Indeed, the author and Tong Liu recently realized that we actually \emph{do not know} if the ring ``$\hR$" is $p$-adically complete or not: this means that there is a gap in our earlier work on limit of torsion semi-stable Galois representations  in \cite{Liu07, Gao18limit} (these results are not used in the current paper). Fortunately, this gap can now be (easily) fixed by using the Breuil-Kisin $G_K$-modules, cf. Appendix \ref{secappB}. Note that the gap arises in the \emph{application} of the theory of $(\varphi, \hat{G})$-modules (namely, the theory is inadequate for this application); the theory \emph{per se} remains valid: see also next item.

\item In \cite{GaoLAV}, we will show that the theory  of Breuil-Kisin $G_K$-modules \emph{specializes to} (and hence recovers)  the theory of $(\varphi, \hat{G})$-modules; cf. \S \ref{72} for more remarks.  Let us mention here that the proof of our main theorem is   independent of the theory of $(\varphi, \hat{G})$-modules (except some relatively easy results, e.g. Prop. \ref{propweakff}); in particular, we will not use $\hR$ anywhere. 
\end{enumerate}
\end{rem}
 

We make some speculations about the theory and its possible applications.   

\begin{rem}
\begin{enumerate}
\item Like all the integral theories listed in Rem. \ref{1110}(1), we would like to use our Breuil-Kisin $G_K$-modules to study reduction of semi-stable Galois representations as well as the relevant semi-stable Galois deformation rings (these results always play important roles in automorphy lifting theorems). 
In fact, our result can already at least simplify some of the constructions of the semi-stable sub-stack in Emerton-Gee's stack of $(\varphi, \Gamma)$-modules, cf. \S \ref{73}. We particularly would like to use our theory to investigate the explicit structures of some semi-stable Galois deformation rings.

\item In some sense, one can also regard Thm. \ref{thm113} as some sort of integral version of the Colmez-Fontaine theorem, particularly because all the modules in both theories are avatars of cohomology theories. 
It is interesting to speculate if our theory can play some \emph{integral} role in places where Colmez-Fontaine theorem is used  (e.g., in the study of $p$-adic period domains and period morphisms). We also wonder if there is any connection between our theory and the Fargues-Fontaine curve.

\item  Recently, Bhatt and Scholze \cite{BS23} established an equivalence between the category of ``prismatic $F$-crystals on $(\O_K)_\prism$" and the category of lattices in crystalline representations of $G_K$. (The semi-stable version is then proved by Du-Liu \cite{DL23} using  the log-prismatic site of Koshikawa \cite{Kos21}).  
It seems that the \emph{direct} link between their theory and ours is still unclear: in particular, can we directly construct a prismatic $F$-crystal from a   Breuil-Kisin $G_K$-module, and vice versa?   
It seems likely explorations about these question could shed new light on prismatic crystals.
\end{enumerate}
\end{rem}


Let us now sketch the main ideas in the proof of Thm. \ref{thm113}. Indeed, a key ingredient is the classification of  Galois representations which are of finite $E(u)$-height.

\begin{defn}\label{defintroht}
Let $T$ be a finite free $\Zp$-representation of $G_K$, it is called \emph{of finite $E(u)$-height} if there exists a finite free Breuil-Kisin module $\gm$ of non-negative finite $E(u)$-height such that there is a $G_\infty$-equivariant isomorphism 
\begin{equation} \label{eqftintro}
T|_{G_\infty} \simeq  (\gM \otimes_{\gs} W(\Fr R))^{\varphi=1},
\end{equation}
where the $G_\infty$-action on right hand side of \eqref{eqftintro} comes from that on $W(\Fr R)$.
\end{defn}

The following theorem in particular  answers positively \cite[Question 4.3.1(2)]{Liu10} by Tong Liu. 

\begin{theorem}\label{thmpstintr} (= Thm. \ref{thmmainpst})
Let $K^{\mathrm{ur}} \subset \overline{K}$ be the maximal unramified extension of $K$, let
\[m : =1+\max \{i \geq 1, \mu_{i} \in K^{\mathrm{ur}}\},\]
and let $K_m=K(\pi_{m-1})$.
Let $T$ be a finite free $\Zp$-representation of $G_K$, and let $V=T[1/p]$. Then $T$ is of finite $E(u)$-height  {if and only if} $V|_{G_{K_m}}$ can be extended to a semi-stable $G_K$-representation with non-negative Hodge-Tate weights.
In particular, if $T$ is of finite $E(u)$-height, then $V$ is potentially semi-stable. 
\end{theorem}

Theorem \ref{thmpstintr} is the most difficult part of our paper, and indeed takes the majority of space. In fact, once Theorem \ref{thmpstintr} is proved, it is then relatively  straightforward to prove the main theorem Thm. \ref{thm113} by using Thm. \ref{thmpstintr} and   ``comparisons" among various modules.



Thus, we will dedicate the entire Subsection \S \ref{subsecstrategy} below to explain the proof of Thm. \ref{thmpstintr}; before we do so, we list some remarks about this theorem.

\begin{rem}
The notion ``of finite $E(u)$-height" indeed depends on the choices $\{\pi_n\}_{n \geq 0}$ (as does the embedding $\gs \hookrightarrow W(R)$). If $T$ is of finite height with respect to \emph{all}   such choices, then we can use Thm. \ref{thmpstintr} to show that $V$ is semi-stable; this intriguing result is due to Gee, cf. Thm. \ref{thmtwou}. Indeed, this result, as we learnt from Gee, is inspired by considerations in the construction of the  stack of (semi-stable) $(\varphi, \Gamma)$-modules in \cite{EG23}; cf. \S \ref{73} for some more comments.
\end{rem}

\begin{rem}\label{remWach}
 In \cite{Wach-smf}, Wach studied some ``finite height" $(\varphi, \Gamma)$-modules; it is shown in \cite[A.5]{Wach-smf} that these ``finite height" $(\varphi, \Gamma)$-modules give rise to de Rham (indeed, potentially crystalline) Galois representations if there is some additional condition on the Lie algebra operator associated to the $\Gamma$-action.
      Note however that the  ``finite height" condition in \emph{loc. cit.} is of a  different type than the one here. Indeed, the analogue of our  ``finite $E(u)$-height" condition in the setting of $(\varphi, \Gamma)$-modules should be the ``finite $q$-height" condition, where $q:=\frac{(1+T)^p-1}{T}$ is a polynomial; cf. e.g., \cite{Ber04, KisRendocumenta}.
      In fact, we can use similar ideas  in the current paper to study ``finite $q$-height" $(\varphi, \Gamma)$-modules; in parallel, we can also study the ``finite height" $(\varphi, \tau)$-modules (without $E(u)$ in the play) similar as in \cite{Wach-smf}. All these will be discussed elsewhere.
\end{rem}

\begin{rem}\label{remcar13}
  Caruso gave a proof of Thm. \ref{thmpstintr} in \cite[Thm. 3]{Car13} (for $p>2$), which unfortunately contains a  rather serious gap. (Indeed, even the statement in \emph{loc. cit.} contains an error, cf. Rem. \ref{remwrongstatement}.) 
  The gap is first discovered by Yoshiyasu Ozeki, and is discussed in Appendix \ref{secapp}. The gap arises when Caruso tries to define a  \emph{monodromy operator} on the  \emph{$(\varphi, \tau)$-modules} (cf. below) associated to finite $E(u)$-height Galois representations, using the ``truncated $\log$" (cf. \cite[\S 3.2.2]{Car13}); very roughly speaking, Caruso tries to define a $\log$-operator via \emph{$p$-adic approximation} technique (using $p$-adic topology on various rings). As will be explained in next subsection, our approach is \emph{completely different} and has no $p$-adic approximation technique; in particular, our method uses \emph{overconvergent} $(\varphi, \tau)$-modules which are made available only very recently.
 Indeed, we do not regard our proof as a ``fix" to Caruso's proof, cf. also Rem.  \ref{remroadmap} and  Rem. \ref{r127}.
\end{rem}

\subsection{Strategy of proof of Thm. \ref{thmpstintr}}\label{subsecstrategy}
The main tool to prove Thm. \ref{thmpstintr} is the theory of \emph{overconvergent $(\varphi, \tau)$-modules}. 
Let us first give some general remarks about the theory of $(\varphi, \tau)$-modules.
 Recall that we already defined $K_{\infty}   = \cup _{n = 1} ^{\infty} K(\pi_n)$ and $\Kpinfty  = \cup _{n = 1} ^{\infty} K(\mu_n)$. Let
   $L:= \Kinfty\Kpinfty.$

As already mentioned in the last subsection, in the algebraic study of $p$-adic Hodge theory, we use various ``linear algebra" tools to study $p$-adic representations of $G_K$.
 A key idea in $p$-adic Hodge theory is to first restrict the Galois representations to some subgroup  of $G_K$. The $(\varphi, \tau)$-modules used in this paper are constructed  by using the subgroup $G_{\infty}:= \gal (\overline K / K_{\infty})$; they are analogues of the more classical $(\varphi, \Gamma)$-modules which are constructed   using the subgroup $G_{p^\infty}:= \gal (\overline K / K_{p^\infty})$.
 Here let us only quickly mention that the ``$\Gamma$" is the group $\gal(\Kpinfty/K)$, and the ``$\tau$" is a topological generator of the group $\gal(L/\Kpinfty)$ (cf. Notation \ref{nota hatG}).


 Similar to the $(\varphi, \Gamma)$-modules, the $(\varphi, \tau)$-modules also classify \emph{all} $p$-adic representations of $G_K$. Although these two theories are equivalent, they each have their own technical advantage, and both are indispensable. The $(\varphi, \Gamma)$-modules are perhaps ``easier" in the sense that both the $\varphi$- and $\Gamma$-actions are defined over a same ring; whereas the $\tau$-action in $(\varphi, \tau)$-modules can only be defined over a much bigger ring.
 However, the $\varphi$-action in $(\varphi, \tau)$-modules stay tractable even when $K$ has ramification; in contrast,  the $\varphi$-action in $(\varphi, \Gamma)$-modules become quite implicit when $K$ has ramification.
 This dichotomy becomes much more substantial when we consider \emph{semi-stable} Galois representations: in this situation, there exist very well-behaved Breuil-Kisin modules (also called Kisin modules, or $(\varphi, \hat{G})$-modules in different contexts) which is a special type of $(\varphi, \tau)$-modules; in contrast, such special type $(\varphi, \Gamma)$-modules (called Wach modules) exist only if we consider crystalline   representations and if $K \subset \cup_{n \geq 1}K_0(\mu_n)$ (e.g., when $K=K_0$ is unramified).
 To save space, we refer the reader  to the introduction in \cite{GP21} for some discussion and comparison of the applications of these two theories in different contexts. 
 Indeed, this paper shows once again that when we consider semi-stable representations, it is fruitful to use the $(\varphi, \tau)$-modules.
 
Let us be more precise now.
First,  recall that the $\varphi$-action  of a $(\varphi, \tau)$-module is defined over the field
\begin{equation}\label{eqbkinf}
\mathbf{B}_{\Kinfty}   :=\{ \sum_{i=-\infty}^{+\infty} a_i u^i : a_i \in K_0, \lim_{i \to -\infty}v_p(a_i) =+\infty, \text{ and } \inf_{i \in \mathbb{Z}}v_p(a_i) >-\infty \}.
\end{equation}
 The $\tau$-action   is defined over a bigger filed ``$\wt{\mathbf{B}}_L$" which we do not recall here, see \S \ref{secmod}.
Indeed, roughly speaking, a  $(\varphi, \tau)$-module $D$ is a finite free $\mathbf{B}_{\Kinfty}$-vector space equipped with  certain commuting maps $\varphi: D \to D$ and $\tau: \wt{\mathbf{B}}_L\otimes_{\mathbf{B}_{\Kinfty} } D \to \wt{\mathbf{B}}_L\otimes_{\mathbf{B}_{\Kinfty} } D$.
By \cite[Thm. 1]{Car13}, the (\'etale) $(\varphi, \tau)$-modules classify all Galois representations of $G_K$. One readily observes that a Galois representation is of finite $E(u)$-height as in Def. \ref{defintroht} if and only if there exists a $\varphi$-stable ``Breuil-Kisin lattice" inside the corresponding $(\varphi, \tau)$-module; note that this property has nothing to do with the $\tau$-action. 


To prove  Thm. \ref{thmpstintr}, we need to define a natural monodromy operator on these $(\varphi, \tau)$-modules. Instead of the $p$-adic rings mentioned in Rem. \ref{remcar13}, what we propose in the current paper is that one should use certain Fr\'echet  rings (e.g.,   various ``Robba-style" rings) instead.
In fact, we can get a monodromy operator \emph{directly} (no approximation needed) using techniques of \emph{locally analytic vectors}. Furthermore, our monodromy operator will be defined for \emph{all} (rigid-overconvergent, cf. Thm. \ref{introOC} below) $(\varphi, \tau)$-modules (not just finite $E(u)$-height ones). Before we state the theorem concerning the monodromy operator, let us recall the overconvergence result of $(\varphi, \tau)$-modules.

\begin{theorem}[\cite{GL20, GP21}] \label{introOC}
The $(\varphi, \tau)$-modules (attached to  $p$-adic representations of $G_K$) are overconvergent. That is, (roughly speaking), the $\varphi$-action can be defined over the sub-field:
\begin{equation}\label{eqbkinfd}
\mathbf{B}_{K_\infty}^\dagger: =\{ \sum_{i=-\infty}^{+\infty} a_i u^i \in \mathbf{B}_{K_\infty}, \lim_{i \to -\infty}(v_p(a_i) +i\alpha)= +\infty \text{ for some } \alpha >0  \};
\end{equation}
 also, the $\tau$-action can  be defined over some sub-field $\wtB_L^\dagger \subset \wtB_L$.
\end{theorem}

\begin{rem}\label{rem2pf}
\begin{enumerate}
\item Thm. \ref{introOC} is first conjectured by Caruso in \cite{Car13} (as an analogue of the classical overconvergence theorem for $(\varphi, \Gamma)$-modules by Cherbonnier-Colmez in \cite{CC98}). A first proof (which only works for $K/\Qp$ a finite extension) is given in a joint work with Liu \cite{GL20}, using a certain ``crystalline approximation" technique; later a second proof (which works for all $K$) is given in a joint work with Poyeton in \cite{GP21}, using the idea of locally analytic vectors.
\item Let us mention that it is the second proof in \cite{GP21} that will be useful in the current paper. Not only because it works for all $K$ (which is a minor issue), but also more importantly, the idea of locally analytic vectors will be very critically used in the current paper to define the monodromy operator.
\end{enumerate}
\end{rem}

Let us introduce the following Robba ring (which contains $\mathbf{B}_{K_\infty}^\dagger$),
\begin{equation}\label{eqbrig}
\begin{split}
\mathbf{B}_{\rig, K_\infty}^\dagger: =\{f(u)= \sum_{i=-\infty}^{+\infty} a_i u^i, a_i \in K_0,   f(u) \text{ converges } \\
 \text{ for all } u \in \overline{K} \text{ with } 0<v_p(u)<\rho(f) \text{ for some } \rho(f)>0\}.
\end{split}
\end{equation}
Let $V$ be a $p$-adic Galois representation of $G_K$, and let $D^\dagger_{\Kinfty}(V)$ be the overconvergent $(\varphi, \tau)$-module associated to $V$ by Thm. \ref{introOC}.
Define
\[D^\dagger_{\rig, \Kinfty}(V):=\mathbf{B}_{\rig, K_\infty}^\dagger\otimes_{\mathbf{B}_{K_\infty}^\dagger}D^\dagger_{\Kinfty}(V),\]
which we call the \emph{rigid-overconvergent $(\varphi, \tau)$-module} associated to $V$; as we will see in the following, it is the natural space where the monodromy operator lives in.

\begin{theorem} \label{thmintmo}
(= Thm. \ref{thmnnabla}) \footnote{L\'eo Poyeton informed the author that he also obtained Thm. \ref{thmintmo} independently. }
 Let
 $\nabla_\tau:= (\log \tau^{p^n})/{p^n}$ for $n\gg 0$ be the Lie-algebra operator with respect to the $\tau$-action, and define
 $N_\nabla:=\frac{1}{p\mathfrak{t}}\cdot \nabla_\tau$
 where $\mathfrak{t}$ is a certain  ``normalizing"  element (cf. \S \ref{secmono}). (Note that there might be   some modification  in certain cases when $p=2$). Then
\[N_\nabla(D^\dagger_{\rig, \Kinfty}(V)) \subset D^\dagger_{\rig, \Kinfty}(V).\]
Namely, $N_\nabla$ is a well-defined  monodromy operator on $D^\dagger_{\rig, \Kinfty}(V)$.
\end{theorem}

\begin{rem}
\begin{enumerate}
\item In comparison, if we use $D^\dagger_{\rig, \Kpinfty}(V)$ (denoted as ``$D^\dagger_{\rig, K}(V)$" in \cite{Ber02}) to denote the rigid-overconvergent $(\varphi, \Gamma)$-module associated to $V$ (which exists by \cite{CC98}), then one can \emph{easily} define a monodromy operator
\[\nabla_V: D^\dagger_{\rig, \Kpinfty}(V) \to D^\dagger_{\rig, \Kpinfty}(V) \]
as in \cite[\S 5.1]{Ber02}. Here ``$\nabla_V$" (notation of \cite{Ber02}) is precisely the Lie-algebra operator associated to $\Gamma$-action.


\item The difficulty in defining $N_\nabla$ for $(\varphi, \tau)$-modules is that $\tau$ (hence $\nabla_\tau$) does not act on  $D^\dagger_{\rig, \Kinfty}(V)$ itself (whereas $\Gamma$ acts directly on $D^\dagger_{\rig, \Kpinfty}(V)$); the action is  defined only when we base change $D^\dagger_{\rig, \Kinfty}(V)$ over a much bigger ring ``$\wt{\mathbf{B}}_{  \rig, L}^{\dagger}$" (cf. Def. \ref{defrigring}). Fortunately, after  dividing $\nabla_\tau$ by $p\mathfrak{t}$, and using ideas of locally analytic vectors, one gets back to the level of $D^\dagger_{\rig, \Kinfty}(V)$.
\end{enumerate}
\end{rem}

Now, to prove Thm. \ref{thmpstintr}, via results of Kisin (and some consideration of locally analytic vectors), it suffices to show the following ``monodromy descent" result, which we achieve via a ``Frobenius regularization" technique.

\begin{prop}\label{p1}
(= Prop. \ref{propregfrob})
Let $\gm$ be the finite height Breuil-Kisin lattice inside a  $(\varphi, \tau)$-module corresponding to a $G_K$-representation of finite $E(u)$-height. Then
\[N_\nabla(\gm) \subset \mathcal{O}\otimes_{\gs} \gm.\]
Here $\mathcal{O} \subset \mathbf{B}_{\rig, K_\infty}^\dagger$ is the subring consisting of $f(u)$ that converges for all $u \in \overline{K}$ such that $0<v_p(u) \leq +\infty$.
\end{prop}

\begin{rem}\label{remroadmap}
Indeed, the ``road map" of our proof of Thm. \ref{thmpstintr} is roughly the same as in \cite{Car13}. Namely, one first defines a certain monodromy operator, then one shows that (in the finite $E(u)$-height case) the operator can be   defined over the smaller ring $\mathcal{O}$.
However, even as Thm. \ref{thmintmo}   provides a correct alternative  in defining the monodromy operator, the technical details in the latter half of our argument (Prop. \ref{p1}, proved in \S \ref{subsfrobreg}) are also completely different from that of Caruso.
Indeed, Caruso's argument uses several newly-defined rings (all with $p$-adic topology), cf. ``the upper half" of \cite[p. 2583, Figure 2 ]{Car13}; as far as we know, these rings have not been used elsewhere in the literature.
In comparison, all the rings we use in \S \ref{subsfrobreg} are already studied in \cite{GP21}; in particular, they are all natural analogues of the rings used in $(\varphi, \Gamma)$-module theory, which have been substantially studied since their introduction in e.g., \cite{Ber02}. Indeed, it seems that our argument is much easier  and   more natural.
\end{rem}

\begin{rem}\label{r127}
As discussed in Rem. \ref{remroadmap}, we do not know if we can actually \emph{fix} the gap in Caruso's work. That is, we do not know if we can use $p$-adic approximation technique to fix \cite[Prop. 3.7]{Car13} (cf. Appendix \ref{secapp}); we further do not know if we can use the $p$-adic argument as in \cite{Car13} to prove Thm. \ref{thmpstintr}.
\end{rem}

 \begin{rem}\label{remindep}
As a final remark, let us mention that  the current paper is (almost completely) independent of  \cite{Car13}.
The only exception is that we do use Caruso's category of \'etale $(\varphi, \tau)$-modules and its equivalence with the category of $p$-adic Galois representations (i.e., the content of  \cite[Thm. 1]{Car13}); but these are easy consequences of the theory of field of norms (with respect to the field $\Kinfty$), which e.g., was already partially developed in \cite[\S 2]{Bre99comp}.
We refer to Rem. \ref{remend} for some more comments regarding the relation between the current paper and \cite{Car13}.
   \end{rem}

\subsection{Structure of the paper}\label{subsecstr}
In \S \ref{secring}, we review many period rings in $p$-adic Hodge theory; in particular we compute locally analytic vectors in some rings.
In \S \ref{secmod}, we review the theory of $(\varphi, \tau)$-modules and the overconvergence theorem. 
In \S \ref{secmono}, we define the monodromy operator on the rigid-overconvergent $(\varphi, \tau)$-modules.
In \S \ref{secsemist}, we review  Kisin's theory of $\mathcal{O}$-modules (for semi-stable representations) and show that the monodromy operator there \emph{coincides} with ours in \S \ref{secmono}.
In \S \ref{secfinht}, when the $(\varphi, \tau)$-module is of finite $E(u)$-height, we use a Frobenius regularization technique to descend the monodromy operator to $\mathcal{O}$; this implies that the attached representation is potentially semi-stable.
In \S \ref{secsimp}, we  construct the Breuil-Kisin $G_K$-modules and prove our main theorem; we also compare our theory with some results of  Gee and Liu. 


\subsection{Some notations and conventions}
\begin{notation} \label{nota fields}
Recall that   we already defined:
\[K_{\infty}   = \cup _{n = 1} ^{\infty} K(\pi_n), \quad K_{p^\infty}=  \cup _{n=1}^\infty
K(\mu_{n}), \quad L =  \cup_{n = 1} ^{\infty} K(\pi_n, \mu_n).\]
Let \[G_{\infty}:= \gal (\overline K / K_{\infty}), \quad G_{p^\infty}:= \gal (\overline K / K_{p^\infty}), \quad G_L: =\gal(\overline K/L), \quad \hat G: =\gal (L/K) .\]
When $Y$ is a ring with a $G_K$-action, $X \subset \overline{K}$ is a subfield, we use $Y_X$ to denote the $\gal(\overline{K}/X)$-invariants of  $Y$; we will use the cases when $X=L, K_\infty$.
\end{notation}

\subsubsection{Locally analytic vectors}
Let us very quickly recall the theory of locally analytic vectors, see \cite[\S 2.1]{BC16} and \cite[\S 2]{Ber16} for more details. Indeed, almost all  the explicit calculations of locally analytic vectors used in this paper are already carried out in \cite{GP21}, hence the reader can refer to \emph{loc. cit.} for more details.

Recall that a $\Qp$-Banach space $W$ is a $\Qp$-vector space with a complete non-Archimedean  norm $\|\cdot\|$ such that $\|aw\|=\|a\|_p\|w\|, \forall a \in \Qp, w \in W$, where $\|a\|_p$ is the   $p$-adic norm on $\Qp$.
Recall the multi-index notations: if $\cbf = (c_1, \hdots,c_d)$ and $\kbf = (k_1,\hdots,k_d) \in \mathbb{N}^d$ (here $\mathbb{N}=\mathbb{Z}^{\geq 0}$), then we let $\cbf^\kbf = c_1^{k_1} \cdot \ldots \cdot c_d^{k_d}$.

Let $G$ be a $p$-adic Lie group, and let $(W, \|\cdot \|)$ be a $\Qp$-Banach representation of $G$.
Let $H$ be an open subgroup of $G$ such that there exist coordinates $c_1,\hdots,c_d : H \to \Zp$ giving rise to an analytic bijection $\cbf : H \to \Zp^d$.
 We say that an element $w \in W$ is an $H$-analytic vector if there exists a sequence $\{w_\kbf\}_{\kbf \in \mathbb{N}^d}$ with $w_\kbf \to 0$ in $W$, such that \[g(w) = \sum_{\kbf \in \mathbb{N}^d} \cbf(g)^\kbf w_\kbf, \quad \forall g \in H.\]
Let $W^{H\dan}$ denote the space of $H$-analytic vectors.
We say that a vector $w \in W$ is \emph{locally analytic} if there exists an open subgroup $H$ as above such that $w \in W^{H\dan}$. Let $W^{G\dla}$ denote the space of such vectors. We have $W^{G\dla} = \cup_{H} W^{H\dan}$ where $H$ runs through  open subgroups of $G$.
We can naturally extend these definitions to the case when $W$ is a Fr\'echet- or LF- representation of $G$, and use $W^{G\dpa}$ to denote the \emph{pro-analytic} vectors, cf. \cite[\S 2]{Ber16}.



\begin{Notation} \label{nota hatG}
Let $\hat{G}=\gal(L/K)$ be as in Notation \ref{nota fields}, which is a $p$-adic Lie group of dimension 2. We recall the structure of this group in the following.
\begin{enumerate}
\item Recall that:
\begin{itemize}
\item if $K_{\infty} \cap K_{p^\infty}=K$ (always valid when $p>2$, cf. \cite[Lem. 5.1.2]{Liu08}), then $\gal(L/K_{p^\infty})$ and $\gal(L/K_{\infty})$ generate $\hat{G}$;
\item if $K_{\infty} \cap K_{p^\infty} \supsetneq K$, then necessarily $p=2$, and $K_{\infty} \cap K_{p^\infty}=K(\pi_1)$ (cf. \cite[Prop. 4.1.5]{Liu10}) and $\pm i \notin K(\pi_1)$, and hence $\gal(L/K_{p^\infty})$ and $\gal(L/K_{\infty})$   generate an open subgroup  of $\hat{G}$ of index $2$.
\end{itemize}
Let us mention already that when $\Kinfty \cap \Kpinfty=K(\pi_1)$, some modifications might be needed in some of our argument,  notably with respect to the $\tau$-operator (cf. \eqref{eq1tau} below), and to the $N_\nabla$-operator (cf. \S \ref{secmono}). (As a side-note, when $p=2$, by \cite[Lem. 2.1]{Wangxiyuan}  we can always choose \emph{some} $\{\pi_n\}_{n \geq 0}$ so that $K_{\infty} \cap K_{p^\infty}=K$.)

\item Note that:
\begin{itemize}
\item $\gal(L/K_{p^\infty}) \simeq \Zp$, and let
$\tau \in \gal(L/K_{p^\infty})$ be \emph{the} topological generator such that
\begin{equation} \label{eq1tau}
\begin{cases} 
\tau(\pi_i)=\pi_i\mu_i, \forall i \geq 1, &  \text{if }  \Kinfty \cap \Kpinfty=K; \\
\tau(\pi_i)=\pi_i\mu_{i-1}=\pi_i\mu_{i}^2, \forall i \geq 2, & \text{if }  \Kinfty \cap \Kpinfty=K(\pi_1).
\end{cases}
\end{equation}
 
\item $\gal(L/K_{\infty})$ ($\subset \gal(K_{p^\infty}/K) \subset \Zp^\times$) is not necessarily pro-cyclic when $p=2$; however, this issue will \emph{never} bother us in this paper.
\end{itemize}
\end{enumerate}
\end{Notation}




\begin{Notation} \label{notataula}
We set up some notations with respect to representations of $\hat{G}$.
\begin{enumerate}
\item Given a $\hat{G}$-representation $W$, we use
\[W^{\tau=1}, \quad W^{\gamma=1}\]
to mean \[ W^{\gal(L/K_{p^\infty})=1}, \quad
W^{\gal(L/K_{\infty})=1}.\]
And we use
\[
W^{\tau\dla}, \quad W^{\tau\dan}, \quad W^{\tau_n\dan} (\text{ for } n \geq 1), \quad  W^{\gamma\dla} \]
to mean
\[
W^{\gal(L/K_{p^\infty})\dla}, \quad
W^{\gal(L/K_{p^\infty})\dan}, \quad
W^{<\tau^{p^n}>\dan}, \quad
W^{\gal(L/K_{\infty})\dla},  \]
where $<\tau^{p^n}> \subset \gal(L/K_{p^\infty})$ is the subgroup topologically generated by $\tau^{p^n}$.


\item  Let
$W^{\tau\dla, \gamma=1}:= W^{\tau\dla} \cap W^{\gamma=1},$
then by \cite[Lem. 3.2.4]{GP21}
\[ W^{\tau\dla, \gamma=1} \subset  W^{\hat{G}\dla}. \]
\end{enumerate}
\end{Notation}



\begin{rem}
Note that we never define $\gamma$ to be an element of $\gal(L/K_\infty)$; although when $p>2$ (or in general, when $\gal(L/K_\infty)$ is pro-cyclic), we could have defined it as a topological generator of $\gal(L/K_\infty)$. In particular, although ``$\gamma=1$" might be slightly ambiguous (but only when $p=2$), we use the notation for brevity.
\end{rem}

\subsubsection{Co-variant functors, Hodge-Tate weights, Breuil-Kisin heights, and minus signs} \label{subsubHTwts}
\begin{itemize}
\item In this paper we will use many categories of modules and the functors relating them; we will always use \emph{co-variant} functors. This makes the comparisons amongst them easier (i.e., using tensor products, rather than $\Hom$'s).
\item For example, our $D_\st(V)$ is defined as $(V\otimes_{\Qp} \bst)^{G_K}$, and hence the Hodge-Tate weight of the cyclotomic character $\chi_p$ is $-1$.
\item  Indeed, in the main argument of the paper, we will focus on representations with \emph{non-negative} Hodge-Tate weights and Breuil-Kisin modules with \emph{non-negative} $E(u)$-heights. For example, the Breuil-Kisin module associated to $\chi_p^{-1}$  has $E(u)$-height $1$.
\item We will define several \emph{differential operators} in this paper, and we always remove minus signs (for our convenience)  in our choices: see in particular Rem. \ref{remNlu} for the $N$-operator and Rem. \ref{compaminus} for the $N_\nabla$-operator.
\end{itemize}



\subsubsection{Some other notations}
Throughout this paper, we reserve $\varphi$ to denote Frobenius operator. We sometimes add subscripts to indicate on which object Frobenius is defined. For example, $\varphi_\huaM$ is the Frobenius defined on $\huaM$. We always drop these subscripts if no confusion arises.  Given a homomorphism of rings $\varphi: A \to A$ and given an $A$-module $M$, denote $\varphi^{\ast}M: =A\otimes_{\varphi, A} M$.
We use $\Mat(A)$ to denote the set of matrices  with entries in $A$ (the size of the matrix is always obvious from context).
Let $\gamma_i(x):=\frac{x^i}{i!}$ be the usual divided power.

\subsection*{Acknowledgement}
First and foremost, the author thanks Yoshiyasu Ozeki for informing him of his extremely meticulous discovery (in early 2015) and for agreeing to write an appendix; his discovery serves as impetus and inspiration for a major part of this work.
The author thanks Tong Liu and L\'eo Poyeton for the useful collaborations in \cite{GL20, GP21}. 
The author thanks Matthew Emerton and Toby Gee for informing him that the results are related with their work on moduli stack of Galois representations.
In addition to the aforementioned mathematicians, the author further thanks Laurent Berger, Xavier Caruso, Heng Du, Peter Scholze, Xin Wan and Liang Xiao for useful discussions and correspondences.
Some part of work are carried out during stays in University of Helsinki, MPIM in Bonn, BICMR and MCM in Beijing, and SCMS in Shanghai; the author thanks these institutions for excellent working conditions. 
Finally, the author thanks the anonymous referees for some  useful suggestions and comments. 
This work was partially supported by the National Natural Science Foundation of China under agreement No. NSFC-12071201.

\section{Rings and locally analytic vectors}\label{secring}
In this section, we review some period rings in $p$-adic Hodge theory. In particular, we compute the locally analytic vectors in some rings. In \S \ref{subsecwte}, we review some basic period rings;  in \S \ref{22}, we discuss   variations of these rings with respect to extension of fields.
In \S \ref{subswtB} and \S \ref{subsec BI}, we define the rings $\wt{\mathbf{B}}^I$, ${\mathbf{B}}^I$ and study their $G_\infty$-invariants; in  \S \ref{subseclav}, we study the relation of these rings via  locally analytic vectors.
In \S \ref{subsecfrakt}, we   study  a locally analytic element $\mathfrak{t}$, which plays a useful role in the definition of our monodromy operator.
In \S \ref{sslog}, we review some log-rings.


\subsection{Some basic period rings}\label{subsecwte}
Let $\wt{\mathbf{E}}^+$ be the tilt of $\mathcal{O}_{C_p}$ (denoted as $R$ in Notation \ref{114}), and let $\wt{\mathbf{E}}:=\mathrm{Fr} \widetilde{\mathbf{E}}^+$ be the tilt of $C_p$.
 An element of $\wt{\mathbf{E}}$ can be uniquely represented by $(x^{(n)})_{n \geq 0}$ where $x^{(n)} \in C_p$ and $(x^{(n+1)})^{p}=(x^{(n)})$; let $v_{\wt{\mathbf{E}}}$ be the usual valuation where $v_{\wt{\mathbf{E}}}(x):=v_p(x^{(0)})$.
Let
\[\wt{\mathbf{A}}^+:= W(\wt{\mathbf{E}}^+), \quad  \wt{\mathbf{A}}:= W(\wt{\mathbf{E}}), \quad \wt{\mathbf{B}}^+:= \wt{\mathbf{A}}^+[1/p], \quad \wt{\mathbf{B}}:= \wt{\mathbf{A}}[1/p],\]
where $W(\cdot)$ means the ring of Witt vectors.
There is a natural Frobenius operator $x \mapsto x^p$ on $\wt{\mathbf{E}}$, which induces  natural  Frobenius operators (always denoted by $\varphi$) on all the rings defined above (and below); there are also natural $G_K$-actions on the rings defined above induced from that on $\wt{\mathbf{E}}$. 
Note that the $G_K$-action on $\wt{\mathbf{E}}$ is continuous with respect to the $v_{\wt{\mathbf{E}}}$-topology (but not the discrete topology); the action on $\wt{\mathbf{B}}$ is continuous with respect to the weak topology (but not  the strong $p$-adic topology).

Let $\upi:=\{\pi_n\}_{n \geq 0} \in \wt{\mathbf{E}}^+$.
Let $\mathbf{E}^+_{K_\infty} :=k[\![\upi]\!]$, $\mathbf{E}_{K_\infty} :=k((\upi))$, and let $\mathbf{E}$ be the separable closure of $\mathbf{E}_{K_\infty}$ in $\wt{\mathbf{E}}$.

Let $[\underline \pi ]\in \wt{\mathbf{A}}^+$ be the Teichm\"uller lift of $\upi$.
Let $\mathbf{A}^+_{K_\infty} : = W(k)[\![u]\!]$ with Frobenius $\varphi$ extending the arithmetic Frobenius on $W(k)$ and $\varphi (u ) = u ^p$.
There is a $W(k)$-linear Frobenius-equivariant embedding $\mathbf{A}^+_{K_\infty} \inj \wt{\mathbf{A}}^+$ via $u\mapsto [\underline \pi]$.
Let $\mathbf{A}_{K_\infty}$ be the $p$-adic completion of $\mathbf{A}^+_{K_\infty}[1/u]$.
Our fixed embedding $\mathbf{A}^+_{K_\infty}\hookrightarrow \wt{\mathbf{A}}^+$ determined by  $\upi$
uniquely extends to a $\varphi$-equivariant embedding $\mathbf{A}_{K_\infty} \hookrightarrow \wt{\mathbf{A}}$, and we identify $\mathbf{A}_{K_\infty}$ with its image in $\wt{\mathbf{A}}$.
We note that $\mathbf{A}_{K_\infty}$ is a complete discrete valuation ring with uniformizer $p$ and residue field
$\mathbf{E}_{K_\infty}$.

Let $\mathbf{B}_{K_\infty}:=\mathbf{A}_{K_\infty}[1/p]$ (which is precisely the field  in \eqref{eqbkinf}). 
Let $\mathbf{B}$ be the completion for the $p$-adic norm of the maximal unramified extension of $\mathbf{B}_{K_\infty}$ inside $\wt{\mathbf{B}}$, and let $\mathbf{A} \subset \mathbf{B}$ be the ring of integers.
Let $\mathbf{A}^+ : = \wt{\mathbf{A}}^+ \cap \mathbf{A}$. Then we have:
\[ (\mathbf{A})^{G_\infty} = \mathbf{A}_{K_\infty}, \quad (\mathbf{B})^{G_\infty} = \mathbf{B}_{K_\infty}, \quad
(\mathbf{A}^+)^{G_\infty} = \mathbf{A}^+_{K_\infty}.  \]

\subsection{Rings with respect to  field extensions}\label{22}
 (The discussions in this subsection will not be used until \S \ref{subsecmodE}.)
Note that rings $\wtE^+, \wtA^+, \wtA$ (and the ``$\B$-variants") depend only on   $C_p$, in the sense that if we let $E$ be another complete discrete valuation field with perfect residue field where $K \subset E \subset C_p$, then we get exactly the same $\wtE^+, \wtA^+, \wtA$ as if we started with $K$. However, the rings $\mathbf{E}, \A, \B$, although without subscripts, indeed \emph{depend} on $K$ and $K_\infty$. 
For example, let $K \subset E \subset C_p$ be as aforementioned, and choose some $E_\infty$ (analogue of $\Kinfty$), then in general we cannot   compare the newly constructed ``$\mathbf{E}, \A, \B$" with the ones constructed using $K$ and $\Kinfty$ (as we cannot even compare $\Kinfty$ and $E_\infty$; in general we do not even have $\Kinfty \subset E_\infty$). 

This is very different from the $(\varphi, \Gamma)$-module setting, where once we fix $\mu_n$ as in Notation \ref{nota fields}, then we can always make some comparison since we always have $\Kpinfty \subset E_{p^\infty}$ (and then we can apply the theory of field of norms). This is indeed used in e.g. \cite[\S II.4]{CC98}.

Fortunately, for our purpose in this paper, we only need to work with certain special case of $K \subset E \subset C_p$, where we can easily make some comparisons.

\begin{notation}\label{notaEK}
Let $K'/K$ be a (not necessarily finite) unramified extension contained in $\overline K$, and let $m\geq 0$. Let $E$ be the $p$-adic completion of $K'(\pi_m)$, and let $E_\infty: =\cup_{n \geq m}E(\pi_n)$.
Let $\underline{\pi_E}: = \{ \pi_n\}_{n \geq m} \in \wt{\mathbf E}^+$, and let $u_E \in \wt{\mathbf{A}}^+$ be its the Teichm\"uller lift.
Then we can analogously construct the
$\mathbf{E}_{E_\infty}, \mathbf{A}^+_{E_\infty}, \mathbf{A}_{E_\infty}$, and $\mathbf{E}(E), \A^+(E), \A(E)$, as well as the $\B$-variants of these rings. (Here we only use the notation $\mathbf{E}^+(E)$ etc. instead of $\mathbf{E}(E, E_\infty)$ etc. for brevity). Then indeed, 
\[\mathbf{E} \hookrightarrow \mathbf{E}(E);\] and the theory of Cohen rings then induces a map $\A \hookrightarrow \A(E)$.
Furthermore, we   have a natural embedding
\[\mathbf{A}_{K_\infty} \hookrightarrow \mathbf{A}_{E_\infty}\]
 using the embedding $W(k) \hookrightarrow W(k')$ (where $k'$ is the residue field of $K'$) and using $u \mapsto u_E^{p^m}$.
\end{notation}

\subsection{The rings $\wt{\mathbf{B}}^{I}$  and their $G_\infty$-invariants} \label{subswtB}
Recall that we defined the element $\underline{\varepsilon} = (1, \mu_1, \mu_{2}, \ldots) \in \wt{\mathbf{E}}^+$ in Notation \ref{114}.
Let $\overline \pi =\underline{\varepsilon} -1 \in \wt{\mathbf E}^+$ (this is not $\underline \pi$), and let $[\overline \pi] \in \wt{\mathbf{A}}^+$ be its Teichm\"uller lift.
When $ A$ is a $p$-adic complete ring, we use $  A\{X, Y\}$ to denote the $p$-adic completion of  $ A[X, Y]$.

\begin{defn}
For $r \in \mathbb{Z}^{\geq 0}[1/p]$, let $ \wt{\mathbf{A}}^{[r, +\infty]}: = \wt{\mathbf{A}}^+ \{\frac{p}{[\overline \pi]^r} \}$, which is a subring of $\wt{\mathbf{A}}$.
Here, to be rigorous, $\wt{\mathbf{A}}^+ \{ {p}/{[\overline \pi]^r}\}$ is defined as $\wt{\mathbf{A}}^+ \{X  \}/([\overline \pi]^rX-p)$, and similarly for  other similar occurrences later; see \cite[\S 2]{Ber02} for more details.
 Let $\wt{\mathbf{B}}^{[r, +\infty]}:=\wt{\mathbf{A}}^{[r, +\infty]}[1/p] \subset \wt{\mathbf{B}} $.
\end{defn}

\begin{defn} \label{defnew}
Suppose $r \in \mathbb{Z}^{\geq 0}[1/p]$, and let $x= \sum_{i \geq i_0} p^i[x_i] \in \wt{\mathbf{B}}^{[r, +\infty]}$ ($\subset \wt{\mathbf{B}} $). Denote $w_k(x) := \inf_{i \leq k} \{v_{\wt{\mathbf E}}(x_i)\}$ .
For $s\geq r$ and $s>0$, let
\[W^{[s, s]}(x) :=\inf_{k \geq i_0} \{k+\frac{p-1}{ps}\cdot v_{\wt{\mathbf E}}(x_k)\} =   \inf_{k \geq i_0} \{ k+\frac{p-1}{ps}\cdot w_k(x)\};\]
this is a well-defined valuation (cf. \cite[Prop. 5.4]{Col08}).
For $I \subset [r, +\infty)$ a non-empty closed interval such that $I \neq [0, 0]$, let
\[W^{I}(x) := \inf_{\alpha \in I, \alpha \neq 0} \{W^{[\alpha, \alpha]}(x) \}.\]
\end{defn}

\begin{rem}\label{rem donot defn}
We do not define ``$W^{[0, 0]}$", cf. \cite[Rem. 2.1.9]{GP21}.
\end{rem}

\begin{lemma} \label{lem W}
Suppose $r \leq s \in \mathbb{Z}^{\geq 0}[1/p]$ and $s>0$, then the following holds.
\begin{enumerate}


\item \label{item comp}When $r>0$, $\wt{\mathbf{A}}^{[r, +\infty]}$ and $\wt{\mathbf{A}}^{[r, +\infty]}[1/[\overline \pi]]$ are complete with respect to $W^{[r, r]}$.
\item  \label{item mult} $W^{[s, s]}(xy)=W^{[s, s]}(x) +W^{[s, s]}(y), \forall x, y \in \wt{\mathbf{B}}^{[r, +\infty]}$; namely, $W^{[s, s]}$ is multiplicative.

\item  \label{item max} Let $x \in \wt{\mathbf{B}}^{[r, +\infty]}$.
\begin{enumerate}
\item When $r>0$, $W^{[r, s]}(x) = \inf \{ W^{[r, r]}(x), W^{[s, s]}(x)\}.$ In particular, this implies that $W^{[r, s]}$ is sub-multiplicative.
\item When $r=0$, $W^{[r, s]}(x) (=W^{[0, s]}(x)) =   W^{[s, s]}(x).$
\end{enumerate}

\end{enumerate}
\end{lemma}

\begin{defn} \label{defnewi}
Let $r \in \mathbb{Z}^{\geq 0}[1/p]$.
\begin{enumerate}
\item Suppose $I=[r, s] \subset [r, +\infty)$ is a non-empty closed interval such that $I \neq [0, 0]$. Let $ \wt{\mathbf{A}}^I$ be the completion of $ \wt{\mathbf{A}}^{[r, +\infty]}$ with respect to $W^I$. Let $ \wt{\mathbf{B}}^I:=\wt{\mathbf{A}}^I[1/p]$.
\item Let
\[\wt{\mathbf{B}}^{[r, +\infty)}: = \cap_{n \geq 0} \wt{\mathbf{B}}^{[r, s_n]}\]
where $s_n  \in \mathbb{Z}^{>0}[1/p]$ is any sequence increasing to $+\infty$.  We equip $\wt{\mathbf{B}}^{[r, +\infty)}$ with its natural Fr\'echet  topology.
\end{enumerate}
\end{defn}

\begin{lem} \label{lemcompa} \cite[Lem. 2.1.10(4)]{GP21}
Let $I=[r, s]$ be a closed interval as above, let $V^I$ be the $p$-adic topology on $\wtb^I$ defined using $\wta^I$ as ring of integers. Then for any $x \in \wtb^I$, we have
$V^I(x) =\lfloor W^I(x) \rfloor.$
\end{lem}

 \begin{remark} \label{rem Ic}
  For  our purposes (indeed, also in other literature concerning these rings), it is only necessary to study (the explicit structure of) these rings when
\[ \inf I, \sup I \in \{ 0, +\infty, (p-1)p^{\mathbb{Z}} \}.\]
  Furthermore, for any interval $I$ such that $\wt{\mathbf{A}}^{I}$ and $\wt{\mathbf{B}}^{I}$ are defined, there is a natural bijection (called Frobenius) $\varphi: \wt{\mathbf{A}}^{I} \simeq \wt{\mathbf{A}}^{pI}$ which is  valuation-preserving. Hence in practice, it would suffice if we can determine the explicit structure of these rings for
 \[I_c: =\{[r_\ell, r_k], [r_\ell, +\infty], [0, r_k], [0, +\infty] \}, \text{ with } \ell \leq k \in \mathbb{Z}^{\geq 0},\]
 where $r_n: =(p-1)p^{n-1}$.
 The cases for $I$ a general closed interval can be deduced using Frobenius operation; the cases for $I=[r, +\infty)$ can be deduced by taking Fr\'echet completion.
\end{remark}

\begin{convention}
Throughout the paper, \emph{all} the intervals $I$ (over ``$\wtB$-rings", ``$\mathbf{B}$-rings," ``$D$-modules", etc.) satisfy
\begin{equation*}\label{eqninfI}
\inf I, \sup I \in \{ 0, +\infty, (p-1)p^{\mathbb{Z}} \}.
\end{equation*}
If they are not closed, then they are of the form $[0, +\infty)$ or $[r, +\infty)$.
$I$ is never allowed to be $[0, 0]$ (or ``$[+\infty, +\infty]$").
\end{convention}



\begin{lemma}\label{lemwta}
We have
\begin{eqnarray*}
 \wt{\mathbf{A}}^{[0, r_k]} &=& \wt{\mathbf{A}}^+ \{\frac{u^{ep^k}}{p}\}, \\
 \wt{\mathbf{A}}^{[r_\ell, +\infty]} &=& \wt{\mathbf{A}}^+ \{ \frac{p}{u^{ep^\ell}}\}, \\
 \wt{\mathbf{A}}^{[r_\ell, r_k]} &=& \wt{\mathbf{A}}^+ \{ \frac{p}{u^{ep^\ell}} , \frac{u^{ep^k}}{p}\}.
\end{eqnarray*}
\end{lemma}
\begin{proof}
Indeed, these equations are used as definitions in \cite[Def. 2.1.1]{GP21}; these definitions are equivalent to our current Def. \ref{defnewi} by Lem. \ref{lemcompa}. See \cite[\S 2.1]{GP21} for more details.
\end{proof}



\begin{prop}\label{lem decomp 3} \cite[Prop. 2.1.14]{GP21} Recall that the subscript  $\Kinfty$ signifies $G_\infty$-invariants. We have
\begin{eqnarray*}
 \wt{\mathbf{B}}_{K_\infty}^{[0, r_k]}   &=&  \wt{\mathbf{A}}^+_{K_\infty}\{ \frac{u^{ep^k}}{p} \}[\frac 1 p] \\
 \wt{\mathbf{B}}_{K_\infty}^{[r_\ell, +\infty]}  &=&   \wt{\mathbf{A}}^+_{K_\infty}\{ \frac{p}{u^{ep^\ell}}    \}[\frac 1 p] \\
 \wt{\mathbf{B}}_{K_\infty}^{[r_\ell, r_k]}   &=&   \wt{\mathbf{A}}^+_{K_\infty} \{ \frac{p}{u^{ep^\ell}}, \frac{u^{ep^k}}{p}  \}[\frac{1}{p}]
\end{eqnarray*}
\end{prop}

\subsection{The rings ${\mathbf B}^{I}$ and their $G_\infty$-invariants} \label{subsec BI}

\begin{definition}
Let $r \in \mathbb{Z}^{\geq 0}[1/p]$.
\begin{enumerate}
\item Let
\[\mathbf{A}^{[r, +\infty]} := \mathbf A \cap \wt{\mathbf{A}}^{[r, +\infty]}, \quad \mathbf{B}^{[r, +\infty]} :=\mathbf B \cap \wt{\mathbf{B}}^{[r, +\infty]} .\]
\item Suppose $ [r, s] \subset [r, +\infty)$ is a non-empty closed interval such that $s \neq 0$. Let $\mathbf{B}^{[r, s ]}$ be the closure of $\mathbf{B}^{[r, +\infty]}$  in $\wt{\mathbf{B}}^{[r, s]}$ with respect to $W^{[r, s]}$.
 Let $\mathbf{A}^{[r, s ]}: =\mathbf{B}^{[r, s ]} \cap \wt{\mathbf{A}}^{[r, s]}$, which is the ring of integers in $\mathbf{B}^{[r, s ]}$.
 \item Let
\[\mathbf{B}^{[r, +\infty)} : =\cap_{n \geq 0} \mathbf{B}^{[r, s_n]}\]
where $s_n  \in \mathbb{Z}^{>0}[1/p]$ is any sequence increasing to $+\infty$.
\end{enumerate}
\end{definition}

\begin{definition}
\begin{enumerate}
\item  For $r \in \mathbb{Z}^{\geq 0}[1/p]$, let $\mathcal{A}^{[r, +\infty]}(K_0)$ be the ring consisting of infinite series $f=\sum_{k \in \mathbb Z} a_kT^k$ where $a_k \in W(k)$ such that $f$ is a holomorphic function on the annulus defined by
\[v_p(T) \in  (0,  \quad     \frac{p-1}{ep}\cdot \frac{1}{r} ].\]
(Note that when $r=0$, it implies that $a_k=0, \forall k<0$).
Let $\mathcal{B}^{[r, +\infty]}(K_0): =\mathcal{A}^{[r, +\infty]}(K_0)[1/p] $.

\item Suppose  $f=\sum_{k \in \mathbb Z} a_kT^k \in \mathcal{B}^{[r, +\infty]}(K_0)$.
\begin{enumerate}
\item  When $s \geq r$, $s>0$, let
\[
\mathcal W^{[s, s]}(f) : =  \inf_{k \in \Z}  \{ v_p(a_k)  + \frac{p-1}{ps}\cdot   \frac{k}{e}     \}. \]

\item For $I  \subset [r, +\infty)$ a non-empty   closed  interval,
let
\[
\mathcal W^{I}(f) := \inf_{\alpha \in I, \alpha \neq 0}\mathcal W^{[\alpha, \alpha]}(f).
\]
\end{enumerate}

\item For $r\leq s \in \mathbb{Z}^{\geq 0}[1/p], s \neq 0$, let $\mathcal{B}^{[r, s]}(K_0)$ be the completion of $\mathcal{B}^{[r, +\infty]}(K_0)$ with respect to $\mathcal W^{[r, s]}$. Let $\mathcal{A}^{[r, s]}(K_0)$ be the ring of integers in $\mathcal{B}^{[r, s]}(K_0)$ with respect to $\mathcal W^{[r, s]}$.
\end{enumerate}

\end{definition}

\begin{lemma}
 \label{lem laurent series}
\begin{enumerate}
\item For $r>0$, $\mathcal{B}^{[r, +\infty]}(K_0)$ is complete with respect to $\mathcal W^{[r, r]}$, and $\mathcal{A}^{[r, +\infty]}(K_0)$ is the ring of integers with respect to this valuation.

\item  For $s>0$, we have $\mathcal W^{[0, s ]}(x) =  \mathcal W^{[s, s]}(x)$. Furthermore, $\mathcal{B}^{[0, s]}(K_0)$ is  the ring consisting of infinite series $f=\sum_{k \in \mathbb Z} a_kT^k$ where $a_k \in K_0$ such that $f$ is a holomorphic function on the closed disk defined by
\[v_p(T) \in  [ \frac{p-1}{ep}\cdot \frac{1}{s}, \quad +\infty ].\]

\item For $I=[r, s] \subset (0, +\infty)$, we have $\mathcal W^{I}(x) = \inf \{\mathcal W^{[r, r]}(x), \mathcal W^{[s, s]}(x)\}$. Furthermore, $\mathcal{B}^{[r, s]}(K_0)$ is  the ring consisting of infinite series $f=\sum_{k \in \mathbb Z} a_kT^k$ where $a_k \in K_0$ such that $f$ is a holomorphic function on the annulus defined by
\[v_p(T) \in  [ \frac{p-1}{ep}\cdot \frac{1}{s},  \quad \frac{p-1}{ep}\cdot \frac{1}{r} ].\]
\end{enumerate}
\end{lemma}
 \begin{proof}
In \cite[Lem. 2.2.5]{GP21}, we stated the results with $[r, s]=[r_\ell, r_k]$; but they are   true for general intervals.
 \end{proof}

\begin{lemma} \cite[Lem. 2.2.7]{GP21}
\label{lem inj}  Let $\mathbf A_{\Kinfty}^I$ be the $G_\infty$-invariants of $\mathbf A^I$.
The map $f(T) \mapsto f(u)$ induces isometric isomorphisms
\begin{eqnarray*}
 \mathcal{A}^{[0, +\infty]}(K_0)  &\simeq& \mathbf{A}^{[0,+\infty]}_{K_\infty}; \\
 \mathcal{A}^{[r, +\infty]}(K_0) &\simeq &  \mathbf{A}^{[r,+\infty]}_{K_\infty}[1/u], \text{ when } r>0;\\
  \mathcal{A}^{I}(K_0)  &\simeq &  \mathbf{A}^{I}_{K_\infty}, \text{ when } I \subset [0, +\infty) \text { is a  closed interval}. \\
\end{eqnarray*}
\end{lemma}

We record an easy corollary of the above explicit description of the rings $\mathbf{B}^{I}_{K_\infty}$.

\begin{cor}\label{corIpI}
Let  $I \subset [0, +\infty]$ be an interval.
Suppose $x \in  \mathbf{B}^{I}_{K_\infty}$ such that $\varphi(x) \in \mathbf{B}^{I}_{K_\infty}$, then we have $x \in  \mathbf{B}^{ I/p}_{K_\infty}$.
\end{cor}
\begin{proof}
Indeed, by Lem. \ref{lem inj}, $x =\sum_{i \in \mathbb Z}a_iu^i$ with $a_i \in K_0$ satisfying certain convergence conditions related with the interval $I$ as described in Lem. \ref{lem laurent series}. We have $\varphi(x) =\sum_{i \in \mathbb Z} \varphi(a_i)u^{pi}$, hence using the explicit  convergence condition, it is easy to see that $\varphi(x) \in \mathbf{B}^{I}_{K_\infty}$ if and only if $x \in  \mathbf{B}^{I/p}_{K_\infty}$.
\end{proof}

Finally, we write out the explicit structures of some of these rings.

 \begin{prop} \cite[Prop. 2.2.10]{GP21} \label{corArlrk}
We have
\begin{eqnarray*}
 \mathbf{A}^{[0, +\infty]}_{K_\infty} &=& \mathbf{A}^+_{K_\infty}, \\
 \mathbf{A}_{K_\infty}^{[0, r_k]} &=& \mathbf{A}^+_{K_\infty} \{\frac{u^{ep^k}}{p}\}, \\
 \mathbf{A}_{K_\infty}^{[r_\ell, +\infty]} &=& \mathbf{A}^+_{K_\infty} \{ \frac{p}{u^{ep^\ell}}\}, \\
  \mathbf{A}_{K_\infty}^{[r_\ell, r_k]} &=& \mathbf{A}^+_{K_\infty} \{ \frac{p}{u^{ep^\ell}} , \frac{u^{ep^k}}{p}\}.
\end{eqnarray*}
\end{prop}

 \begin{rem}
 Note that  $\varphi: \wt{\mathbf{B}}^I \to \wt{\mathbf{B}}^{pI}$ is always a bijection; however the map $\varphi:  {\mathbf{B}}^I \to  {\mathbf{B}}^{pI}$ is only an injection. Indeed, $\mathbf{B}/\varphi(\mathbf{B})$ is a degree $p$ field extension. However, one can always find explicit expressions for $\mathbf B_{\Kinfty}^I$ using Lem \ref{lem laurent series}.
 \end{rem}

\subsection{Locally analytic vectors in rings}\label{subseclav}

Recall that we use the subscript $L$ to indicate the $\gal(\overline{K}/L)$-invariants. Recall for a $\hat{G}=\gal(L/K)$-representation $W$, we denote $W^{\tau\dla, \gamma=1}:= W^{\tau\dla} \cap W^{\gamma=1}$. The following theorem in \cite{GP21} is the main result concerning calculation of locally analytic vectors in period rings.

\begin{theorem} \cite[Lem. 3.4.2, Thm. 3.4.4]{GP21}
 \label{thmLAVmain}
\begin{enumerate}
\item \label{item-varphiu}
 For $I=[r_\ell, r_k]$ or $[0, r_k]$, and for each $n \geq 0$,  $\varphi^{-n}(u) \in (\wt{\mathbf{B}}^{I}_L)^{\tau_{n+k}\dan}$ (cf. Notation \ref{notataula}). So in particular, \[u \in (\wtB_L^{[0, r_0]})^{\tau\dan}.\]
  
\item \label{itemkinf}
For $I=[r_\ell, r_k]$ or $[0, r_k]$, we have $(\wt{\mathbf{A}}^I_{ L})^{\tau\dla, \gamma=1} = \cup_{m\geq 0} \varphi^{-m}(\mathbf{A}^{p^mI}_{K_\infty}).$
\item For any $r \geq 0$,
$(\wt{\mathbf{B}}_{L}^{[r, +\infty)})^{\tau\dpa, \gamma=1} =    \cup_{m\geq 0} \varphi^{-m}(\mathbf{B}_{ K_\infty}^{[p^mr, +\infty)}).$
\end{enumerate}
\end{theorem}

\begin{rem}
Let us point out some (fortunately very minor) errors in the proof of the theorem above, all relating to the ``$\tau$-issue" in Notation \ref{nota hatG}.
\begin{enumerate}
\item Firstly, in \cite[Notation 3.2.1]{GP21}, we should always \emph{fix} the $\tau$ as we now do in Notation \ref{nota hatG}; we are implicitly using the same $\tau$ in \cite{GP21}, but only when $\Kinfty \cap \Kpinfty=K$. 

\item The problem with this $\tau$-issue in concrete computations is that we have
\begin{equation} 
\begin{cases} 
\tau(u)=u[\underline \varepsilon], &  \text{if }  \Kinfty \cap \Kpinfty=K; \\
\tau(u)=u[\underline \varepsilon]^2, & \text{if }  \Kinfty \cap \Kpinfty=K(\pi_1).
\end{cases}
\end{equation}

\item In \cite[Lem. 3.4.2]{GP21} and \cite[Thm. 3.4.4]{GP21}, if $\Kinfty \cap \Kpinfty=K(\pi_1)$, then we should change \emph{some} of the ``$a$" there to ``$2a$", in the equation above (3.4.2),  in (3.4.3), and in the equation below (3.4.8); this is effectively because we now have $\tau(u)=u(1+v)^2$. The changes to ``$2a$" only make the relevant convergence even easier, hence do not change the final results.

\end{enumerate}
\end{rem}

\begin{defn} \label{defrigring}
\begin{enumerate}
\item Define the following rings (which are LB spaces):
\[
\wt{\mathbf{B}}^{\dagger}: = \cup_{r \geq 0} \wt{\mathbf{B}}^{[r, +\infty]},
\quad  \mathbf{B}^{\dagger}: = \cup_{r \geq 0} \mathbf{B}^{[r, +\infty]},
\quad \mathbf{B}_{  K_\infty}^{\dagger}: = \cup_{r \geq 0} \mathbf{B}_{K_\infty}^{[r, +\infty]}.
\]

\item Define the following rings (which are LF spaces):
\[ \wt{\mathbf{B}}_{  \rig}^{\dagger}: = \cup_{r \geq 0} \wt{\mathbf{B}}^{[r, +\infty)},
\quad  \mathbf{B}_{  \rig}^{\dagger}: = \cup_{r \geq 0} \mathbf{B}^{[r, +\infty)},
\quad \mathbf{B}_{\rig, K_\infty}^{\dagger}: = \cup_{r \geq 0} \mathbf{B}_{K_\infty}^{[r, +\infty)}.
\]

\item Define the following notations:
\[ \wtB^+_\rig:=\wtB^{[0, +\infty)} , \quad \B^+_{\rig, \Kinfty}:=\B_{ \Kinfty}^{[0,+\infty)}. \]
\end{enumerate}
\end{defn}
Note that $\mathbf{B}_{  K_\infty}^{\dagger}$ (resp. $\mathbf{B}_{\rig, K_\infty}^{\dagger}$) is precisely the explicitly defined ring in \eqref{eqbkinfd} (resp. \eqref{eqbrig}).

\begin{cor} \label{cor rig la} 
We have
\begin{eqnarray*}
(\wt{\mathbf{B}}_{  \rig, L}^{\dagger})^{\tau\dpa, \gamma=1} &=& \cup_{m\geq 0} \varphi^{-m}({\mathbf{B}}_{  \rig, K_\infty}^{\dagger});\\
(\wt{\mathbf{B}}_{  \rig, L}^{+})^{\tau\dpa, \gamma=1} &=& \cup_{m\geq 0} \varphi^{-m}({\mathbf{B}}_{  \rig, K_\infty}^{+}).
\end{eqnarray*}
\end{cor}

\subsection{The element $\mathfrak{t}$} \label{subsecfrakt}
In this subsection, we   study  a locally analytic element $\mathfrak{t}$, which plays a useful role in the definition of our monodromy operators.

Recall we defined the element $[\underline \varepsilon] \in \wt{\mathbf{A}}^+$ in Notation \ref{114}.
Let $t=\log([\underline \varepsilon])  \in \Bcris^+$ be the usual element.
Define the element
\[
\lambda :=\prod_{n \geq 0} (\varphi^n(\frac{E(u)}{E(0)})) \in \mathbf{B}_{K_\infty}^{[0,+\infty)} \subset \Bcris^+.\]
The equation $\varphi(x)=\frac{pE(u)}{E(0)} \cdot x$ over $\wt{\mathbf{A}}^+$ has solutions in $\wt{\mathbf{A}}^+ \backslash p\wt{\mathbf{A}}^+$, which is unique up to units in $\Zp$ (cf. the paragraph above \cite[Thm. 3.2.2]{Liu07}).
By the discussion in \cite[Example 5.3.3]{Liu07}, there exists a \emph{unique} solution $\mathfrak{t} \in \wt{\mathbf{A}}^+$ such that
\begin{equation}\label{eqfrakt}
{p\lambda} \mathfrak{t} = {t},
\end{equation}
which  holds as an equation in $\Bcris^+.$
Since $\mathfrak{t} \in \wt{\mathbf{A}}^+ \subset \wt{\mathbf{B}}_L^\dagger$, and since $\wt{\mathbf{B}}_L^\dagger$ is a field (\cite[Prop. 5.12]{Col08}), there exists some $r(\mathfrak{t})>0$ such that $1/\mathfrak{t} \in \wt{\mathbf{B}}_L^{[r(\mathfrak{t}), +\infty]}$.


\begin{lemma} \label{lem b}
\cite[Lem. 5.1.1]{GP21}
We have $\mathfrak{t}, 1/\mathfrak{t} \in (\wt{\mathbf{B}}^{[r(\mathfrak{t}), +\infty)}_{ L})^{\hat{G}\dpa}$.
\end{lemma}

\subsection{Some ``log"-rings} \label{sslog}
In this subsection, we introduce some log rings (corresponding to the ``rig"-rings in Def. \ref{defrigring}). We first introduce a convention often used here.

\begin{convention}\label{convtop}
Let $A$ be a topological ring, and let $Y$ be a variable. Then we always equip $A[Y]$ with the inductive topology using $A[Y]: =\cup_{n \geq 0} (\oplus_{i=0}^n A\cdot Y^i)$ where each $A\cdot Y^i$ has the topology induced from that on $A$.
\end{convention}

Choose some $\underline{p}:=(p_0, p_1, \cdots, p_n, \cdots) \in \wt{\mathbf E}^+$ where $p_0=p$ and $p_{n+1}^p=p_n$ for all $n \geq 0$.
Let $X$ be a formal variable, and define
\[\wtb_{\log}^\dagger: = \wt{\mathbf{B}}^{ \dagger }_{\rig}[X].\]
Extend the $\varphi$-operator and $G_K$-action on $\wt{\mathbf{B}}_\rig^{ \dagger }$ to $\wtb_{\log}^\dagger$  such that $\varphi(X)=pX$ and $g(X)=X+c(g)t$ where $c(g)$ is the co-cycle such that $g(\underline p)=\underline p \cdot \underline \varepsilon^{c(g)}$; define a  $\wt{\mathbf{B}}_\rig^{ \dagger }$-derivation $N$ on $\wtb_{\log}^\dagger$ such that $N(X)=e$ (cf. Rem. \ref{remNlu} for this convention). 
Let
\[\wtb_{\log}^+: = \wt{\mathbf{B}}^+_{\rig}[X];\]
it is a subring of $\wtb_{\log}^\dagger$ which is $(\varphi, G_K, N)$-stable.


\begin{prop} \cite[Prop. 5.15]{Col08} \label{proplog}
With respect to the choice $\underline{p}$,
there exists a   $\varphi$- and $G_K$-equivariant map 
\[\log: (\wt{\mathbf{B}}^{\dagger } )^\ast \to \wt{\mathbf{B}}^{ \dagger }_{\rig}[X]\]
which is uniquely determined by the following properties:
\begin{enumerate}
\item $\log xy=\log x +\log y;$
\item $\log x =\sum_{n=1}^{+\infty} \frac{-(1-x)^n}{n}$ if the series converges;
\item $\log[a]=0$ if $a \in \overline k$;
\item $\log p=0$ and $\log[\underline p]=X.$
\end{enumerate}
For this log map, if $x =\sum_{k=k_0}^{+\infty} p^k[x_k]$ with $x_{k_0}\neq 0$, then 
\begin{equation} \label{eqnlogx}
N(\log x)=e\cdot v_{\wtE}(x_{k_0}).
\end{equation}
\end{prop}
\begin{proof}
This is exactly the same as \cite[Prop. 5.15]{Col08} except the equation in \eqref{eqnlogx}: in \cite{Col08}, it is ``$N(\log x)=-v_{\wtE}(x_{k_0}).$" Our Eqn.  \eqref{eqnlogx} matches with the choice $N(X)=e$ (see Rem. \ref{remNlu} for the reason of this choice):  in \cite{Col08}, it is ``$N(X)=-1.$".
 
Here, let us sketch the construction of this log map.
By (4), it suffices to consider $x \in (\wt{\mathbf{B}}^{\dagger } )^\ast$ such that $v_p(x)=0$. 
Then by (1), it suffices to consider the case when  $v_{\wtE}(\overline x) \in \mathbb Z^{\geq 0}$; in this case, it can be uniquely written as
\begin{equation}\label{eqxp}
x=[\underline{p}^\alpha][a]y, \quad \text {where } \alpha \in \mathbb  Z^{\geq 0}, a \in \overline k, y \in \wt{\mathbf A} \cap \wt{\B}^\dagger, \text{ such that } v_{\wt{\mathbf E}}(\overline{y}-1)>0;
\end{equation}    then we can define
 \[\log (x)=\alpha X +\log (y)\] where we have $\log (y) \in \wt{\B}^\dagger_\rig$ by  \cite[Lem. 5.14(2)]{Col08}.
\end{proof}

\begin{rem} \label{remNlu}
Note that the $N$ operator here equals to ``$-e \cdot N$" in \cite{Col08}. We make this choice so that we have $N(\log u)=1$. This choice is the same as that in \cite{Kis06}; in particular, we can remove minus signs everywhere for this $N$-operator. This is the ``good" choice for us since $\log u$  is important in the study of $(\varphi, \tau)$-modules; e.g., cf. Lem \ref{lem266} below.
See also Rem. \ref{compaminus} later where we make some choice to remove minus signs for the $N_\nabla$-operator.
\end{rem}

\begin{lemma}
\cite[Lem. 5.14(1)]{Col08} \label{lem514}
Suppose $x =\sum_{k=0}^{+\infty} p^k[x_k]$ is a unit in $\wtA^+$  such that $v_{\wt{\mathbf E}}(x_0 -1)>0$, then $\sum_{n=1}^{+\infty} \frac{-(1-x)^n}{n}$ converges inside $ \wt{\mathbf{B}}^{+}_{\rig}$.
\end{lemma}

\begin{lemma} \label{lem274}
Let  $\beta   \in \wtE^+$ such that $\alpha:=v_{\wt{\mathbf E}}(\beta) \neq 0$,  and let $x:=[\beta] \in \wtA^+$ be the Teichm\"uller lift, then we have
\begin{eqnarray*}
\wt{\mathbf{B}}^{ \dagger }_{\log}&=&\wt{\mathbf{B}}^{ \dagger }_{\rig}[\log x];  \\
\wt{\mathbf{B}}^{+}_{\log}&=&\wt{\mathbf{B}}^{ +}_{\rig}[\log x].
\end{eqnarray*}
\end{lemma}
\begin{proof}
Write $x=[\underline{p}^\alpha][a]y$ as in Eqn. \eqref{eqxp}, then $y$ satisfies the condition in Lem. \ref{lem514}, and hence $\log(x)=\alpha X +\log (y)$ with $\log y \in  \wt{\mathbf{B}}^{+}_{\rig}$.
\end{proof}

\begin{defn} \label{defellu}
Let $\ell_u:=\log(u)=\log([\upi])$, and define
\begin{eqnarray*}
 {\mathbf{B}}^{ \dagger }_{\log, \Kinfty}&:=& {\mathbf{B}}^{ \dagger }_{\rig, \Kinfty}[\ell_u] \subset \wt{\mathbf{B}}^{ \dagger }_{\log};  \\
  {\mathbf{B}}^{ + }_{\log, \Kinfty}&:=& {\mathbf{B}}^{ + }_{\rig, \Kinfty}[\ell_u] \subset \wt{\mathbf{B}}^{+}_{\log}.
\end{eqnarray*}
  (The containments follow from Lem. \ref{lem274}.)
\end{defn}
 

\begin{lemma}\label{lem266}
Let $I \subset [0, +\infty)$ be a closed interval, and let $k \geq 1$, then
\begin{equation*}
 \ell_u^k \in \left(\oplus_{i=0}^k  \wtB_L^I \cdot\ell_u^i \right)  ^{\tau\dan, \gamma=1}
\end{equation*}
where $\oplus_{i=0}^k  \wtB_L^I \cdot\ell_u^i$ is regarded as a Banach space.
Namely, $\ell_u$ is always a $\tau$-\emph{analytic} vector (not just locally analytic).
\end{lemma}
\begin{proof}
Note that 
\begin{equation}\label{eq277new}
g(\ell_u) =\ell_u + \sigma(g)t, \quad \forall g\in G_K,
\end{equation}
where $\sigma(g) \in \Zp^\times$ (is the co-cycle) such that $g([\upi])=[\upi][\underline \varepsilon]^{\sigma(g)}$; it is then easy to deduce that $\ell_u$ is a $\tau$-analytic vector (e.g., using \cite[Lem. 3.1.8]{GP21}). 
\end{proof}

\begin{cor}\label{cor277}
\begin{eqnarray*}
(\wt{\mathbf{B}}^{ \dagger }_{\log, L})^{\tau\dpa, \gamma=1} & =& \cup_{m \geq 0} \varphi^{-m}( {\mathbf{B}}^{ \dagger }_{\log, \Kinfty});\\
(\wt{\mathbf{B}}^{ + }_{\log, L})^{\tau\dpa, \gamma=1} & =& \cup_{m \geq 0} \varphi^{-m}( {\mathbf{B}}^{ +}_{\log, \Kinfty}).
\end{eqnarray*}
\end{cor}
 \begin{proof}
 It follows from Lem. \ref{lem266} and \cite[Prop. 3.1.6]{GP21} (as all $\ell_u^i$ are analytic vectors).
 \end{proof}

\section{Modules and locally analytic vectors}\label{secmod}

In this section, we recall the theory of \'etale $(\varphi, \tau)$-modules and their overconvergence property. In particular, we discuss the relation between locally analytic vectors and the overconvergence property.

In this section and in \S \ref{secsemist}, we will introduce several categories of \emph{modules with structures}. We will always omit the definition of morphisms for these categories, which are always obvious (i.e., module homomorphisms compatible with various structures). 


\subsection{\'Etale $(\varphi, \tau)$-modules}
\begin{defn} \label{defetphi}
Objects in the following are called \emph{\'etale $\varphi$-modules}.
\begin{enumerate}
\item Let $\text{Mod}_{\mathbf{A}_{K_\infty}}^\varphi$ denote the category of finite free $\mathbf{A}_{K_\infty}$-modules $M $  equipped with a $\varphi_{\mathbf{A}_{K_\infty}}$-semi-linear endomorphism $\varphi _M : M\to M$ such that $1 \otimes \varphi : \varphi ^*M \to M $ is an isomorphism.

\item  Let $\textnormal{Mod}_{\mathbf{B}_{K_\infty}}^\varphi$ denote the category of finite free $\mathbf{B}_{K_\infty}$-modules (indeed, vector spaces) $D$  equipped with a $\varphi_{\mathbf{B}_{K_\infty}}$-semi-linear endomorphism $\varphi _D : D\to D$ such that there exists a finite free $\mathbf{A}_{K_\infty}$-lattice $M$ such that $M[1/p]=D$, $\varphi_D(M)\subset M$, and $(M, \varphi_D|_M) \in  \textnormal{Mod}_{\mathbf{A}_{K_\infty}}^\varphi$.
\end{enumerate}
\end{defn}

\begin{defn} \label{defphitaumod}
Objects in the following are called \emph{\'etale $(\varphi, \tau)$-modules}.
\begin{enumerate}
\item Let $\textnormal{Mod}_{\mathbf{A}_{K_\infty}, \wt{\mathbf{A}}_L}^{\varphi, \hat{G}}$ denote the category consisting of triples $(M, \varphi_M, \hat G)$ where
\begin{itemize}
\item $(M , \varphi_M) \in \text{Mod}_{\mathbf{A}_{K_\infty}}^\varphi$;
\item $\hat G$ is a  continuous  $\varphi_{\hat M}$-commuting  $\wt{\mathbf{A}}_L$-semi-linear $\hat G$-action on $\hat M : =\wt{\mathbf{A}}_L \otimes_{\mathbf{A}_{K_\infty}} M$ (here, continuity is with respect to  topology induced by the weak topology on $\wt{\mathbf{A}}$);
\item regarding $M$ as an $\mathbf{A}_{K_\infty} $-submodule in $ \hat M $, then $M
\subset \hat M ^{\gal(L/K_\infty)}$.
\end{itemize}

\item Let $\textnormal{Mod}_{\mathbf{B}_{K_\infty}, \wt{\mathbf{B}}_L}^{\varphi, \hat{G}}$ denote the category consisting of triples $(D, \varphi_D, \hat G)$ which contains a lattice (in the obvious fashion) $(M, \varphi_M, \hat G) \in \textnormal{Mod}_{\mathbf{A}_{K_\infty}, \wt{\mathbf{A}}_L}^{\varphi, \hat{G}}$.
\end{enumerate}
\end{defn}

\subsubsection{} Let $\textnormal{Rep}_{\Qp}(G_{\infty}) $ (resp. $\textnormal{Rep}_{\Qp}(G_{K}) $ ) denote the category of finite dimensional $\Qp$-vector spaces $V$   with  continuous $\Qp$-linear $G_{\infty}$ (resp. $G_K$)-actions.
\begin{itemize}[leftmargin=0cm]
\item For $D \in \textnormal{Mod}_{\mathbf{B}_{K_\infty}}^\varphi$, let
\[ V(D):= ( \wt{\mathbf{B}} \otimes_{\mathbf{B}_{K_\infty}} D) ^{\varphi =1},\]
then $V(D) \in \textnormal{Rep}_{\Qp}(G_{\infty})$.
If furthermore $(D, \varphi_D, \hat G) \in \textnormal{Mod}_{\mathbf{B}_{K_\infty}, \wt{\mathbf{B}}_L}^{\varphi, \hat{G}}$, then $ V(D) \in \textnormal{Rep}_{\Qp}(G_{K})$.

\item For $V \in \textnormal{Rep}_{\Qp}(G_{\infty}) $, let
\[     D_{K_\infty}(V):= (\mathbf B \otimes_{\Qp} V) ^{G_\infty}, \]
then $ D_{K_\infty}(V) \in \textnormal{Mod}_{\mathbf{B}_{K_\infty}}^\varphi$.
If furthermore $V \in \textnormal{Rep}_{\Qp}(G_K) $, let
\[     \wt{D}_L(V):= ( \wt{\mathbf{B}} \otimes_{\Qp} V) ^{G_L}, \]
then $\wt{D}_L(V) = \wt{\mathbf{B}}_L \otimes_{\mathbf{B}_{K_\infty}}  D_{K_\infty}(V)$ has a $\hat{G}$-action, making $(D_{K_\infty}(V), \varphi, \hat G)$ an \'etale $(\varphi, \tau)$-module.
\end{itemize}

As we already mentioned in Rem. \ref{remindep},
Thm. \ref{thmphitau} below is the only place we \emph{use} results from \cite{Car13}.
\begin{thm}  \label{thmphitau}
\hfill
\begin{enumerate}
\item \cite[Prop. A 1.2.6]{Fon90}  The functors $V$ and $D_{K_\infty}$ induce  an exact tensor equivalence between the categories $\textnormal{Mod}_{\mathbf{B}_{K_\infty}}^\varphi$ and $\textnormal{Rep}_{\Qp}(G_{\infty}) $.
\item \cite[Thm. 1]{Car13} The functors $V$ and $(D_{K_\infty}, \wt{D}_L)$ induce  an exact tensor equivalence between the categories $\textnormal{Mod}_{\mathbf{B}_{K_\infty}, \wt{\mathbf{B}}_L}^{\varphi, \hat{G}}$ and $\textnormal{Rep}_{\Qp}(G_K) $.
\end{enumerate}
\end{thm}
\begin{proof}
Note that Item (2) was written  only for $p>2$ in \cite{Car13}.
Our Def. \ref{defphitaumod} is slightly different from Caruso's definition (although the underlying idea is the same),  cf. the discussion in \cite[Rem. 2.1.6]{GL20}. In particular, Item (2) above is valid for all $p$.
\end{proof}

\subsection{Overconvergence and locally analytic vectors}

Let  $V\in \textnormal{Rep}_{\Qp}(G_K) $.
Given $I \subset [0, +\infty]$, let
\begin{eqnarray*}
D_{K_\infty}^I(V) &:=& ( \mathbf{B}^I \otimes_{\Qp} V) ^{G_\infty}, \\
 \wt{D}_L^I(V)&:=&( \wt{\mathbf{B}}^I \otimes_{\Qp} V) ^{G_L}.
\end{eqnarray*}

\begin{defn} \label{def oc}
Let $V\in \textnormal{Rep}_{\Qp}(G_K) $, and let $\hat{D}= (D_{K_\infty}(V), \varphi, \hat G)$ be the \'etale $(\varphi, \tau)$-module associated to it.
We say that $\hat{D}$ is \emph{overconvergent} if there exists $r>0$, such that for $I'=[r, +\infty]$,
\begin{enumerate}
\item $D_{K_\infty}^{I'}(V)$ is finite free over  $\mathbf{B}^{I'}_{K_\infty}$, and
 $ \mathbf{B}_{K_\infty} \otimes_{\mathbf{B}^{I'}_{K_\infty}} D_{K_\infty}^{I'}(V) = D_{K_\infty}(V);$
\item $\wt{D}_L^{I'}(V)$ is finite free over  $\wt{\mathbf{B}}^{I'}_L$ and
 \[   \wt{\mathbf{B}}_L \otimes_{\wt{\mathbf{B}}^{I'}_L}  \wt{D}_L^{I'}(V) = \wt{D}_L(V).\]
\end{enumerate}
\end{defn}

\begin{theorem}\label{thm final} \cite{GL20, GP21}
For any $V\in \textnormal{Rep}_{\Qp}(G_K) $, its associated  \'etale $(\varphi, \tau)$-module is overconvergent.
\end{theorem}
\begin{rem}
As we already mentioned in Rem. \ref{rem2pf}, a first proof of Thm. \ref{thm final} is given in \cite{GL20} (which only works for $K/\Qp$ a finite extension), and a second proof is given in \cite{GP21}; it is the second proof that will be useful for the current paper, e.g., see Cor. \ref{corlamod} below.
\end{rem}

Let $V\in \textnormal{Rep}_{\Qp}(G_K) $ of dimension $d$,  and let  $D_{K_\infty}^{I'}(V)$ be as in Def. \ref{def oc} for $I'=[r(V), +\infty]$ (which exists by Thm. \ref{thm final}). Let
\begin{eqnarray}
D^\dagger_{ \Kinfty}(V) &:=& D_{K_\infty}^{I'}(V) \otimes_{\mathbf{B}_{K_\infty}^{I'}} \mathbf{B}^\dagger_{  \Kinfty};\\
\label{eqringmod} D^\dagger_{\rig, \Kinfty}(V) &:=& D_{K_\infty}^{I'}(V) \otimes_{\mathbf{B}_{K_\infty}^{I'}} \mathbf{B}^\dagger_{\rig, \Kinfty}.
\end{eqnarray}  
We call $D^\dagger_{ \Kinfty}(V)$ resp. $D^\dagger_{\rig, \Kinfty}(V)$ the \emph{overconvergent $(\varphi, \tau)$-module} resp.  the \emph{rigid-overconvergent $(\varphi, \tau)$-module} associated to $V$.

\begin{cor} \label{corlamod}
The subset $D^\dagger_{\rig, \Kinfty}(V)$ generates the $(\wt{\mathbf{B}}^\dagger_{\rig, L})^{\tau\dpa, \gamma=1}$-module $((\wt{\mathbf{B}}^\dagger_{\rig } \otimes_{\Qp} V  )^{G_L})^{\tau\dpa, \gamma=1}$. Indeed,
\[((\wt{\mathbf{B}}^\dagger_{\rig } \otimes_{\Qp} V  )^{G_L})^{\tau\dpa, \gamma=1}
=D^\dagger_{\rig, \Kinfty}(V)\otimes_{\mathbf{B}^\dagger_{\rig, \Kinfty}} (\wt{\mathbf{B}}^\dagger_{\rig, L})^{\tau\dpa, \gamma=1}.\]
\end{cor}
\begin{proof}
This is extracted from the proof of Thm. \ref{thm final} in \cite[Thm. 6.2.6]{GP21}; indeed, it easily follows from \cite[Eqn. (6.2.5)]{GP21}.
\end{proof}

\subsection{Modules with respect to  field extensions} \label{subsecmodE}
(The discussion here is  continuation of \S \ref{22}.)
Let $V\in \textnormal{Rep}_{\Qp}(G_K) $. 
 Let $E$ be  as in Notation \ref{notaEK}. Then with respect to the $G_E$-representation $V|_{G_E}$, we can also constrcut the corresponding ``$(\varphi, \tau)$-module"  and its overconvergent version; denote them as 
\[ D_{E_\infty}(V|_{G_E}), \quad D_{E_\infty}^\dagger(V|_{G_E}),  \quad D_{\rig, E_\infty}^\dagger(V|_{G_E}).\]
These are finite free modules over the rings $\B_{E_\infty}, \B^\dagger_{E_\infty}, \B^\dagger_{\rig, E_\infty} $ respectively, which are constructed analogously as in Notation \ref{notaEK}.

\begin{lemma}\label{lemidEK}
We have   $\varphi$-equivariant isomorphisms
\begin{eqnarray}
D_{K_\infty}(V)\otimes_{\B_{\Kinfty}}  \B_{E_\infty} &\simeq &  D_{E_\infty}(V|_{G_E}); \\
D_{K_\infty}^\dagger(V)\otimes_{\B^\dagger_{\Kinfty}}  \B^\dagger_{E_\infty} &\simeq &   D^\dagger_{E_\infty}(V|_{G_E}); \\
D^\dagger_{\rig, K_\infty}(V)\otimes_{\B^\dagger_{\rig,\Kinfty}}  \B^\dagger_{\rig, E_\infty} &\simeq &   D^\dagger_{\rig, E_\infty}(V|_{G_E}). 
\end{eqnarray}
\end{lemma}


\begin{proof}
The first two isomorphisms are obvious since both $\B_{E_\infty}$ and $ \B^\dagger_{E_\infty}$ are fields (cf. \cite[Prop. II.1.6(1)]{CC98} in the $(\varphi, \Gamma)$-module setting); the third isomorphism then follows. 
\end{proof}


\section{Monodromy operator for $(\varphi, \tau)$-modules}\label{secmono}
In this section, we   define a natural monodromy operator on the rigid-overconvergent $(\varphi, \tau)$-modules.

  Let $\log(\tau^{p^n})$ denote  the (formally written) series $(-1)\cdot \sum_{k \geq 1} (1-\tau^{p^n})^k/k$, then $\nabla_\tau :=\frac{\log(\tau^{p^n})}{p^n}$ for $n \gg 0$ is a well-defined Lie-algebra operator acting on $\hat{G}$-locally analytic representations.

\subsection{Monodromy operator over rings}\label{subsecnring} 
Recall that by Cor. \ref{cor277},
\[(\wt{\mathbf{B}}_{  \log, L}^{\dagger})^{\tau\dpa, \gamma=1} = \cup_{m\geq 0} \varphi^{-m}({\mathbf{B}}_{  \log, K_\infty}^{\dagger}).\]
Hence $\nabla_\tau$ induces a map:
\[\nabla_\tau: {\mathbf{B}}_{  \log, K_\infty}^{\dagger} \to  (\wt{\mathbf{B}}_{  \log, L}^{\dagger})^{\hat{G}\dpa}.\]
Recall by Lem. \ref{lem b}, we have $1/\mathfrak{t} \in (\wt{\mathbf{B}}^{[r(\mathfrak{t}), +\infty)}_{ L})^{\hat{G}\dpa}$.
We can define the following operator:
\begin{equation} \label{eqnbladef}
N_\nabla: {\mathbf{B}}_{  \log, K_\infty}^{\dagger} \to   (\wt{\mathbf{B}}_{  \log, L}^{\dagger})^{\hat{G}\dpa},
\end{equation}
by setting
\begin{equation}\label{eqnnring}
{N_\nabla:=}
\begin{cases} 
\frac{1}{p\mathfrak{t}}\cdot \nabla_\tau, &  \text{if }  \Kinfty \cap \Kpinfty=K; \\
& \\
\frac{1}{p^2\mathfrak{t}}\cdot \nabla_\tau=\frac{1}{4\mathfrak{t}}\cdot \nabla_\tau, & \text{if }  \Kinfty \cap \Kpinfty=K(\pi_1), \text{ cf. Notation \ref{nota hatG}. }
\end{cases}
\end{equation}

  
\begin{rem}\label{remcompaKis}
The   $p$ (resp. $p^2$) in the denominator of \eqref{eqnnring} makes  our monodromy operator compatible with earlier theory of Kisin in \cite{Kis06}, but \emph{up to a minus sign}, see Rem. \ref{compaminus} below.
\end{rem}

\begin{lemma} \label{lemnnring}
The image of $N_\nabla$ in \eqref{eqnbladef} falls inside ${\mathbf{B}}_{  \log, K_\infty}^{\dagger}$, and hence  induces
\begin{equation}\label{eqringdes}
N_\nabla: {\mathbf{B}}_{  \log, K_\infty}^{\dagger} \to   {\mathbf{B}}_{  \log, K_\infty}^{\dagger}.
\end{equation}
Explicitly, the differential map $N_\nabla$ sends 
  $x \in {\mathbf{B}}_{  \rig, K_\infty}^{\dagger}$ to $\lambda\cdot u\frac{d}{du}(x)$, and $N_\nabla(\ell_u)= \lambda$. Furthermore, the rings ${\mathbf{B}}_{  \log, K_\infty}^{+}$, ${\mathbf{B}}_{  \rig, K_\infty}^{\dagger}$, and ${\mathbf{B}}_{ K_\infty}^{I}$ (for any $I \subset [0, +\infty)$) are all stable under this $N_\nabla$ map.
\end{lemma}
\begin{proof}
Everything follows from easy explicit calculations.
For example, when $\Kinfty \cap \Kpinfty=K$, then  $\tau(u)=u[\underline{\varepsilon}]$, hence we have (using any $n \gg 0$ in the following)
\[N_\nabla(u) =\frac{1}{p\mathfrak{t}} \cdot \frac{-1}{p^n} \cdot \sum_{k \geq 1} \frac{u(1- [\underline{\varepsilon}]^{p^n})^k}{k} = \frac{1}{p\mathfrak{t}} \cdot \frac{1}{p^n} \cdot u \cdot (p^n t) =\frac{ut}{p\mathfrak{t}}=\lambda \cdot u.\]
 The fact that $N_\nabla(\ell_u)= \lambda$ follows  from similar computation using \eqref{eq277new}.
\end{proof}


 \begin{rem}\label{compaminus}  Our $N_\nabla$  equals to ``$-N_\nabla$" in \cite[\S 1.1.1]{Kis06}. 
 Certainly, this sign change makes no difference for the results in \cite{Kis06} (which we will use later). (Alternatively, we could have added minus signs in Eqn. \eqref{eqnnring} so that everything is strictly compatible with the conventions in \cite{Kis06}; but we prefer to remove the minus signs everywhere, cf. also the choice we made in Rem. \ref{remNlu}).\end{rem}

\subsection{Monodromy operator over modules}

By Cor. \ref{corlamod}, we have
\begin{equation}
((\wt{\mathbf{B}}^\dagger_{\rig } \otimes_{\Qp} V  )^{G_L})^{\hat{G}\dpa}
=D^\dagger_{\rig, \Kinfty}(V)\otimes_{\mathbf{B}^\dagger_{\rig, \Kinfty}} (\wt{\mathbf{B}}^\dagger_{\rig, L})^{\hat{G}\dpa}.
\end{equation}
Hence similarly as in   \S \ref{subsecnring}, we have a map (using exactly the same formulae as in \eqref{eqnnring}):
\begin{equation}\label{eqnnbig}
N_\nabla:    D^\dagger_{\rig, \Kinfty}(V) \to D^\dagger_{\rig, \Kinfty}(V)\otimes_{\mathbf{B}^\dagger_{\rig, \Kinfty}} (\wt{\mathbf{B}}^\dagger_{\rig, L})^{\hat{G}\dpa}.
\end{equation}

\begin{theorem} \label{thmnnabla}
The map \eqref{eqnnbig} induces a map
\begin{equation}\label{eqnnrig}
N_\nabla: D^\dagger_{\rig, \Kinfty}(V) \to D^\dagger_{\rig, \Kinfty}(V).
\end{equation}
Namely, $N_\nabla$ is a well-defined operator on $D^\dagger_{\rig, \Kinfty}(V)$.
Furthermore, there exists some $r' \geq r(V)$ (cf. the notation above \eqref{eqringmod})  such that if $I  \subset [r', +\infty)$ , then
\begin{equation}\label{eqnnI}
N_\nabla: D^I_{\Kinfty}(V) \to D^I_{ \Kinfty}(V),
\end{equation}
where recall $D^I_{\Kinfty}(V) =
 D_{K_\infty}^{[r(V), +\infty]}(V) \otimes_{\mathbf{B}_{K_\infty}^{[r(V), +\infty]}} \mathbf{B}^I_{\Kinfty}$.
\end{theorem}
\begin{proof}

Note that we have the relation
\begin{equation*}
g \tau g^{-1}=\tau^{\chi_p(g)}, \text{ for } g\in \gal(L/\Kinfty),
\end{equation*}
where $\chi_p$ is the cyclotomic character. Note also $g(\mathfrak{t})=\chi_p(g)\mathfrak{t}$ for $g\in \gal(L/\Kinfty)$. Hence we have the relation
\begin{equation}\label{eq426}
gN_\nabla=N_\nabla g \text{ for } g\in \gal(L/\Kinfty).
\end{equation}
Since $\gal(L/\Kinfty)$ acts trivially on $D^\dagger_{\rig, \Kinfty}(V)$, hence  $\gal(L/\Kinfty)$ also acts trivially on $N_\nabla(D^\dagger_{\rig, \Kinfty}(V))$ using \eqref{eq426}. Thus, we have
\begin{eqnarray*}
N_\nabla(D^\dagger_{\rig, \Kinfty}(V)) &\subset & \left(D^\dagger_{\rig, \Kinfty}(V)\otimes_{B^\dagger_{\rig, \Kinfty}} (\wt{\mathbf{B}}^\dagger_{\rig, L})^{\hat{G}\dpa}\right)^{\gamma=1}\\
&=& D^\dagger_{\rig, \Kinfty}(V)\otimes_{\mathbf{B}^\dagger_{\rig, \Kinfty}} (\wt{\mathbf{B}}^\dagger_{\rig, L})^{\gamma=1, \tau\dpa}\\
&=& D^\dagger_{\rig, \Kinfty}(V)\otimes_{\mathbf{B}^\dagger_{\rig, \Kinfty}} (\cup_{m\geq 0}\varphi^{-m}(\mathbf{B}^\dagger_{\rig, \Kinfty})), \text{ by Cor. \ref{cor rig la}.}
\end{eqnarray*}
Choose a basis $\vec{e}$ of $D^\dagger_{\rig, \Kinfty}(V)$, then \[N_\nabla(\vec{e}) \subset D^\dagger_{\rig, \Kinfty}(V)\otimes_{\mathbf{B}^\dagger_{\rig, \Kinfty}} \varphi^{-m}(\mathbf{B}^\dagger_{\rig, \Kinfty}) \text{ for some } m \gg 0.\]

Recall $\varphi(\mathfrak{t})=\frac{pE(u)}{E(0)}\mathfrak{t}$.
Note that $\varphi \tau=\tau\varphi$ over $D^\dagger_{\rig, \Kinfty}(V)\otimes_{B^\dagger_{\rig, \Kinfty}} (\wt{\mathbf{B}}^\dagger_{\rig, L})^{\hat{G}\dpa}$,
hence we have the relation
\begin{equation}\label{eqrelnphi}
N_\nabla \varphi =\frac{pE(u)}{E(0)}\varphi N_\nabla.
\end{equation}
Using \eqref{eqrelnphi}, we can easily deduce that
\[N_\nabla(\varphi^m(\vec{e})) \subset D^\dagger_{\rig, \Kinfty}(V). \] 
Note that $\varphi^m(\vec{e})$ is also a basis of $D^\dagger_{\rig, \Kinfty}(V)$, hence we can conclude the proof of \eqref{eqnnrig}.
Finally, we can use $r'=p^mr(V)$ to derive \eqref{eqnnI}.
\end{proof}

\section{Monodromy operator for semi-stable representations}\label{secsemist}
In this section, we will recall Kisin's definition of a monodromy operator (up to a minus sign, by Rem. \ref{compaminus}) on certain modules associated to semi-stable  representations.
The main aim of this section is to show that Kisin's monodromy operator \emph{coincides} with the one   defined by us in \S \ref{secmono} (in the semi-stable case).

In \S \ref{subsecChe}, we review a result of Cherbonnier on maximal overconvergent submodules and use it to study finite height modules. 
In \S \ref{subsecKis}, we review Kisin's construction of $\mathcal O$-modules starting from   Fontaine's filtered $(\varphi,N)$-modules.
In \S \ref{subseccoin}, we prove the main result of this section, namely the coincidence of monodromy operators.

We first recall some ring notations commonly used in    \cite{Che} and \cite{Kis06}. These notations, unlike the systematic ``$\mathbf{A}, \B$-" notations in \S \ref{secring}, are \emph{ad hoc}; but they are convenient for writing.

\begin{notation}\label{501}
\hfill
\begin{enumerate}
\item Let $\gs:= \mathbf{A}^+_{K_\infty}$.
\item Let
${\mathcal{O}_{\mathcal{E}}}: =\mathbf{A}_{\Kinfty},
\quad \oedagger: =\mathbf{A}_{\Kinfty}^\dagger :=\mathbf{A}_{\Kinfty} \cap \mathbf{B}_{\Kinfty}^\dagger.$
 \item Let $\mathcal{O}:=\mathbf{B}_{\Kinfty}^{[0, +\infty)}$ (denoted as $\B^+_{\rig, \Kinfty}$ in Def. \ref{defrigring}; also denoted as $\mathcal{O}^{[0, 1)}$ in \cite{Kis06}). Explicitly,
\[
\mathcal{O} = \{f(u)= \sum_{i=0}^{+\infty} a_i u^i, a_i \in K_0 \mid   f(u) \text{ converges }, \forall u \in \mathfrak{m}_{\mathcal{O}_{\overline{K}}}   \},
\]
where $\mathfrak{m}_{\mathcal{O}_{\overline{K}}}  $ is the maximal ideal in $\mathcal{O}_{\overline{K}}$;
i.e., it consists of series that converge on the entire open unit disk. 
Recall that $N_\nabla$ is the operator $u\lambda\frac{d}{du}$ on $\mathcal{O}$.

\item Let $\mathcal{R}$ be the Robba ring as in \cite[\S 1.3]{Kis06}, which is precisely $\mathbf{B}_{\rig,  K_\infty}^{\dagger} $ in our Def. \ref{defrigring}; i.e., it consists of series that converge near the boundary of the open unit disk. 
Note that
\[\mathcal{O}= \mathbf{B}_{ K_\infty}^{[0, +\infty)} \subset \mathbf{B}_{\rig,  K_\infty}^{\dagger}=\mathcal{R}.\]
\end{enumerate}
\end{notation}

\begin{convention}
From now on, we focus on modules of \emph{non-negative}  heights, and  Galois representations of \emph{non-negative} (Hodge-Tate) weights (as we use co-variant functors throughout this paper). Hence, although the notations such as $\Mod_{\gs}^{\varphi, \geq 0}$ and $\mathrm{MF}^{\varphi, N, \geq 0}_{K_0}$ etc. in the following might be more rigorous, we use $\Mod_{\gs}^{\varphi}$ and $\MF$ etc. for short.
\end{convention}





\subsection{Finite height modules and overconvergent modules}\label{subsecChe}
In this subsection, we review a result of Cherbonnier, which says that a  finite free \'etale $\varphi$-modules (cf. Def. \ref{defetphi}) always contains a  maximal finite free ``overconvergent submodule" (possibly of a smaller rank). If the \'etale $\varphi$-module is of finite height, we show that the $\gs$-module inside it generates the maximal overconvergent submodule.

Recall $\Mod_{{\mathcal{O}_{\mathcal{E}}}}^\varphi$ is the category of finite free \'etale $\varphi$-modules (cf. Def. \ref{defetphi}). Define $\Mod_{\oedagger}^\varphi$  analogously; indeed, it consists of finite free  $\oedagger$-modules $M$ equipped with $\varphi_{\oedagger}$-semi-linear $\varphi: M \to M$ such that $1\otimes \varphi: \varphi^\ast M\to M$ is an isomorphism.


\begin{defn}\label{defleftadj}
\begin{enumerate}
\item Let \[j^{\dagger \ast}:  \Mod_{\oedagger}^\varphi \to  \Mod_{{\mathcal{O}_{\mathcal{E}}}}^\varphi\] denote the functor where $\mathcal N \mapsto \mathcal N\otimes_{\oedagger} {\mathcal{O}_{\mathcal{E}}}$.

\item Let $\mathcal{M} \in \Mod_{{\mathcal{O}_{\mathcal{E}}}}^\varphi$. Let $ F_{\dagger}(\mathcal{M})$ be be the set consisting of $\oedagger$-submodules $\mathcal N \subset \mathcal{M}$ of finite type such that $\varphi(\mathcal N) \subset \mathcal{N}.$
Let $j^{\dagger}_{\ast}(\mathcal{M}) $ be the union of elements in $ F_{\dagger}(\mathcal{M})$.
\end{enumerate}
\end{defn}

\begin{prop} \label{propmaxoc}
Let $\mathcal{M}\in \Mod_{{\mathcal{O}_{\mathcal{E}}}}^\varphi$.
\begin{enumerate}
\item  We have $j^{\dagger}_{\ast}(\mathcal{M}) \in \Mod_{\oedagger}^{\varphi}$, and hence $j^{\dagger}_{\ast}$ defines a functor
\[j^{\dagger}_{\ast}:  \Mod_{{\mathcal{O}_{\mathcal{E}}}}^\varphi \to  \Mod_{\oedagger}^\varphi. \]
Furthermore, we have
\[\mathrm{rk}_{\oedagger} j^{\dagger}_{\ast}(\mathcal{M})    \leq \mathrm{rk}_{{\mathcal{O}_{\mathcal{E}}}} \mathcal{M}.  \]
In addition, we have $ j^{\dagger}_{\ast} \circ j^{\dagger \ast} \simeq \mathrm{id}$.

\item The functor $j^{\dagger}_{\ast}$ is  a right adjoint of $j^{\dagger \ast}$, namely, if $\mathcal N_1 \in \Mod_{\oedagger}^\varphi$, $\mathcal N_2 \in \Mod_{{\mathcal{O}_{\mathcal{E}}}}^\varphi$, then
\[\Hom(\mathcal N_1, j^{\dagger}_{\ast}(\mathcal N_2)) = \Hom(j^{\dagger \ast}(\mathcal N_1), \mathcal N_2),  \]
where $\Hom$ denotes the set of morphisms in each category.
\end{enumerate}
\end{prop}
\begin{proof}
All are contained in \cite[\S 3.2, Prop. 2]{Che} except the claim that $ j^{\dagger}_{\ast} \circ j^{\dagger \ast} \simeq \mathrm{id}$: the claim follows from the fact that $\oedagger \to \OE$ is faithfully flat (as noted in beginning of \cite[\S 3.2]{Che}).
Let us mention that the ``$\varphi$-operator" in \cite[\S 3.2, Prop. 2]{Che} can be \emph{any lift} of the Frobenius $x\mapsto x^p$ on the ring $\mathcal{O}_{\mathcal{E}}/p\mathcal{O}_{\mathcal{E}}$ (cf. the beginning of \cite[\S 2.2]{Che}); hence \cite[\S 3.2, Prop. 2]{Che} applies (to the different $\varphi$-actions) in both the $(\varphi, \Gamma)$-module setting and $(\varphi, \tau)$-module setting.
\end{proof}

\begin{defn}\label{d513}
Let $\text{Mod}_{\gs}^{\varphi}$ be the category consisting of $(\gm, \varphi)$ where $\gm$ is a finite free $\gs$-module, and $\varphi: \gm \to \gm$ is a $\varphi_\gs$-semi-linear map such that the $\gs$-linear span of $\varphi(\gm)$ contains $E(u)^h\gm$ for some $h\geq 0$. We say that $\gm$ is of $E(u)$-height $\leq h$. When $\gm$ is of rank $d$, then \[T_{\gs}(\gm):=(\gm \otimes_\gs \wtA)^{\varphi=1}\]
is a finite free $\Zp$-representation of $G_\infty$ of rank $d$.
\end{defn}

\begin{defn} \label{deffinht}
Let  $T \in \textnormal{Rep}_{\Zp}(G_{\infty})$. We say $T$ is \emph{of finite $E(u)$-height with respect to $\vec{\pi}=\{\pi_n\}_{n \geq 0}$} if there exists some (hence by \cite[Prop. 2.1.12]{Kis06}, unique up to isomorphism) $\gm \in \text{Mod}_{\gs}^{\varphi}$ such that $T_\gs (\gm) \simeq T$.
We say that $V \in \textnormal{Rep}_{\Qp}(G_{\infty})$ is of finite $E(u)$-height with respect to $\vec{\pi}$ if there exist some $G_{\infty}$-stable $\Zp$-lattice $T$ (equivalently, any  $G_{\infty}$-stable lattice by  \cite[Lem. 2.1.15]{Kis06})  which is so. We say that $V \in \textnormal{Rep}_{\Qp}(G_{K})$ is of finite $E(u)$-height with respect to $\vec{\pi}$ if $V|_{G_{\infty}}$ is so.
Throughout the paper, when $E(u)$ and  $\vec{\pi}$ are unambiguous, we just say  \emph{of finite height} for short.
\end{defn}

\begin{rem}\label{remfinht}
Let  $T \in \textnormal{Rep}_{\Zp}(G_{\infty})$, and let $\cm \in \Mod_{{\mathcal{O}_{\mathcal{E}}}}^\varphi$ be its associated \'etale $\varphi$-module. Then $T$ is of finite height if and only if there is a  $\gm \in \text{Mod}_{\gs}^{\varphi}$  such that there is a $\varphi$-equivariant isomorphism $\mathbf{A}_{K_\infty} \otimes_{\gs}\gm \simeq \cm$.
\end{rem}


\begin{prop}\label{propfinoc}
Suppose $T \in \Rep_{\Zp}(G_\infty)$ is of finite height, and let $\cm, \gm$ be as in Rem. \ref{remfinht}. Then
$\mathfrak M \otimes_{\gs}\mathcal{O}_{\mathcal E}^\dagger \simeq j^\dagger_{\ast}(\bigM).$
\end{prop}
\begin{proof}
Clearly $\gm \subset j^\dagger_{\ast}(\bigM)$ by Def. \ref{defleftadj}.
By Prop. \ref{propmaxoc}(2),
\begin{equation}\label{eqhomhom}
\Hom(\gm \otimes_{\gs}\mathcal{O}_{\mathcal E}^\dagger, j^\dagger_{\ast}(\bigM)) =\Hom(\gm \otimes_{\gs}\mathcal{O}_{\mathcal E}, \bigM).
\end{equation}
Using the fact that
\begin{equation}\label{eqcapoe}
\mathcal{O}_{\mathcal{E}}^\dagger \cap (\mathcal{O}_{\mathcal{E}})^\times =(\mathcal{O}_{\mathcal{E}}^\dagger)^\times,
\end{equation}
one can check that the morphism on the left hand side of \eqref{eqhomhom} corresponding to the isomorphism on the right hand side has to be an isomorphism itself.
(See e.g. the paragraph above  \cite[\S 3.1, Def. 5]{Che} for a proof of \eqref{eqcapoe}).
\end{proof}

\subsection{Filtered $(\varphi, N)$-modules and Kisin's  $\mathcal{O}$-modules} \label{subsecKis}

\begin{defn}\label{deffilnmod}
Let $\MF$ be the category of
 \emph{(effective) filtered $(\varphi,N)$-modules over $K_0$} which consists of  finite dimensional $K_0$-vector spaces $D$ equipped with
\begin{enumerate}
\item a Frobenius $\varphi: D \to D$ such that $\varphi(ax)=\varphi(a)\varphi(x)$ for all $a \in K_0, x \in D$;
\item a monodromy $N: D \to D$, which is a $K_{0}$-linear map such that $N\varphi=p\varphi N$;
\item a filtration $(\Fil^{i}D_{K})_{i\in\mathbb{Z}}$ on $D_{K}=D\otimes_{K_0} K$, by decreasing $K$-vector subspaces such that $\Fil^{0}D_{K}= D_K$  and $\Fil^{i}D_{K}=0$ for $i \gg 0$.
\end{enumerate}

Let $\MFwa$ denote the usual subcategory of $\MF$ consisting of \emph{weakly admissible} objects.
\end{defn}

\begin{defn}
 Let $\textnormal{Mod}_{ \mathcal{O}}^{\varphi, N_\nabla }$ be the category   consisting of finite free $\mathcal{O}$-modules $M$ equipped with
\begin{enumerate}
\item a $\varphi_{\mathcal O}$-semi-linear morphism $\varphi : M \to M$ such that   the cokernel of $1 \otimes \varphi : \varphi ^*M \to M $ is killed by $E(u)^h$ for some $h \in \mathbb{Z}^{\geq 0}$;
\item $N_\nabla: M \to M$ is a  map such that $N_\nabla(fm)=N_\nabla(f)m+fN_\nabla(m)$ for all $f\in \mathcal O$ and $m \in M$, and $N_\nabla\varphi=\frac{pE(u)}{E(0)} \varphi N_\nabla$.
\end{enumerate}
\end{defn}

\subsubsection{} \label{subsubMD}
For $D\in \MF$, we can associate an object $M \in \textnormal{Mod}_{ \mathcal{O}}^{\varphi, N_\nabla}$ by \cite{Kis06}. The construction is rather complicated, and we only give a very brief sketch. (We want  to give a sketch here, since we care about the construction of the $N_\nabla$-operator in this section).

For $n \geq 0$, let $K_{n+1}=K(\pi_{n})$ (hence $K_1=K$), and let $\wh{\gs}_n$ be the completion of $K_{n+1} \otimes_{W(k)} \gs$ at the maximal ideal $(u-\pi_n)$; $\wh{\gs}_n$  is equipped with its $(u-\pi_n)$-adic filtration, which extends to a filtration on the quotient field  $\wh{\gs}_n[1/(u-\pi_n)]$.

There is a natural $K_0$-linear map $\mathcal O \to \wh{\gs}_n$ simply by sending $u$ to $u$. Recall that we have maps $\varphi: \gs \to \gs$ and $\varphi: \mathcal O \to \mathcal O$ which extends the absolute Frobenius on $W(k)$ and sends $u$ to $u^p$. Denote $\varphi_W: \gs \to \gs$ and $\varphi_W: \mathcal O \to \mathcal O$ which only acts as absolute Frobenius on $W(k)$  and sends $u$ to $u$.

Let $\ell_u=\log u$ as in Def. \ref{defellu}. We can extend the map  $\mathcal O \to \wh{\gs}_n$ to  $\mathcal{O}[\ell_u] \to \wh{\gs}_n$ which sends $\ell_u$ to
\[ \sum_{i=1}^\infty (-1)^{i-1}i^{-1}(\frac{u-\pi_n}{\pi_n})^i \in   \wh{\gs}_n.\]
Note that $\mathcal{O}[\ell_u]$ is precisely the $\B_{\log, \Kinfty}^+$ in Def. \ref{defellu}; we use the explicit notation $\mathcal{O}[\ell_u]$ for brevity and for easier comparison with Kisin's exposition.
By the constructions in \S \ref{sslog} and Lem. \ref{lemnnring}, we can  naturally extend $\varphi$ to  $\mathcal{O}[\ell_u]$ by setting $\varphi(\ell_u)=p\ell_u$, and extend $N_\nabla$ to  $\mathcal{O}[\ell_u]$ by setting $N_\nabla(\ell_u)= \lambda$ (which, recall by Rem. \ref{compaminus}, differs from Kisin's convention by a minus sign).   
 Finally, write $N$ for the derivation on $\mathcal{O}[\ell_u]$ which acts as $\mathcal O$-derivation with respect to the formal variable $\ell_u$, i.e., $N(\ell_u)=1$, cf. Rem. \ref{remNlu} for this convention.

 
Given $D \in \MF$, write $\iota_n$ for the following composite map:
\begin{equation}
\mathcal{O}[\ell_u] \otimes_{K_0} D \xrightarrow{\varphi_W^{-n}\otimes \varphi^{-n}} \mathcal{O}[\ell_u]\otimes_{K_0} D \rightarrow \wh{\gs}_n\otimes_{K_0} D = \wh{\gs}_n\otimes_{K} D_K
\end{equation}
where the second map is induced from the map  $\mathcal{O}[\ell_u] \to \wh{\gs}_n$. The composite map extends to
\begin{equation}
\iota_n: \mathcal{O}[\ell_u, 1/\lambda] \otimes_{K_0} D \rightarrow \wh{\gs}_n[1/(u-\pi_n)]\otimes_{K} D_K
\end{equation}
Now, set
\begin{equation}\label{eqdefmd}
M(D): = \{ x\in (\mathcal{O}[\ell_u, 1/\lambda] \otimes_{K_0} D )^{N=0}: \iota_n(x)\in \Fil^0  \left( \wh{\gs}_n[1/(u-\pi_n)]\otimes_{K} D_K\right), \forall n\geq 1  \},
\end{equation}
where the $\Fil^0$ in \eqref{eqdefmd} comes from tensor product of two filtrations.
Then Kisin shows that $M(D)$ is in fact a finite free $\mathcal O$-module. 
The map $\varphi \otimes \varphi$ on $\mathcal{O}[\ell_u, 1/\lambda] \otimes_{K_0} D$ induces a map $\varphi$ on  $M(D)$; 
the map  $N_\nabla\otimes 1$ on $\mathcal{O}[\ell_u, 1/\lambda] \otimes_{K_0} D$ induces a map  $N_\nabla$ on $M(D)$.
 Kisin shows that this makes $M(D)$ into an object in  $\textnormal{Mod}_{ \mathcal{O}}^{\varphi, N_\nabla}$.


Conversely, let $M \in \textnormal{Mod}_{ \mathcal{O}}^{\varphi, N_\nabla}$, then one can define $D(M):=M/uM$ with the induced $\varphi, N$-structures (where $N:= N_\nabla/u N_\nabla$); using a certain \emph{unique} $\varphi$-equivariant \emph{section}   $\xi: D(M) \to M$ as in \cite[Lem. 1.2.6]{Kis06}, one can also define a filtration on $D(M)\otimes_{K_0} K$. This gives rise to an object in $\MF$.

\begin{thm}\label{thmcomdiaa}  \cite[Thm. 1.2.15]{Kis06}
The constructions in \S \ref{subsubMD} induce  an  equivalence  between  $\MF$ and $
\textnormal{Mod}_{ \mathcal{O}}^{\varphi, N_\nabla}$.
\end{thm}
 
Let $\textnormal{Mod}_{ \mathcal{O}}^{\varphi, N_\nabla, 0}$ be the subcategory of $\textnormal{Mod}_{ \mathcal{O}}^{\varphi, N_\nabla}$ consisting of objects $M$ such that $\mathcal{R}\otimes_{\mathcal O}M$ is \emph{pure of slope $0$} in the sense of Kedlaya (cf. \cite{Ked04, Ked05} or \cite[\S 1.3]{Kis06}).

\begin{thm} \cite[Thm. 1.3.8]{Kis06}
The equivalence in Thm. \ref{thmcomdiaa}  induces  an equivalence between 
$\MFwa$ and $
\textnormal{Mod}_{ \mathcal{O}}^{\varphi, N_\nabla, 0}.$
\end{thm}

Let $\Mod_{\gs}^{\varphi, N}$ be the category where an object is a $\gm \in \Mod_{\gs}^{\varphi}$ together with a $K_0$-linear map $N: \gm/u\gm[1/p] \to \gm/u\gm[1/p]$ such that $N\varphi=p\varphi N$ over $\gm/u\gm[1/p]$. Let $\Mod_{\gs}^{\varphi, N}\otimes \Qp$ be   its isogeny category.

\begin{theorem}\label{thm539}
There exists a fully faithful $\otimes$-functor from $\MFwa$ to $\Mod_{\gs}^{\varphi, N}\otimes \Qp$. Furthermore, suppose $D \in \MFwa$ maps to $(\gm, \varphi, N)$, then
\begin{enumerate}
\item   there is a $\varphi$-equivariant isomorphism 
\begin{equation}\label{537}
\gm \otimes_\gs \mathcal{O} \simeq M(D);
\end{equation}
 
\item there is a canonical $G_\infty$-equivariant isomorphism 
\begin{equation}\label{538}
T_\gs(\gm)\otimes_{\Zp}\Qp \simeq V_\st(D)|_{G_\infty},
\end{equation}
where $V_\st(D)$ is the usual \emph{co-variant} Fontaine functor.
\end{enumerate} 
\end{theorem}
\begin{proof}
Item (1) follows from \cite[Lem. 1.3.13, Cor. 1.3.15]{Kis06}, and Item (2)   is the co-variant version of   \cite[Prop. 2.1.5]{Kis06}.
\end{proof}

\subsection{Coincidence of monodromy operators} \label{subseccoin}
In this subsection, we show that the monodromy operators in Kisin's construction  and in our construction  \emph{coincide} in the case of semi-stable representations.  

\begin{lemma} \label{lem541}
Let $D\in \MF$.
There is a  $(\varphi, N)$-equivariant isomorphism
\begin{equation}\label{541}
M(D)\otimes_{\mathcal{O}} \mathcal{O}[\ell_u, 1/\lambda] \simeq D\otimes_{K_0} \mathcal{O}[\ell_u, 1/\lambda].
\end{equation}
\end{lemma}
\begin{proof}
Let $\mathcal{D}_0:=(\mathcal{O}[\ell_u] \otimes_{K_0} D )^{N=0}$ (also considered in the proof of \cite[Lem. 1.2.2]{Kis06}). Solving this differential equation  using the fact that $N_D$ is nilpotent (namely, after choosing a basis of $D$, we get an easy differential equation), we have that $\mathcal{D}_0$ is a finite free $\mathcal{O}$-module of rank $d$ and
\begin{equation}
\mathcal{D}_0 \otimes_{\mathcal{O}} \mathcal{O}[\ell_u] \simeq D \otimes_{K_0}   \mathcal{O}[\ell_u].
\end{equation}
Furthermore, by the construction in Eqn. \eqref{eqdefmd}, we have
\begin{equation}
M(D) \otimes_{\mathcal{O}} \mathcal{O}[1/\lambda] \simeq\mathcal{D}_0 \otimes_{\mathcal{O}} \mathcal{O}[1/\lambda].
\end{equation}
Hence both sides of \eqref{541} are isomorphic to $\mathcal{D}_0 \otimes_{\mathcal{O}} \mathcal{O}[\ell_u, 1/\lambda]$.
\end{proof}

We record some other comparisons between $M(D)$ and $D$  that will be useful later.
\begin{cor}\label{cor542} Let $D\in \MF$. 
We have  $(\varphi, N)$-equivariant isomorphisms
\begin{eqnarray}
\label{new544}
M(D) \otimes_{\mathcal{O}} \wtb_{\log}^+[1/t] &\simeq & D\otimes_{K_0}  \wtb_{\log}^+[1/t];\\
\label{new544add}
M(D) \otimes_{\mathcal{O}} \wtb_{\log}^\dagger[1/t] &\simeq & D\otimes_{K_0}  \wtb_{\log}^\dagger[1/t];\\
\label{new545}
M(D) \otimes_{\mathcal{O}} \wtB^{[0, \frac{r_0}{p}]}[\ell_u] &\simeq & D\otimes_{K_0}  \wtB^{[0, \frac{r_0}{p}]}[\ell_u].
\end{eqnarray}
These isomorphisms induce (compatible) $G_K$-actions on left hand side of these equations.
\end{cor}
\begin{proof}
The isomorphisms all follow from Lem. \ref{lem541} (for \eqref{new545}, note that $\lambda$ is a unit in $\wtB^{[0, \frac{r_0}{p}]}$). Note that one could change the ``$[1/t]$" in \eqref{new544add} to $[1/\lambda]$ since $t/\lambda$ is a unit in $\wtb_{\log}^\dagger$ (cf. \S \ref{subsecfrakt}); but one cannot change  ``$[1/t]$" in \eqref{new544} to $[1/\lambda]$.
They induce $G_K$-actions on the left hand side of these equations, because the right hand side of these equations are $G_K$-stable.
\end{proof}

\begin{prop}\label{prop543}
Let $V \in \Rep_{\Qp}^{\st, \geq 0}(G_K)$, and  let $D=D_\st(V)$. Then we have a $(\varphi, N, G_K)$-equivariant isomorphism
\begin{equation}
 \label{new550} D\otimes_{K_0}  \wtB_{\log}^\dagger [1/t] \simeq D_{\rig, \Kinfty}^\dagger(V) \otimes_{\B_{\rig, \Kinfty}^\dagger} \wtB_{\log}^\dagger [1/t]. 
\end{equation}
\end{prop}
\begin{proof} 
By \cite[Prop. 3.4, Prop. 3.5]{Ber02}, we have
\begin{equation}
\label{new548}
D  \otimes_{K_0} \wtB_{\log}^\dagger [1/t]    \simeq V\otimes_{\Qp} \wtB_{\log}^\dagger [1/t].
\end{equation}
By our overconvergence theorem in \S \ref{secmod}, we have
\begin{equation}
\label{549} V\otimes_{\Qp} \wtB_{\rig}^\dagger  \simeq D_{\rig, \Kinfty}^\dagger(V) \otimes_{\B_{\rig, \Kinfty}^\dagger} \wtB_{\rig}^\dagger.
\end{equation}
Hence Eqn. \eqref{new550}   holds by combining \eqref{new548} and  \eqref{549}.
 \end{proof}

\begin{thm}\label{thmcoinc}
Let $V \in \Rep_{\Qp}^{\st, \geq 0}(G_K)$, and  let $D=D_\st(V)$.
\begin{itemize}
\item Let $D^\dagger_{\rig, \Kinfty}(V)$ be the rigid-overconvergent $(\varphi, \tau)$-module attached to $V$, and let $N_\nabla^{\mathrm{la}}$ denote the monodromy operator defined in Thm. \ref{thmnnabla}.
\item Let $(M(D), \varphi, N_\nabla^{\mathrm{Kis}}) \in \textnormal{Mod}_{ \mathcal{O}}^{\varphi, N_\nabla, 0}$ be the module corresponding to $D$ constructed by Kisin.
Extend $N_\nabla^{\mathrm{Kis}}$ to $M(D)\otimes _{\mathcal{O}} \mathbf{B}^\dagger_{\rig, \Kinfty}$ by $N_\nabla^{\mathrm{Kis}}\otimes 1+1\otimes N_{\nabla, \mathbf{B}^\dagger_{\rig, \Kinfty}}$, which we still denote as $N_\nabla^{\mathrm{Kis}}$.
\end{itemize}
Then there is a $\varphi$-equivariant isomorphism
\begin{equation}\label{542}
M(D)\otimes _{\mathcal{O}} \mathbf{B}^\dagger_{\rig, \Kinfty} \simeq D^\dagger_{\rig, \Kinfty}(V).
\end{equation} 
Furthermore, with respect to this isomorphism, we have 
\[N_\nabla^{\mathrm{Kis}} =N_\nabla^{\mathrm{la}}.\]
\end{thm}

\begin{proof}
By \eqref{537}, there is a $\varphi$-equivariant isomorphism $\gm \otimes_\gs \mathcal{O} \simeq M(D)$, hence it suffices to show that there is a $\varphi$-equivariant isomorphism
\begin{equation}\label{eqgmot}
\gm \otimes_\gs  \mathbf{B}^\dagger_{\Kinfty} \simeq D^\dagger_{\Kinfty}(V),
\end{equation}
where $D^\dagger_{\Kinfty}(V)$ is the overconvergent $\varphi$-module associated to $V$. The isomorphism in \eqref{eqgmot} holds by using Prop. \ref{propfinoc} and \eqref{538}.

By using Lem. \ref{lem266} for $\ell_u$ and \cite[Lem. 3.1.2(2)]{GP21} for $1/\lambda$, we have
\begin{equation}\label{eqnjuly}
\mathcal{O}[\ell_u, 1/\lambda]\subset (  \wtB_{\log, L}^\dagger [1/\lambda])^{\tau\dpa, \gamma=1}.
\end{equation}
By the construction \eqref{eqdefmd} and use \eqref{eqnjuly} above, it is clear that 
\begin{equation}
M(D) \subset D\otimes_{K_0} \mathcal{O}[\ell_u, 1/\lambda] \subset
(D\otimes_{K_0} \wtB_{\log, L}^\dagger [1/\lambda])^{\tau\dpa, \gamma=1}
\end{equation}
 Hence the $N_\nabla^{\mathrm{Kis}}$-operator on $M(D)$, which was defined in an ``algebraic" fashion below \eqref{eqdefmd}, indeed is induced by the ``locally analytic" $N_\nabla$-operator constructed using the locally analytic $\tau$-action on  $(D\otimes_{K_0} \wtB_{\log, L}^\dagger [1/\lambda])^{\tau\dpa}$.

 By \eqref{new550}, we have $G_K$-equivariant isomorphisms
\begin{equation}
 D\otimes_{K_0} \wtb_{\log}^\dagger[1/\lambda]\simeq  D\otimes_{K_0} \wtb_{\log}^\dagger[1/t] \simeq
  D_{\rig, \Kinfty}^\dagger(V) \otimes_{\B_{\rig, \Kinfty}^\dagger} \wtB_{\log}^\dagger[1/t];
\end{equation}
  here, the first isomorphism follows from the fact that $t/\lambda$ is a unit in $\wtb_{\log}^\dagger$ (cf. \S \ref{subsecfrakt}). 
 Thus the $G_K$-action on   $D\otimes_{K_0} \wtb_{\log}^\dagger[1/\lambda]$ is ``compatible" with the $G_K$-action on the   rigid-overconvergent $(\varphi, \tau)$-module (which induces $N_\nabla^{\mathrm{la}}$). Hence we must have $N_\nabla^{\mathrm{Kis}} =N_\nabla^{\mathrm{la}}.$
\end{proof}




\section{Frobenius regularization and finite height representations}\label{secfinht}
In this section, we use our monodromy operator to study finite $E(u)$-height representations.
In \S \ref{subsfrobreg}, we show that the monodromy operator can be descended to the ring $\mathcal{O}$ for a finite $E(u)$-height representation; in \S \ref{subsecpst}, we show such representations are potentially semi-stable.
The  results will be used in \S \ref{secsimp} to construct the Breuil-Kisin $G_K$-modules.

\subsection{Frobenius regularization of the monodromy operator} \label{subsfrobreg}

\begin{prop}\label{propregfrob}
Suppose $T \in \Rep_{\Zp}(G_K)$ is of finite $E(u)$-height (with respect to the fixed choice  of $\vec{\pi}=\{\pi_n\}_{n \geq 0}$), and let $\gm \in \Mod^{\varphi}_{\gs}$ be the corresponding Breuil-Kisin module.
Let $N_\nabla$ be the monodromy operator constructed in Thm. \ref{thmnnabla}, then $N_\nabla(\huaM) \subset \huaM \otimes_{\gs} \mathcal{O}$.
\end{prop}

We will use a ``Frobenius regularization" technique to prove
 Prop. \ref{propregfrob}. Roughly, by Thm. \ref{thmnnabla}, we already know that the coefficients of   the matrix for  $N_\nabla$  live near the boundary of the open unit disk; to show that they indeed live on the entire open unit disk in the finite height case, we use the Frobenius operator to ``extend" their range of convergence: this is where we critically use  the finite height condition for the Frobenius operator.
 Indeed, the proof relies on the following key lemma.

\begin{lemma} \label{lemfrobreg} Let $h \in \mathbb{R}^{>0}$ and $r \in \mathbb{Z}^{>0}$. For $s>0$, let $\wt{\mathbf{A}}^{[s, +\infty)}$ be the subset consisting of $x\in \wt{\mathbf{B}}^{[s, +\infty)}$ such that $W^{[s, s]}(x) \geq 0$, and let ${\mathbf{A}}^{[s, +\infty)} : =\wt{\mathbf{A}}^{[s, +\infty)} \cap {\mathbf{B}}^{[s, +\infty)}$. Then we have:
\begin{eqnarray}
\label{eqcap1}\bigcap_{n \geq 0} p^{-hn}\wt{\mathbf{A}}^{[r/p^n, +\infty]}&=&\wt{\mathbf{A}}^{[0, +\infty]}=\wt{\mathbf{A}}^+\\
\label{eqcap2}\bigcap_{n \geq 0} p^{-hn}\wt{\mathbf{A}}^{[r/p^n, +\infty)}&\subset&\wt{\mathbf{B}}^{[0, +\infty)}\\
\label{eqcap3}\bigcap_{n \geq 0}p^{-hn} {\mathbf{A}}_{\Kinfty}^{[r/p^n, +\infty]} &=& {\mathbf{A}}_{\Kinfty}^{[0, +\infty]}=\gs\\
\label{eqcap4}\bigcap_{n \geq 0}p^{-hn} {\mathbf{A}}_{\Kinfty}^{[r/p^n, +\infty)} &\subset& {\mathbf{B}}_{\Kinfty}^{[0, +\infty)}=\mathcal{O}
\end{eqnarray}
\end{lemma}
\begin{proof}
The relations \eqref{eqcap1} and \eqref{eqcap2} are from \cite[Lem. 3.1]{Ber02}. We can intersect $\mathbf{B}_{\Kinfty}$ with \eqref{eqcap1} to get \eqref{eqcap3}.
For \eqref{eqcap4}, it follows from similar argument as in \cite[Lem. 3.1]{Ber02}. Indeed, suppose $x \in \text{LHS}$, then for each $n \geq 0$, we can write $x=a_n+b_n$ with $a_n \in p^{-hn} {\mathbf{A}}_{\Kinfty}^{[r/p^n, +\infty]}$ and $b_n \in {\mathbf{B}}_{\Kinfty}^{[0, +\infty)}$. Then
\[a_n-a_{n+1} \subset (p^{-h(n+1)} {\mathbf{A}}_{\Kinfty}^{[r/p^n, +\infty]})\cap {\mathbf{B}}_{\Kinfty}^{[0, +\infty)} = p^{-h(n+1)} {\mathbf{A}}_{\Kinfty}^{[0, +\infty]} \subset {\mathbf{B}}_{\Kinfty}^{[0, +\infty)}=\mathcal{O}.\]
By modifying $a_{n+1}$, we can assume $a_n=a$ for all $n \geq 0$, and hence $a\in \gs$ by \eqref{eqcap3}. Thus, $x=a+b_0 \in \mathcal{O}$.
\end{proof}

\begin{rem}
We use an example to illustrate the idea of \eqref{eqcap4}. Consider the element $1/u \in \cap_{n \geq 0}  {\mathbf{B}}_{\Kinfty}^{[r/p^n, +\infty)}$, it does not belong to the left hand side of \eqref{eqcap4} because the valuations $W^{[r/p^n,r/p^n]}(1/u)$ converge to $-\infty$ in an exponential rate, rather than the linear rate on the left hand side of \eqref{eqcap4}.
\end{rem}

By \eqref{eqcap4}, the proof of Prop.  \ref{propregfrob} will rely on some calculations of valuations and ranges of convergence; we first list two   lemmas.
For the reader's convenience, recall that for $x= \sum_{i \geq i_0} p^i[x_i] \in \wt{\mathbf{B}}^{[r, +\infty]}$ and for $s\geq r, s>0$ we have the formula
\begin{equation} \label{614ab}
W^{[s, s]}(x) :=\inf_{k \geq i_0} \{k+\frac{p-1}{ps}\cdot v_{\wt{\mathbf E}}(x_k)\} =   \inf_{k \geq i_0} \{ k+\frac{p-1}{ps}\cdot w_k(x)\}. 
\end{equation}

\begin{lemma}\label{lemsseu}
\begin{enumerate}
\item  Suppose $x \in \wt{\mathbf{B}}^+$, then
\begin{eqnarray}
\label{eqrrss}  W^{[r, r]}(x)& \geq &W^{[s, s]}(x),\quad \forall0 <r<s <+\infty.\\
\label{eqssphi} W^{[s, s]}(\varphi(x)) & \geq &W^{[s, s]}(x),\quad \forall s \in (0, +\infty)
\end{eqnarray}

\item We have
\begin{equation*}
\label{sseu} W^{[s, s]}(E(u)) \in (0, 1], \quad \forall s \in (0, +\infty)
\end{equation*}

\item Recall $\lambda =\prod_{n \geq 0} \varphi^n(\frac{E(u)}{E(0)})$. Let $\ell \in \mathbb{Z}^{\geq 0}$. Then
\begin{equation*}
\label{sslambda} W^{[s, s]}(\lambda) \geq -\ell, \quad \forall s \in (0, r_\ell].
\end{equation*}
\end{enumerate}
\end{lemma}
\begin{proof}
Item (1) follows from definition.

For Item (2), first recall that $W^{[s, s]}$ is multiplicative.
Note that $E(u)= \sum_{i=0}^{e }a_iu^i$ where $a_e=1$, $p|a_i$ for $i>0$ and $p||a_0$. It is easy to see that $W^{[s, s]}(E(u))>0$ for all $s>0$.
When $s\leq r_0$,   $W^{[s, s]}(u^e+a_0)=1$ by Eqn. \eqref{614ab}, $W^{[s, s]}(\sum_{i=1}^{e-1 }a_iu^i)>1$, and hence $W^{[s, s]}(E(u))=1$ . Hence $W^{[s, s]}(E(u))\leq 1$ if $s>r_0$ by Item (1).

For Item (3), we have $W^{[s, s]}(\lambda) =\sum_{i \geq 0} W^{[s, s]}(\varphi^i(E(u)/E(0)))$.
 Using Item (1), it suffices to treat the case $s=r_\ell$. We have
 \begin{eqnarray*} 
   W^{[r_\ell, r_\ell]}(\lambda) &=&W^{[r_\ell, r_\ell]}(\varphi^\ell(\lambda)) +  W^{[r_\ell, r_\ell]}(\prod_{i=0}^{\ell -1} \varphi^i(\frac{E(u)}{E(0)}))\\
    &\geq &  W^{[r_0, r_0]}(\lambda) +W^{[r_\ell, r_\ell]}(\prod_{i=0}^{\ell -1} \varphi^i(\frac{1}{E(0)}))\\
    &\geq & -\ell
 \end{eqnarray*}
 where the last row uses that $W^{[r_0, r_0]}(\lambda)>0$ (note $W^{[r_0, r_0]}(\frac{E(u)}{E(0)})=0$ and apply \eqref{eqssphi}).
\end{proof}

\begin{lemma}\label{lemeu}
  Suppose $x \in \mathbf{B}_{\Kinfty}^{I}$.
  Let $g(u) \in K_0[u]$ be an irreducible polynomial such that $g(u) \notin K_0[u^p] =\varphi(K_0[u])$.
Suppose   $g(u)^{h}\varphi(x) \in \mathbf{B}_{\Kinfty}^I$ for some $h \geq 1$, then we have $x\in \mathbf{B}_{\Kinfty}^{I/p}$.
\end{lemma}
\begin{proof}
We can suppose $\inf I \neq 0$ (otherwise the lemma is trivial).
Then we can further suppose $g(u) \neq au$ for some $a \in K_0$ (otherwise the lemma follows easily from Cor. \ref{corIpI}).
First we treat the case $h=1$.
Suppose $g=\sum_{i=0}^N b_iu^i$. Let $g_1:=\sum_{p\mid i} b_iu^i$ (i.e., $g_1$ contains all the $u$-powers with  $p$-divisible exponents, including the non-zero constant term) and let $g_2:=\sum_{p\nmid i} b_iu^i$; hence $g_2 \neq 0$ but $g_2 \neq g$.
Recall that if we write  $x =\sum_{i \in \mathbb Z}a_iu^i$ with $a_i \in K_0$, then it satisfies certain convergence conditions with respect to $I$ as described in Lem. \ref{lem laurent series}; similarly, the expansion of the product $g(u)\varphi(x)=(g_1+g_2) \cdot (\sum_{i \in \mathbb Z}\varphi(a_i)u^{pi})$ satisfies these convergence conditions. However, there is no ``intersection" between the two parts $g_1\varphi(x)$ and $g_2\varphi(x)$ as the former part contains exactly all the  $u$-powers with  $p$-divisible exponents.
Hence \emph{both} $g_1\varphi(x)$ and $g_2\varphi(x)$ satisfy the aforementioned convergence conditions, and hence \emph{both} $g_1\varphi(x), g_2\varphi(x) \in \mathbf{B}_{\Kinfty}^{I}$.
Since   $g$ and $g_2$ are co-prime in  $K_0[u]$,   we must have $\varphi(x) \in \mathbf{B}_{\Kinfty}^I$. Thus  $x\in \mathbf{B}_{\Kinfty}^{I/p}$ by Cor. \ref{corIpI}.

Suppose now $h\geq 2$, and write $g^h=g_1+g_2$ similarly as above. Similarly we have $g_2\varphi(x) \in \mathbf{B}_{\Kinfty}^I$. If $(g^h, g_2)=g^{k}$ with $k<h$, then $g^k \varphi(x) \in \mathbf{B}_{\Kinfty}^{I}$. This reduces the proof to an induction argument.
\end{proof}

\begin{rem}
The condition $g(u) \notin K_0[u^p]$ in Lem. \ref{lemeu} is necessary. For example, suppose $e=p$, let
$g(u)=u^p-p$, and let $x=$ ``$\frac{1}{u-p}$" $:= \frac{1}{u}(\sum_{i \geq 0}(\frac{p}{u})^i) \in \mathbf{B}_{\Kinfty}^{[r_0, +\infty)}$. We have $g(u)\varphi(x)=1$, but  $x \notin \mathbf{B}_{\Kinfty}^{[r_0/p, +\infty)}$.
\end{rem}

\begin{proof}[\textbf{Proof} of Prop.  \ref{propregfrob}]
Let $\vece$ be an $\gs$-basis of $\huaM$, and suppose
\[\varphi(\vece)=\vece A, N_\nabla(\vece)=\vece M, \text{ with } A \in \Mat(\gs), M\in \Mat(\mathbf{B}^\dagger_{\rig, \Kinfty}).\]
Since $N_\nabla \varphi =\frac{p}{E(0)}E(u)\varphi N_\nabla$, we have
\begin{equation} \label{eqphi1}
MA+N_\nabla(A) =\frac{p}{E(0)}E(u) A\varphi(M).
\end{equation}
Let $B \in \Mat(\gs)$ such that  $AB=E(u)^h\cdot Id$, then we get
\begin{equation} \label{eqB}
BMA+BN_\nabla(A) = \frac{p}{E(0)}E(u)^{h+1}\varphi(M).
\end{equation}

Suppose $N \geq 0$ is the maximal number such that $E(u) \in K_0[u^{p^N}]$. Denote
\[D_i(u): =  \prod_{k=1}^i \varphi^{-k}(E(u)), \quad \forall 0\leq i \leq N; \]
where we always let $D_0(u):=1$.
For brevity, we denote $E:=E(u)$ and $D_i:=D_i(u)$.

Let
\[\wt{M}:= D_N^{h+1}M,\]
then from \eqref{eqB}, we have
\begin{equation}\label{eqwtm}
B\wt{M}A+ D_N^{h+1}BN_\nabla(A)=\frac{p}{E(0)}\varphi^{-N}(E^{h+1})\varphi(\wt M).
\end{equation}
Let $\ell \gg 0$   so that
\[\wt{M}\in \Mat( \mathbf{B}_{\Kinfty}^{[r_\ell, +\infty)}).\]
Note that $A, B, D_i \in \Mat(\gs)$ and $N_\nabla(A) \in \lambda \cdot \Mat(\gs)$, hence
\[\LHS \text{ of \eqref{eqwtm} } \in \Mat( \mathbf{B}_{\Kinfty}^{[r_\ell, +\infty)}).\]
Since $\varphi^{-N}(E)$ satisfies the conditions in  Lem. \ref{lemeu}, we can iteratively use the lemma and \eqref{eqwtm} to conclude that
\begin{equation}\label{eqNminusall}
\wt{M}\in \Mat \left(\bigcap_{n \geq 0} \mathbf{B}_{\Kinfty}^{[r_\ell/p^{n}, +\infty)}\right).
\end{equation}
To show that $M\in \Mat(\mathcal{O})$, our proof proceeds in 2 steps: we first show  that $\wt{M} \in \Mat(\mathcal{O})$ using Lem. \ref{lemfrobreg}; then we show that $M\in \Mat(\mathcal{O})$ using a trick of Caruso.

\textbf{Step 1:}
Write $r=r_\ell.$
Choose $c\gg 0$ (in particular, we ask that $c>\ell$) such that
\[\wt M \in \Mat( p^{-c}\mathbf{A}_{\Kinfty}^{[r, +\infty)}).\]
Then we have
\begin{eqnarray*}
W^{[r, r]}(\text{LHS of \eqref{eqwtm}}) &\geq & \min \{W^{[r, r]}(\wt M), W^{[r, r]}(\lambda) \}, \text{ since } W^{[r, r]}   \text{ is multiplicative} \\
&\geq & \min \{ -c, -\ell\}, \text{ by Lem. \ref{lemsseu}(3) } \\
&= &  -c
\end{eqnarray*}
By Lem. \ref{lemsseu}(2), we have
\[W^{[r, r]}(\varphi^{-N}(E)) = W^{[p^Nr, p^Nr]}(E)  \leq 1.\]
Hence using \eqref{eqwtm}, we have
\[W^{[r/p, r/p]}(\wt M) =W^{[r, r]}(\varphi(\wt M)) \geq -c-(h+1).\]
Iterate the above argument, we get that for all $n$,
\begin{equation*}\label{eqrnn}
W^{[r/p^n, r/p^n]}(\wt M)  \geq -c-n(h+1).
\end{equation*}
By Lem. \ref{lemcompa}, we have
\[\wt M \in \Mat( p^{ -c-n(h+1)} \mathbf{A}^{[r/p^n, +\infty)}_{\Kinfty}), \quad \forall n. \]
Using Eqn. \eqref{eqcap4} in Lem. \ref{lemfrobreg}, we can conclude that
\[\wt M \in \Mat(\mathcal{O}).\]

\textbf{Step 2:} In this  step, we show that $M\in \Mat(\mathcal{O})$. If $N=0$, then there is nothing to prove. Suppose now $N \geq 1$.
So we have
\begin{equation}\label{eq31}
\wt M=D_N^{h+1}\cdot M = \varphi^{-N}(E^{h+1})\cdot D_{N-1}^{h+1}\cdot M  \in \Mat (\mathcal{O}).
\end{equation}
From \eqref{eqphi1}, we also have
\begin{equation*}
ME^h+N_\nabla(A)B =\frac{p}{E(0)}E A\varphi(M)B,
\end{equation*}
and hence
\begin{equation*}
E^{h+1} D_{N-1}^{h+1}\cdot (ME^h+N_\nabla(A)B) = \frac{p}{E(0)}E  A\varphi(\wt M)B.
\end{equation*}
So we have
\begin{equation}\label{eq32}
E^{h+1} D_{N-1}^{h+1} \cdot ME^h =E^{2h+1}\cdot D_{N-1}^{h+1} \cdot M \in \mat(\mathcal{O}).
\end{equation}
Note that both $\varphi^{-N}(E)$ and $E$ are irreducible in $K_0[u]$ and hence they are co-prime. Hence we can use \eqref{eq31} and \eqref{eq32} to conclude that
\begin{equation*}
D_{N-1}^{h+1} \cdot M\in \mat(\mathcal{O}).
\end{equation*}
If $N-1\geq 1$, we can repeat the above argument. (Note that in the argument of this Step 2, we do not use the fact $\varphi^{-N}(E) \notin K_0[u^p]$; indeed, this condition is used only  to conclude \eqref{eqNminusall}). Hence in the end we must have
\begin{equation*}
 M\in \mat(\mathcal{O}).
\end{equation*}
\end{proof}

\begin{rem}\label{remstep3}
The argument in Step 2 above is taken from the paragraph containing \cite[p. 2595, Eqn. (3.15)]{Car13}; in particular, the use of $D_N(u)$ is inspired by the argument in \emph{loc. cit.}. However, the argument before Step 2 are completely different from those in \cite{Car13}.
\end{rem}

\subsection{Potential  semi-stability of finite height representations}\label{subsecpst}
In this subsection, we show that finite height representations are potentially semi-stable; in fact, our result is more precise and stronger. Let us first recall two useful lemmas.

For any $K \subset X \subset \overline{K}$, let
\[m(X): =1+\max \{i \geq 1, \mu_{i} \in X\} .\]
Recall for each $n \geq 1$, we let $K_n=K(\pi_{n-1})$ (hence $K_1=K$). Note that
\begin{itemize}
\item (for $n\geq 2$), $K_n$ here $=``K_n"$ in \cite{Kis06} $=``K_{n-1}"$ in \cite{Liu10} and \cite{Oze17}.
\end{itemize}

\begin{lemma}\label{lemozfix}
Let $m:=m(K^{\mathrm{ur}})$  where $K^{\mathrm{ur}}$ is the maximal unramified extension of $K$ (contained in $\overline K$).
Suppose $V \in \Rep_{\Qp}(G_K)$ is semi-stable over $K_n$ for some $n \geq 1$, then $V$ is semi-stable over $K_m$.
\end{lemma}
\begin{proof}
This is proved in \cite[Rem. 2.5]{Oze17}.
Note that this fixes a gap in \cite[Thm. 4.2.2]{Liu10}, where Liu claims a similar statement using ``$m:=m(K)$" instead. Note that in general, it is possible that $m(K^{\mathrm{ur}})>m(K)$, cf. \cite[\S 5.4]{Oze17}. Let us mention that this lemma is completely about Galois theory of the fields $K_n$, and has nothing to do with the ``$(\varphi, \hat{G})$-modules" in \emph{loc. cit.}.
\end{proof}

\begin{lemma}\label{lemramres}
Let $K \subset  K^{(1)} \subset  K^{(2)}$ be finite extensions such that $K^{(2)}/K^{(1)}$ is totally ramified, then the restriction functor from semi-stable $G_{K^{(1)}}$-representations to semi-stable $G_{K^{(2)}}$-representations is fully faithful.
\end{lemma}
\begin{proof}
This is \cite[Lem. 4.11]{Oze17}. (It is basically \cite[Prop. 3.4]{Car13}, but it is completely elementary).
\end{proof}

The following definition comes from \cite{Oze17}.

 \begin{defn}
 Fix a choice of $\vec{\pi}$.
Let $n \geq 1$.
we denote by  $\mathcal{C}_n(\vec{\pi})$
the category  of finite free $\Zp$-representations $T$ of $G_K$
such that we have an $G_{K_n}$-equivariant isomorphism
\[T[1/{p}] |_{G_{K_n}} \simeq W|_{G_{K_n}}\]
 for some $W\in \Rep_{\Qp}^{\st, \geq 0}(G_K)$; namely $T[1/{p}] |_{G_{K_n}}$ is semi-stable and can be extended to a semi-stable $G_K$-representation.
\end{defn}

\begin{thm}\label{thmmainpst}
Let $\Rep_{\Zp}^{E(u)\mbox{-}\mathrm{ht}}(G_K)$ be the category of finite $E(u)$-height $\Zp$-representations (with respect to the fixed choice of $\vec{\pi}$). Then 
\[\Rep_{\Zp}^{E(u)\mbox{-}\mathrm{ht}}(G_K) =\mathcal{C}_{m(K^{\mathrm{ur}})}(\vec{\pi}).\]
In particular, if $T \in \Rep_{\Zp}^{E(u)\mbox{-}\mathrm{ht}}(G_K)$, then $T[1/p]|_{G_{K_{{m(K^{\mathrm{ur}})}}}}$ is semi-stable.
\end{thm}

\begin{proof}
Suppose $T \in \mathcal{C}_{m(K^{\mathrm{ur}})}(\vec{\pi})$, and let $V=T[1/p]$. Let  $U\in \Rep_{\Qp}^{\st, \geq 0}(G_K)$ such that  \[V |_{G_{K_{m(K^{\mathrm{ur}})}}} \simeq U|_{G_{K_{m(K^{\mathrm{ur}})}}}.\] 
By Kisin's result, $U$ is of finite $E(u)$-height \emph{with respect to} $\vec{\pi}$, hence so are $U|_{G_\infty} \simeq V|_{G_\infty}$, and hence so are $V$ and $T$ (by Def. \ref{deffinht}). (Note that the semi-stability of $V |_{G_{K_{m(K^{\mathrm{ur}})}}}$ is \emph{not} enough to guarantee that $V$ is of finite height with respect to $\vec{\pi}$; it only guarantees finite heightness with respect to $\{\pi_n\}_{n \geq m(K^{\mathrm{ur}})}$.)


Conversely, let $V$ be a finite height representation.
By Prop. \ref{propregfrob},   we can construct a triple $(\gm \otimes_{\gs} {\mathcal{O}}, \varphi, N_\nabla) \in \textnormal{Mod}_{ \mathcal{O}}^{\varphi, N_\nabla, 0}$, which gives us a \emph{semi-stable} $G_K$-representation $W$.
Here $(\gm \otimes_{\gs} {\mathcal{O}}, \varphi)$ is pure of slope zero because $(\gm \otimes_{\gs} \mathbf{B}_{\Kinfty}, \varphi)$ is part of the \'etale $(\varphi, \tau)$-module associated to $V$.

Let $(D_{K_\infty}(V), \varphi, \hat{G}_V)$ be the \'etale $(\varphi, \tau)$-module associated to $V$, and let $D_{K_\infty}^\dagger(V)$ (resp. $D_{\rig, K_\infty}^\dagger(V)$) be the overconvergent (resp. rigid-overconvergent) module equipped with the induced $\varphi$ and $\hat{G}_V$. Let $(D_{K_\infty}(W), \varphi, \hat{G}_W)$, $D_{K_\infty}^\dagger(W)$, $D_{\rig, K_\infty}^\dagger(W)$ be similarly defined.
It is clear that
\begin{align*}
D_{K_\infty}(V) &= \gm \otimes_{\gs} \mathbf{B}_{K_\infty} = D_{K_\infty}(W)\\
D_{K_\infty}^\dagger(V) &= \gm \otimes_{\gs} \mathbf{B}_{K_\infty}^\dagger  =  D_{K_\infty}^\dagger(W)\\
D_{\rig, K_\infty}^\dagger(V) &= \gm \otimes_{\gs} \mathbf{B}_{\rig, K_\infty}  = D_{\rig, K_\infty}^\dagger(W)
\end{align*}
and they carry the same $\varphi$-action.

In the following, let $Q \in \{V, W\}$.
By the definition of  locally analytic actions, the $\tau_{Q}$-action  over $D_{\rig, K_\infty}^\dagger(Q)$ can be ``locally" recovered by 
\begin{equation*}
\nabla_{\tau, Q}=
\begin{cases} 
p\mathfrak{t}\cdot  N^{\la}_{\nabla, Q}, &  \text{if }  \Kinfty \cap \Kpinfty=K, \\
  4\mathfrak{t}\cdot  N^{\la}_{\nabla, Q}, & \text{if }  \Kinfty \cap \Kpinfty=K(\pi_1),
\end{cases}
\end{equation*}
 where $N^{\la}_{\nabla, Q}$ is the monodromy operator defined in Thm. \ref{thmnnabla}.
Indeed, let $\vec{e}$ be an $\gs$-basis of $\gm$, then there exists some $a\gg 0$ such that
\begin{equation}\label{eqtauvw}
\tau^{\alpha}_{Q}(\vec{e}) =\sum_{i=0}^{\infty} \alpha^{i}\cdot\frac{(\nabla_ {\tau, Q})^i(\vec{e})}{i!}, \quad \forall \alpha \in p^a \Zp.
\end{equation}
Note that \emph{a priori}, the series in \eqref{eqtauvw} converges to some element in $\wt{\mathbf{B}}_{\rig, L}^\dagger \otimes_{\mathbf{B}_{\rig, K_\infty}^\dagger} D_{\rig, K_\infty}^\dagger(Q)$;
however since $\vec{e}$ is also basis for $D_{K_\infty}(Q)$ and $D_{K_\infty}^\dagger(Q)$, the limit has to fall inside
\[\wt{\mathbf{B}}_{  L}^\dagger \otimes_{\mathbf{B}_{\Kinfty}^\dagger} D_{K_\infty}^\dagger(Q) \subset \wt{\mathbf{B}}_{ L} \otimes_{\mathbf{B}_{\Kinfty} }D_{K_\infty}(Q).\]
Now, $N^{\la}_{\nabla, V} =N_\nabla$ by construction, and
$N^{\la}_{\nabla, W} =N_\nabla$ by Thm. \ref{thmcoinc}.
Thus we have
\[N^{\la}_{\nabla, V}=N^{\la}_{\nabla, W}.\]
This implies that the $\gal(L/\Kpinfty(\pi_a))$-action on the two  $(\varphi, \tau)$-modules are the same. Since $\Kpinfty(\pi_a) \cap \Kinfty \subseteq K(\pi_{a+1})$ (possible equality only when $p=2$, cf. Notation \ref{nota hatG}).
Hence we must have
\[V|_{G_{K_{a+2}}}=W|_{G_{K_{a+2}}}.\]



We can always first choose $a \geq m(K^{\mathrm{ur}}) $, and hence by Lem. \ref{lemozfix}, $V$ is semi-stable over $K_{m(K^{\mathrm{ur}})}$. Thus by Lem. \ref{lemramres} (using the totally ramified extension $K_{a+2}/K_{m(K^{\mathrm{ur}})}$), we have:
\[V|_{G_{K_{m(K^{\mathrm{ur}})}}}=W|_{G_{K_{m(K^{\mathrm{ur}})}}}.\]
\end{proof}

\begin{rem}\label{remwrongstatement}
In \cite[Thm. 3]{Car13}, the statement there claims that $\Rep_{\Zp}^{E(u)\mbox{-}\mathrm{ht}}(G_K)$ is the same as $\mathcal{C}_{m(K)}(\vec{\pi})$. However, by our Thm. \ref{thmmainpst}, and by the  examples in \cite[Prop. 3.22(1)]{Oze17} which shows that $\mathcal{C}_{m(K)}(\vec{\pi}) \neq \mathcal{C}_{m(K^\ur)}(\vec{\pi})$ in general, Caruso's statement is in general false.
\end{rem}

\section{Breuil-Kisin $G_K$-modules} \label{secsimp}

\newcommand{\bx}{\wtB_{\log}^\dagger[1/t]}
\newcommand{\by}{\wtB^{[0, \frac{r_0}{p}]}[\ell_u]}
\newcommand{\bz}{\wtb_{\log}^+[1/t]}

In \S \ref{subssimple}, we construct the Breuil-Kisin $G_K$-modules and show that they classify integral semi-stable Galois representations.
In \S \ref{72},  we discuss the relation between our theory and Liu's $(\varphi, \hat{G})$-modules (only preliminarily here).
In \S \ref{73}, we discuss the relation between our theory and  some results of Gee and Liu.

\subsection{Breuil-Kisin $G_K$-modules} \label{subssimple}
In this subsection, we construct the Breuil-Kisin $G_K$-modules. 
Recall that in Notation \ref{114}, we defined the notations $R, W(R), \mathfrak{m}_R$; the ring $W(R)$ is precisely $\wtA^+$ in \S \ref{secring}, and is also denoted as $\Ainf$ in the literature. Let $\Fr R$ be the fractional field of $R$, and let $W(\Fr R)$ be the Witt vectors; this is precisely $\wtA$ in  \S \ref{secring}.

 
\begin{defn}\label{defwr}
Let $\textnormal{Mod}_{\gs, W(R)}^{\varphi, G_K}$
be the category consisting of triples $(\mathfrak{M}, \varphi_{\gm}, G_K)$ which we call the (effective) \emph{Breuil-Kisin $G_K$-modules}, where 
\begin{enumerate}
\item $(\mathfrak{M}, \varphi_\mathfrak{M})\in \text{Mod}_{\gs}^{\varphi}$;
\item  $G_K$ is a continuous $W(R)$-semi-linear   $G_K$-action on $\hM:= W(R)
\otimes_{ \gs} \mathfrak{M}$;
\item $G_K$ commutes  with $\varphi_{\widehat{\mathfrak{M}}}$ on $\widehat{\mathfrak{M}}$;

\item $\gm \subset \hM^{\gal(\overline{K}/\Kinfty)}$ via the embedding $\gm \hookrightarrow \hM$;
\item  $\gm/u\gm \subset (\hM/W(\mathfrak{m}_R)\hM)^{G_K}$ via the embedding $\gm/u\gm \hookrightarrow \hM/W(\mathfrak{m}_R)\hM$.
\end{enumerate}
  \end{defn}



We record an equivalent condition for Def. \ref{defwr}(5).
\begin{lemma}\label{lemeq5}
Let  $(\mathfrak{M}, \varphi_{\gm}, G_K)$ be a triple satisfying Items (1)-(4) of Def. \ref{defwr}, then the condition in Def. \ref{defwr}(5) is satisfied if and only if
$\hM/W(\mathfrak{m}_R)\hM$ is fixed by $G_{K^\ur}$.
\end{lemma}
\begin{proof}
Necessity is obvious, and we prove sufficiency. If $\hM/W(\mathfrak{m}_R)\hM$ is fixed by $G_{K^\ur}$, then so is $\gm/u\gm$. But  $\gm/u\gm$ is also fixed by $G_{\Kinfty}$ by Def. \ref{defwr}(4), hence it is fixed by $G_K$ as $K^\ur \cap \Kinfty=K$.
\end{proof}

\begin{defn}\label{defwrweak}
  Let $\textnormal{wMod}_{\gs, W(R)}^{\varphi, G_K}$
denote the category consisting of triples $(\mathfrak{M}, \varphi_{\gm}, G_K)$ satisfying Items (1)-(4) of Def. \ref{defwr}. 
\end{defn}

For $\hM=  (\mathfrak{M}, \varphi_{\gm}, G_K)$ as in Def. \ref{defwrweak}, we can associate a $\Z_p[G_K]$-module:
\begin{equation}\label{weak1th}
T_{W(R)} (\hM)  :=  ( \hM \otimes_{W(R)} W(\Fr R))^{\varphi=1}.
\end{equation}

\begin{prop} \label{propweakff}
Eqn. \eqref{weak1th} induces a rank-preserving (i.e., $\mathrm{rk}_{\Zp} T_{W(R)} (\hM) =\mathrm{rk}_{W(R)}\hM$), exact, and fully faithful functor
$T_{W(R)}: \textnormal{wMod}_{\gs, W(R) }^{\varphi, G_K} \to \Rep_{\Zp}(G_K)$.
\end{prop}
\begin{proof}
The contra-variant version of this proposition (except exactness) is proved below  \cite[Rem. 3.1.5]{Liu10}. Note that the proof makes critical use of \cite[Lem. 3.1.2, Prop. 3.1.3]{Liu10}; but these results have nothing to do with the ring ``$\hR$" there (and are relatively easy). 
By the argument in the proof \cite[Thm. 2.3.2]{GL20}, the $\Zp$-dual of our $T_{W(R)} (\hM)$ is isomorphic to ``$\hat{T}(\hM)$" of \cite{Liu10}; hence the proposition (except exactness) follows.

It is shown in  \cite[Cor. 2.1.4]{Kis06} that the contra-variant functor $\gm \mapsto  \Hom_{\gs, \varphi} (\gm, \mathbf{A})$ is exact (note that ``$\gs^{\ur}$" there is precisely our $\mathbf A$ in \S \ref{subsecwte}). By the discussions in \cite[\S 2.1, 2.2, 2.3]{GL20} about co-variant versions of various functors in \cite{Kis06}, one deduces that the co-variant functor $\gm \mapsto  (\gm \otimes_\gs W(\Fr R))^{\varphi=1}$ in Def. \ref{d513} is exact. This implies that $T_{W(R)}$ is exact.
\end{proof}

\begin{theorem}\label{thmnoweak}
We have equivalences of categories
\begin{equation}
\textnormal{wMod}_{\gs, W(R)}^{\varphi, G_K }
\xrightarrow{T_{W(R)}}
\Rep_{\Zp}^{E(u)\mbox{-}\mathrm{ht}}(G_K)
 \xrightarrow{=} \mathcal{C}_{m(K^\ur)}(\vec{\pi}).
\end{equation}
\end{theorem}

\begin{proof}
  The last equivalence is Thm. \ref{thmmainpst}. $T_{W(R)}$ is fully faithful by Prop. \ref{propweakff}. Hence it suffices to show that $T_{W(R)}$ is essentially surjective.

Suppose $T \in \Rep_{\Zp}^{E(u)\mbox{-}\mathrm{ht}}(G_K)$, let $\gm$ be the   Breuil-Kisin module  attached to $T$, and let $V=T[1/p]$. 
By Thm. \ref{thmmainpst}, $V|_{G_{K_m}}$ is semi-stable where $m=m(K^\ur)$. Let $E$ be the $p$-adic completion of $K^{\ur}(\pi_{m-1})$; then $V|_{G_{K^{\ur}(\pi_{m-1})}}$ is a semi-stable representation of $G_E$.
Note that this $E$   satisfies the assumption in Notation \ref{notaEK}. Note furthermore that $K^{\ur}(\pi_{m-1})$ is Galois over $K$.

Since the $G_E$-representation is semi-stable, we can construct an  \emph{$\gs_E$-Breuil-Kisin module} (whose obvious definition we leave it to the reader)    $\gm_E$.  
Recall by Lem. \ref{lemidEK}, we have the $\varphi$-equivariant isomorphism
\begin{equation}\label{dmme}
D_{K_\infty}(V)\otimes_{\B_{\Kinfty}}  \B_{E_\infty}  \simeq   D_{E_\infty}(V|_{G_E}).
\end{equation}
We claim that 
\begin{equation}\label{mme}
\gm \otimes_{\gs}  \gs_E \simeq \gm_{E}.
\end{equation}
To prove the claim, note that under the map $u \mapsto u_E^{p^{m-1}}$ (cf. Notation \ref{notaEK}), $E(u)$ maps to the minimal polynomial  of $\pi_{m-1}$ over $W(\overline k)[1/p]$; hence $\gm \otimes_{\gs}  \gs_E$ is also a $\gs_E$-Breuil-Kisin module. The isomorphism \eqref{mme} follows because both sides give rise to the same $G_{E_\infty}$-representation $T|_{G_{E_\infty}}$.

Now we want to construct a $G_K$-action on $\gm \otimes_{\gs}W(R)$: indeed, we will show that $\gm \otimes_{\gs}W(R)$ is $G_K$-stable inside $D^\dagger_{\rig, \Kinfty}(V) \otimes \bx$.
By the isomorphism \eqref{mme}, it is natural to try to use the Galois actions related to $\gm_E$.
 Indeed, we will make use of the following commutative diagram. Here we omit all the subscripts of the tensor products for brevity; the $G_K$'s  over the arrows signify $G_K$-equivariances.
\begin{equation}\label{diag714}
 \xymatrix
{
V\otimes  \bx      \ar[r]^-\simeq_{G_K}  &  D_{\st}^E(V) \otimes \bx    \ar[r]^-\simeq_{G_K}   & D^\dagger_{\rig, E_\infty}(V) \otimes \bx   \ar[r]^-\simeq_{G_K}   & D^\dagger_{\rig, \Kinfty}(V) \otimes \bx \\
&  D_{\st}^E(V) \otimes \bz  \ar[r]^-\simeq_{ f}   \ar@{^{(}->}[u]_{G_K}^{i} &  \gm_E \otimes \bz  \ar@{^{(}->}[u]_{} \ar[r]^-\simeq_{ g}  & \gm \otimes \bz \ar@{^{(}->}[u]_{}
}
 \end{equation}
Let us explain the content of this diagram:
\begin{enumerate}
\item Before we even define the objects and maps in the diagram, let us mention that all maps are obviously $\varphi$- and $N$-equivariant; we will hence focus on Galois equivariances of these maps.

\item We define $D_\st^E(V):=(V\otimes \bst)^{G_E}$ (note that $G_K$ acts on it since $K^{\ur}(\pi_{m-1})/K$ is Galois).

\item By Prop. \ref{prop543}, we have $G_E$-equivariant  isomorphisms 
\[V\otimes  \bx \simeq  D_{\st}^E(V) \otimes \bx   \simeq D^\dagger_{\rig, E_\infty}(V) \otimes \bx;\]
 they are furthermore $G_K$-equivariant by the construction of  $D_{\st}^E(V)$ and $D^\dagger_{\rig, E_\infty}(V) $.

\item  The vertical embeddings are just sub-modules.  

\item The isomorphisms $f$ and $g$ follow  from Cor. \ref{cor542} and Eqn. \eqref{mme} respectively.  
Since $i$ is $G_K$-equivariant (namely, $\bz$ is $G_K$-stable), hence $\gm_E \otimes \bz$ and $\gm \otimes \bz$ are $G_K$-stable. Thus, the maps $f, g$, and hence all the vertical embeddings are  indeed $G_K$-equivariant. (Thus we can put $G_K$ over \emph{all} the arrows in the diagram)
\end{enumerate}

So now $\gm \otimes \wtB_{\log}^+[1/t]$ is $G_K$-stable. By overconvergence theorem, we also have
$\gm \otimes \wtB_{}^\dagger$ is $G_K$-stable.
We have
\begin{eqnarray}\label{715}
 \wtB_{}^\dagger \cap  \wtB_{\log}^+  &=&\wtB_{}^\dagger \cap  \wtB_{\rig}^+  \text{ by taking } N=0 \text{ part }\\
&=& W(R)[1/p], \text{ by \cite[Lem. 2.18]{Ber02}}.
\end{eqnarray}
Hence  $\gm \otimes W(R)[1/p][1/t]$ is $G_K$-stable.

Now fix a  basis $\vec{m}$ of $\gm$. Suppose $\varphi(\vec{m})=\vec{m} A$ and let $B \in \Mat(\gs)$ such that $AB=E(u)^h$ for some $h \geq 0$. For any
$g\in G_K$, suppose $g(\vec{m})=\vec{m}M_g$ where $M_g = t^{-a}M$ with $M \in \Mat(W(R)[1/p])$ for some $a \geq 0$.  
Since $\varphi$ and $g$ commutes, we have $
A\varphi(M) =p^aMg(A)$, and hence
\begin{equation}
\varphi(M)E(u)^{h}=p^aBMg(A)
\end{equation} 
For a matrix $X$ defined over $W(R)$, let $v_p(X)$ be the minimum of the $p$-adic valuation of all its entries.
We then have $v_p(\varphi(M)E(u)^{h})=v_p(M)$ and   hence we must have $a=0$. This shows that  $\gm \otimes W(R)[1/p]$ is $G_K$-stable. Since   $\gm \otimes W(\Fr R)$ is also $G_K$-stable, we finally have $\gm \otimes W(R)$ is $G_K$-stable.
\end{proof}

 We introduce some notations   before we prove our main theorem.

\subsubsection{} Let $\nu: W(R) \to W(R)/W(\mathfrak{m}_R) =W(\overline{k})$ be the reduction ring homomorphism. It naturally extends to $\nu: \wtB^{[0, \frac{r_0}{p}]}  \to W(\overline{k})[1/p]$, e.g., by using the explicit expression of $ \wtB^{[0, \frac{r_0}{p}]}$ using Lem. \ref{lemwta}. It then extends to 
\[\nu: \wtB^{[0, \frac{r_0}{p}]}[\ell_u] \to W(\overline{k})[1/p]\]
by setting $\nu(\ell_u) =0$.
The map $\nu$ is a $\varphi$-equivariant ring homomorphism; it is furthermore $G_K$-equivariant by using \eqref{eq277new} (and note $\nu(t)=0$). For any $A \subset \wtB^{[0, \frac{r_0}{p}]}[\ell_u]$ a subring, and $M$ a finite free $A$-module, the $\nu$-map extends to 
\begin{equation}\label{eqnu}
\nu: M \to W(\overline{k})[1/p]\otimes_A M.
\end{equation}
If $A$ is $G_K$-stable and $M$ is equipped with a $A$-semi-linear $G_K$-action, then the map $\nu$ in \eqref{eqnu} is $G_K$-equivariant.

The following is our main theorem.

\begin{theorem}\label{thmsimple} 
The functor $T_{W(R)}$ induces an
equivalence  of categories 
\begin{equation}
\label{eqstrong}
\textnormal{Mod}_{\gs, W(R)}^{\varphi, G_K}
\xrightarrow{ }   \Rep_{\Zp}^{\st,\geq 0}(G_K).
\end{equation}
\end{theorem}

\begin{proof}
 \textbf{Part  1.}
We first show that if $\hM \in \textnormal{Mod}_{\gs, W(R)}^{\varphi, G_K }$, then $V:=T_{W(R)}(\hM)[1/p]$ is semi-stable, and hence $T_{W(R)}$ indeed induces the functor in \eqref{eqstrong}. 
Note that $V$ is of finite $E(u)$-height. Let $E, \gs_E, \gm_E$ etc. be exactly the same as in the proof of Thm. \ref{thmnoweak}. 
Similar to the  bottom row of \eqref{diag714}, we have   isomorphisms (where $E_0=W(\overline{k})[1/p]$):
\begin{equation}\label{714new}
D_{\st}^E(V) \otimes_{E_0} \by \simeq \gm_E \otimes_{\gs_E} \by \simeq \gm \otimes_\gs \by,
\end{equation}
using Cor. \ref{cor542} and \eqref{mme} respectively; again, they are $G_E$-equivariant \emph{a priori}, but are indeed $G_K$-equivariant since $D_{\st}^E(V) \otimes_{E_0} \by \subset D_{\st}^E(V)\otimes_{E_0} \bx$ is $G_K$-stable.
Apply the $G_K$-equivariant map $\nu$ to \eqref{714new}, we get $G_K$-equivariant isomorphisms:
\begin{equation}
D_{\st}^E(V) \simeq \gm_E/u_E\gm_E \otimes_{W(\overline{k})}W(\overline{k})[1/p] \simeq \gm/u\gm \otimes_{W(k)}W(\overline{k})[1/p].
\end{equation}
Hence $D_{\st}^E(V)$ is fixed by $G_{K^\ur}$, and hence $V$ is a semi-stable representation of $G_K$.

\textbf{Part 2.} Note that $T_{W(R)}$ is fully faithful by Prop. \ref{propweakff}.
We now show that $T_{W(R)}$ is essentially surjective. Let $T \in \Rep_{\Zp}^{\st,\geq 0}(G_K)$; it is of finite height by \cite{Kis06}. Hence by Thm. \ref{thmnoweak}, we can already get a unique $(\gm, \varphi, G_K) \in \textnormal{wMod}_{\gs, W(R)}^{\varphi, G_K}$. It suffices to show that $\gm/u\gm$ is fixed by $G_K$.
By \eqref{new545}, we have the $G_K$-equivariant isomorphism:
\[D_{\st}(T[1/p]) \otimes_{K_0} \by \simeq  \gm \otimes_\gs \by.\]
Applying the $\nu$-map shows that $\gm/u\gm\otimes_{W(k)} W(\overline{k})[1/p]$ is fixed by $G_{K^\ur}$, hence we can conclude using Lem. \ref{lemeq5}.
\end{proof}
 
 \begin{rem}
  By Prop. \ref{propweakff}, the functor $T_{W(R)}$ is exact. However, its quasi-inverse in Thm. \ref{thmsimple}  is in general only left exact; cf. the discussions in \cite[Lem. 2.19, Ex. 2.21]{Liu12}.
 \end{rem}


We now give some results related to the theory of Breuil-Kisin $G_K$-modules. In Prop. \ref{propcryscrit}, we give a crystallinity criterion; in Prop. \ref{cor710}, we prove the ``algebraic avatar" of the de Rham comparison \eqref{compadr} (cf.  Rem. \ref{rem1111}). We start with a lemma.

\begin{lemma} \label{lemtt}
We have $(\varphi(\mathfrak{t})  \cdot  \wtB^{[0, {r_0} ]}) \cap W(R) =\varphi(\mathfrak{t})  \cdot W(R)$.
\end{lemma}
\begin{proof}
Let $I^{[1]}W(R):= \{a \in W(R), \varphi^n(a) \in \Ker \theta, \forall n \geq 0\}.$
By the proof of \cite[Lem. 3.2.2]{Liu10} (using a result of Fontaine), we have
\begin{equation}\label{eqi1}
I^{[1]}W(R)=\varphi(\mathfrak{t})  \cdot W(R).
\end{equation}
 Note that the map $\theta \circ \iota_0: \wtB^{[0, {r_0} ]} \to \bdr^+ \to C_p$ (cf. \cite{Ber02} for the map $\iota_0$) induces the $\theta$ map on $W(R)$. Thus one easily checks that $(\varphi(\mathfrak{t})  \cdot  \wtB^{[0, {r_0} ]}) \cap W(R) \subset I^{[1]}W(R).$ 
\end{proof}

\begin{prop}\label{propcryscrit}
  Let $\hM \in \textnormal{Mod}_{\gs, W(R)}^{\varphi, G_K }$,
then $V=T_{W(R)}(\hM)[1/p]$ is a crystalline representation if and only if
\begin{equation}\label{taucris}
(\tau -1) (\gm) \subset  \mathfrak{t}W(\mathfrak{m}_R)\otimes_{\gs}\gm.
\end{equation}
\end{prop}
\begin{proof}
\textbf{Step 1}: we first show that for any $\hM \in \textnormal{Mod}_{\gs, W(R)}^{\varphi, G_K }$, we always have \begin{equation}\label{taust}
(\tau -1) (\gm) \subset  \mathfrak{t}W(R)\otimes_{\gs}\gm.
\end{equation}
Indeed, by \eqref{541}, we have 
\begin{equation}
\gm \otimes_\gs \mathcal{O}[\ell_u, 1/\lambda] =D\otimes_{K_0}  \mathcal{O}[\ell_u, 1/\lambda] \subset D \otimes_{K_0} \wtB^{[0, \frac{r_0}{p}]}[\ell_u].
\end{equation}
Note that by Thm. \ref{thmLAVmain}(1),
\[ \mathcal{O}[1/\lambda]  \subset \B_{\Kinfty}^{[0, \frac{r_0}{p}]} = \left( \wtB_L^{[0, \frac{r_0}{p}]} \right)^{\tau\dan, \gamma=1};\]
also by Lem. \ref{lem266}, $\ell_u^k$ are $\tau$-\emph{analytic} vectors. 
Hence for some $k \gg 0$, we have
\begin{equation}\label{7112}
\gm \subset D\otimes_{K_0}  \left(\oplus_{i=0}^k  \mathcal{O}[1/\lambda] \cdot\ell_u^i \right)
 \subset \left(D_{K_0} \otimes \left(\oplus_{i=0}^k  \wtB_L^{[0, \frac{r_0}{p}]} \cdot\ell_u^i \right) \right)^{\tau\dan, \gamma=1} =(D\otimes Q)^{\tau\dan, \gamma=1};
\end{equation}
where for brevity, we denote
\[Q:= \oplus_{i=0}^k  \wtB_L^{[0, \frac{r_0}{p}]} \cdot\ell_u^i. \]
This means that over $\gm$, we have
 \begin{equation*}
 \tau = \sum_{i=0}^\infty \frac{ \nabla_\tau  ^i}{i!},
 \end{equation*}
  where $\nabla_\tau$ equals to 
 $p\mathfrak t N_\nabla$ or  $p^2\mathfrak t N_\nabla$.
Since $N_{\nabla}(\gm) \subset \gm \otimes_\gs \mathcal{O}$, the map $\nabla_\tau$ induces
\begin{equation*}
\nabla_\tau :\gm    \to \gm \otimes_\gs \mathfrak{t} \cdot\mathcal{O}.
\end{equation*} 
One easily checks that $\nabla_\tau(\mathfrak{t}^k) \in \mathfrak{t}^{k+1}\cdot \mathcal{O}$; also note $\nabla_\tau(\mathcal{O}) \in \mathfrak{t} \cdot \mathcal{O}$. Hence inductively, we can show that for all $i \geq 1$,
 \begin{equation}\label{eq7119new}
\nabla_\tau^i :\gm   \to \gm \otimes_\gs \mathfrak{t}^i\cdot\mathcal{O}.
\end{equation} 
Hence if we choose a basis of $\gm$, then for any $m \in \gm$, the coefficient of $(\tau -1)(m)$ inside $\gm \otimes Q$, expressed using that basis, lies in $\mathfrak{t}\cdot  \wtB^{[0, \frac{r_0}{p}]}  \cap W(R) = \mathfrak{t}W(R)$  by Lem. \ref{lemtt}.
This finishes the proof of \eqref{taust}.

\textbf{Step 2}: suppose now $V$ is  crystalline.  To show \eqref{taucris}, it   suffices to show that
\begin{equation} \label{eq7120}
\nu  \left( \frac{\tau-1}{\mathfrak{t}} (\gm)\right) =0.
\end{equation}
Here, the expression \eqref{eq7120} (and similar expressions below) means that the image of $ \frac{\tau-1}{\mathfrak{t}} (\gm)$ is zero under the map 
\[\nu: \gm \otimes_\gs \wtB^{[0, \frac{r_0}{p}]}[\ell_u] \to \gm\otimes_\gs W(\overline{k})[1/p].\]
Note that 
\begin{equation*}
\frac{\tau-1}{\mathfrak{t}} = \frac{\nabla_\tau}{\mathfrak{t}}+\sum_{i \geq 2}\frac{1}{\mathfrak{t}} \cdot \frac{\nabla_\tau^{i }}{i!}.
\end{equation*}
Note that $\nu (N_\nabla(\gm))=0$ (since $N_\nabla/uN_\nabla=N_{D_{\st}(V)}=0$ as $V$ is crystalline), hence $\nu (\frac{\nabla_\tau}{\mathfrak{t}}(\gm))=0$. For each $i \geq 2$, $\nu(\frac{1}{\mathfrak{t}} \cdot \frac{\nabla_\tau^{i }}{i!}(\gm))=0$ by Eqn. \eqref{eq7119new} since $\nu(\mathfrak{t})=0$.
  This concludes the proof of \eqref{eq7120}.

\textbf{Step 3}:
Conversely, suppose now \eqref{taucris} is satisfied.
Note that \eqref{eqi1} implies that $\mathfrak{t}W(R)$ and hence $\mathfrak{t}W(\mathfrak{m}_R)$ are $G_K$-stable.
Hence for any $a \geq 0$, we also have 
\[(\tau^{p^a} -1) (\gm) \subset  \mathfrak{t}W(\mathfrak{m}_R)\otimes_{\gs}\gm.\]
Using the definition $N_\nabla=\frac{\nabla_\tau}{p\mathfrak t}$ (or $\frac{\nabla_\tau}{p^2\mathfrak t}$),  one can easily show $\nu(N_\nabla(\gm))=0$ and hence $N_{D_\st(V)}=0$. 
\end{proof}

\begin{prop} \label{cor710}
Let $\hM$ be an object   in $\textnormal{Mod}_{\gs, W(R)}^{\varphi, G_K }$. Let $\varphi^\ast \gm: =\gs \otimes_{\varphi, \gs} \gm$ and $\varphi^\ast \hM: =W(R) \otimes_{\varphi, W(R)} \hM$, then 
\begin{equation}\label{taudr}
\varphi^\ast \gm/E(u)\varphi^\ast \gm \subset (\varphi^\ast\hM/E(u)\varphi^\ast \hM)^{G_K}.
\end{equation} 
\end{prop}
\begin{proof}
By \eqref{taust}, $ \varphi^\ast \gm/E(u)\varphi^ \ast \gm  $ is fixed by $\tau$ since $\varphi(\mathfrak{t})\cdot W(R) \subset E(u)\cdot W(R)$. This proves \eqref{taudr} when $\Kinfty \cap \Kpinfty=K$. In Thm. \ref{thmf11}, we will  show (no circular reasoning here) that the subset $\varphi^\ast \gm/E(u)\varphi^\ast \gm$ is indeed independent of    choices of $\Kinfty$, and hence \eqref{taudr} holds in general (as there exists \emph{some} choice such that $\Kinfty \cap \Kpinfty=K$; cf. Notation \ref{nota hatG}). 
\end{proof}

\subsection{Specialization to Liu's $(\varphi, \hat{G})$-modules}\label{72}
In \cite{GaoLAV}, we will show that using some ``specialization" maps, our argument and results  recover that of Liu's theory  of  $(\varphi, \hat{G})$-modules. (We only give a quick review of the $(\varphi, \hat{G})$-modules later in  Appendix \ref{secappB}, as we do not really use them here.)
The proof in \cite{GaoLAV}   makes systematic use of locally analytic vectors, and makes the link between the different theories completely transparent.
 Here, let us give some hint about what we mean by ``specialization"; cf. \emph{loc. cit.} for more details.

Recall that if $r_n \in I$, then there are
  continuous embeddings (cf. \cite{Ber02})
\begin{equation*}
\iota_n: \wt{\mathbf{B}}^{I} \hookrightarrow \mathbf{B}_\dR^+;
\end{equation*}  which we call \emph{de Rham specialization  maps}.
One easily checks that  
the image of the embedding
$\iota_0: \wt{\mathbf{B}}^{[0, r_1]}  \hookrightarrow \bdr^{+} $
lands inside $\bcris^{+}$. Furthermore, the induced embedding
\begin{equation}\label{eqembcris}
 \iota_0: \wt{\mathbf{B}}^{[0, r_1]}  \hookrightarrow  \bcris^{+}
\end{equation}
 is \emph{continuous}.
 The map \eqref{eqembcris} also induces the continuous  composites:
 \begin{equation} \label{eqspcr}
 \varphi:  \wtb^{[0, r_0]}  \xrightarrow{\varphi} \wtb^{[0, r_1]}  \xrightarrow{\iota_0}  \bcris^{+}.
 \end{equation}
 We call the maps in \eqref{eqembcris} and \eqref{eqspcr} the \emph{crystalline specialization maps}.
By adjoining $\ell_u$, we also get the continuous   \emph{semi-stable specialization maps} 
\begin{eqnarray}
\iota_0 &:& \wt{\mathbf{B}}^{[0, r_1]}[\ell_u]  \hookrightarrow  \bst^{+},\\
\varphi &:& \wtb^{[0, r_0]}[\ell_u] \hookrightarrow  \bst^{+}.
\end{eqnarray}

Note that in \eqref{new545}, we have (for a semi-stable representation):
\[\gm \otimes_{\gs} \wtB^{[0, \frac{r_0}{p}]}[\ell_u] = D\otimes_{K_0}  \wtB^{[0, \frac{r_0}{p}]}[\ell_u];\]
we are using $\wtB^{[0, \frac{r_0}{p}]}$ because it has the advantage that $\lambda$ is a unit in it, and indeed $1/\lambda$ is an $\tau$-\emph{analytic} vector in it.
But from \eqref{541}, we also have
\begin{equation}\label{nn541}
\gm \otimes_{\gs} \wtB^{[0, r_0]}[1/\lambda, \ell_u] = D\otimes_{K_0}  \wtB^{[0, r_0]}[1/\lambda, \ell_u];
\end{equation}
note however these modules are not $G_K$-stable!
However, using the semi-stable specialization map (note here $\varphi(\lambda)$ is a unit in $\bcris^+$, but $\lambda$ is not)
\begin{equation*}
\varphi: \wtB^{[0, r_0]}[1/\lambda, \ell_u] \to      \bst^+,
\end{equation*}
we get a $G_K$-equivariant identification
\begin{equation}\label{new727}
\gm \otimes_{\varphi, \gs} \bst^+ = D\otimes_{\varphi, K_0}  \bst^+.
\end{equation}
Let \begin{equation} \label{defS}
S := \{\sum_{n=0}^{\infty}a_{n}\frac{E(u)^n}{n!}, a_{n}\in \gs, a_{n} \to 0 \: p\text{-adically} \} \subset \Acris.
\end{equation} 
Then it is easy to check that elements of $S$ are all $\tau$-\emph{analytic} vectors in $\mathbf{B}_{\cris, L}^+:=(\bcris^+)^{G_L}$. Hence by \eqref{new727} and \eqref{7112}, we have that for $k \gg 0$, 
\begin{equation}
\gm \otimes_{\varphi, \gs} S \subset \left(  D\otimes_{K_0} \left(   \oplus_{i=0}^k {\mathbf{B}_{\cris, L}^+}\cdot \ell_u^i   \right)\right)^{\tau\dan, \gamma=1}.
\end{equation}
  Hence $\tau = \sum_{i=0}^\infty \frac{ \nabla_\tau  ^i}{i!}$ again holds over $\gm \otimes_{\varphi, \gs} S$. Using some results about ``filtered $(\varphi, N)$-modules over $S$" in \cite{Bre97}, this easily recovers the Galois action on $\gm \otimes_{\varphi, \gs} \bcris^+$ as in \cite[\S 5.1]{Liu08}. In particular, rather than ``defining" the Galois action in an \emph{ad hoc} fashion as in \emph{loc. cit.} and then showing it is \emph{compatible} with Galois representations (cf. \cite[Lem. 5.2.1]{Liu08}, our specialization obviously implies the compatibility. Let us mention that this Galois action is indeed one of the keys in Liu's theory of $(\varphi, \hat{G})$-modules.

\subsection{Relation with some results of  Gee and Liu}\label{73}
In this subsection, we discuss the relation between our results and some recent results of  Toby Gee and Tong Liu. 
 
The statement and idea of proof of the following Thm. \ref{thmtwou} is due to Toby Gee. As we learnt from Gee, this result is inspired by Caruso's result, and was originally used to construct the semi-stable substack inside the stack of \'etale $(\varphi, \Gamma)$-modules, cf. \cite[Appendix F]{EG23}. We thank Toby Gee for allowing us to include it here.

\begin{thm}(Gee) \label{thmtwou}
Let $T \in \Rep_{\Zp}(G_K)$.
Then $T  \in \Rep_{\Zp}^{\st, \geq 0}(G_K)$ \emph{if and only if} $T$  is of finite height with respect to \emph{all} choices of $\vec{\pi}$.
\end{thm}

We first introduce an elementary lemma.
\begin{lemma}\label{lemdisj}
Let $K \subset M_1, M_2$ be finite extensions, and let $M=M_1M_2$. Suppose $M/K$ is Galois and totally ramified, and suppose $M_1\cap M_2=K$. Let $V$ be a $G_K$-representation and suppose it is semi-stable over both $M_1$ and $M_2$, then $V$ is semi-stable over $K$.
\end{lemma}
\begin{proof}
It is easy to see that $\gal(M/K)$ acts trivially on $D_{\st}^M(V)=(V\otimes_{\Qp}\bst)^{G_M}$, and hence $D_{\st}^M(V)=D_{\st}^K(V)$.
\end{proof}

\begin{proof}[Proof of Thm. \ref{thmtwou}]
Necessity is proved in \cite{Kis06}. We prove sufficiency.
Let $V:=T[1/p]$, let $U:=V|_{G_{K^{\ur}}}$, and let $\wh{K^{\ur}}$ denote the completion of $K^{\ur}$.
It suffices to show that the $G_{\wh{K^{\ur}}}$-representation $U$ is semi-stable.
First fix one $\vec{\pi}$. Choose another unifomizer $\pi'$ of  $K$  such that $\pi/\pi' \in \mathcal{O}_{\wh{K^{\ur}}}^\times  \backslash (\mathcal{O}_{\wh{K^{\ur}}}^\times)^p$, and choose any compatible system  $\vec{\pi'}=\{\pi_n'\}_{n \geq 0}$.
By Kummer theory, for each $i \geq 1$, the two fields $\wh{K^{\ur}}(\pi_i)$ and $\wh{K^{\ur}}(\pi'_i)$ are different.
Combined with \cite[Lem. 4.1.3]{Liu10}, it is easy to show that
\[\wh{K^{\ur}}(\pi_i) \cap \wh{K^{\ur}}(\pi'_i)=\wh{K^{\ur}}, \forall i \geq 1.\]
Let $m:=m(K^{\mathrm{ur}})$. Consider the 3-step extensions
\[\wh{K^{\ur}} \subset  \wh{K^{\ur}}(\pi_{m-1}),  \wh{K^{\ur}}(\pi'_{m-1})   \subset M=\wh{K^{\ur}}(\pi_{m-1}, \pi'_{m-1}).\]
Since $\mu_{m-1}\in \wh{K^{\ur}}$, $M/\wh{K^{\ur}}$ is Galois and totally ramified.
By Thm. \ref{thmmainpst}, $V$ is semi-stable over $K(\pi_{m-1})$, hence $U$ is semi-stable over $\wh{K^{\ur}}(\pi_{m-1})$;  similarly $U$ is semi-stable over $\wh{K^{\ur}}(\pi_{m-1}')$. Thus  $U$ is semi-stable over $\wh{K^{\ur}}$ by Lem. \ref{lemdisj}.
\end{proof}

Before the author proves Thm. \ref{thmmainpst} which makes Thm. \ref{thmtwou} possible, Gee and Liu proved the following weaker Thm. \ref{thmf11} which is sufficient for the construction of the semi-stable sub-stack mentioned above. 
 In the following Def. \ref{defcss}, the notation $\mathcal C_{ss}$ comes from ``$\mathcal C_{d,ss,h}$" in \cite[Def. 4.5.1]{EG23}; here, we  omit  the rank $d$ and the height $h$. For each choice $\vec{\pi}=\{\pi_n\}_{n \geq 0}$, we let $K_{\vec{\pi}}:=\cup_{n \geq 0}K(\pi_n)$. We can regard $\vec{\pi}$ as an element in $R$, and let $[\vec{\pi}] \in W(R)$ be its Teichm\"uller lift. Let $\gs_{\vec{\pi}}$ be the image of $W(k)[\![X]\!] \hookrightarrow W(R), X \mapsto [\vec{\pi}]$. Let $E(X)$ be the minimal polynomial of $\pi_0$ over $K_0$, and let $E_{\vec{\pi}}:=E([\vec{\pi}]) \in \gs_{\vec{\pi}}$.

\begin{defn} \label{defcss}  \cite[Def. F.7]{EG23}
Let $\mathcal C_{ss}(\Zp)$ be the category consisting of the following data, which are called ``\emph{Breuil-Kisin-Fargues $G_K$-modules admitting all descents}":
\begin{enumerate}
\item  $\gm^{\inf}$ is a finite free Breuil-Kisin-Fargues module with $W(R)$-semi-linear $\varphi$-commuting $G_K$-action;
\item For each $\vec{\pi}$, $\gm_{\vec{\pi}} \in (\gm^{\inf})^{\gal(\overline{K}/K_{\vec{\pi}})}$ is a finite free Breuil-Kisin module  over $\gs_{\vec{\pi}}$, such that the induced morphism  $\gm_{\vec{\pi}} \otimes_{\gs_{\vec{\pi}}} W(R) \to \gm^{\inf}$ is a $\varphi$-equivariant isomorphism;
\item The $W(k)$-mod $\gm_{\vec{\pi}}/[\vec{\pi}]\gm \subset W(\overline k)\otimes_{W(R)} \gm^\inf $ is independent of $\vec{\pi}$.
\item The $\mathcal O_K$-mod $\varphi^\ast \gm_{\vec{\pi}}/ E_{\vec{\pi}}\varphi^\ast\gm   _{\vec{\pi}} \subset \mathcal O_C\otimes_{W(R)} \varphi^\ast \gm^\inf $  is independent of $\vec{\pi}$.
\end{enumerate}
 \end{defn}

\begin{theorem} (Gee-Liu) \cite[Thm. F.11]{EG23} \label{thmf11}
The functor $T_{W(R)} (\gm^{\inf})  :=  ( \gm^{\inf} \otimes W(\Fr R))^{\varphi=1}$  induces an equivalence between $\mathcal C_{ss}(\Zp)$ and $\Rep_{\Zp}^{\st, \geq 0}(G_K)$.
\end{theorem}
The proof of Gee-Liu makes use of a result of Heng Du (cf. \cite[Prop. F.13]{EG23}), which shows that the Conditions (1), (2), and (4) are enough to guarantee the attached representation is \emph{de Rham}; then they use Condition (3) to show semi-stability. 
Here, we  give a very brief sketch of proof using our results in order to illustrate the relation between the different approaches.

\begin{proof}
Apparently, given a module in $\mathcal C_{ss}(\Zp)$, the associated representation is semi-stable by Thm. \ref{thmtwou}. Conversely, give a semi-stable representation, the $\gm_{\vec{\pi}}$'s are constructed by Kisin. To verify Condition (2) (in Def. \ref{defcss}), it suffices to check that the various tensor products ``$\gm_{\vec{\pi}} \otimes_{\gs_{\vec{\pi}}} W(R)$" can all be \emph{identified} inside $V\otimes \bx$. Note $\gm_{\vec{\pi}} \otimes_{\gs_{\vec{\pi}}} \bz$ can be identified with $D\otimes \bz$, and $\gm_{\vec{\pi}} \otimes \wtb^\dagger$ can be identified with $V\otimes \wtb^\dagger$; 
hence by \eqref{715}, $\gm_{\vec{\pi}} \otimes W(R)[1/p][1/t]$ can all be identified with each other. Then one can follow similar strategy as below \eqref{715} to show that $\gm_{\vec{\pi}} \otimes W(R)$ can all be identified. (We leave the details to the readers).
For Condition (3): note that the independence of $\gm_{\vec{\pi}} \otimes W(R)$ implies the independence of $\gm_{\vec{\pi}}/[\vec{\pi}]\gm_{\vec{\pi}} \otimes W(\overline{k})$, and hence it suffices to show the independence of $\gm_{\vec{\pi}}/[\vec{\pi}]\gm_{\vec{\pi}}[1/p] =M_{\vec{\pi}}/[\vec{\pi}]M_{\vec{\pi}}$ where $M_{\vec{\pi}}$ are the $\mathcal{O}_{\vec{\pi}}$-modules as constructed in \S \ref{subsecKis}. Then Lem. \ref{lem541} implies that $M_{\vec{\pi}}/[\vec{\pi}]M_{\vec{\pi}}$ can all be identified with $D$ (which is well-known to be independent of $\vec{\pi}$). The verification of Condition (4) is similar, by using the $\varphi$-twist of \eqref{new545}. 
\end{proof}



\subsection{Independence from Caruso's work}
\begin{rem}\label{remend}
 \begin{enumerate}
   \item The only place in the current paper where we actually \emph{use} results from \cite{Car13} is in Thm. \ref{thmphitau} (also Def. \ref{defphitaumod}) which is \cite[Thm. 1]{Car13}. (Lem. \ref{lemramres} is also from \cite{Car13}, but it is completely elementary).

   \item Besides those in \S \ref{secintro} and in Thm. \ref{thmphitau}, the only places where we \emph{make references to} \cite{Car13} is in Rem. \ref{remstep3}, where we mention  some  argument    from \cite{Car13}.
   \item The results in  \cite{Car13} are not  used  in any of the cited papers in our bibliography, except \cite[Thm. 1]{Car13}, which is used in \cite{GL20, GP21}.
       \item Hence,  the current paper is independent of the work \cite{Car13} (except \cite[Thm. 1]{Car13}).
 \end{enumerate}
  \end{rem}

\appendix

\section{A gap in Caruso's work, by Yoshiyasu Ozeki} \label{secapp}
The statement of \cite[Prop. 3.7]{Car13} is false.
Here is a counter-example:
\begin{itemize}
\item Let $A=\mathbb{Q}_p$ with $p=3$, let $m=4$ and let $\Lambda=p^{-4}\cdot \mathbb{Z}_p$. Let $a=b=0$.
\end{itemize}
The mistake in the proof of \cite[Prop. 3.7]{Car13} is rather hidden. Indeed, on \cite[p. 2580]{Car13}, Caruso states the formula:
\[\log_m(ab) - (\log_m a + \log_m b) \in \sum_{j=1}^{p^m-1}
\frac{p^{\ell(p^m-j) + \ell(j)}} j \cdot \Lambda.\]
As far as we understand, Caruso is implicitly using that $\Lambda$ is a \emph{ring} (namely, a $\mathbb{Z}_p$-algebra). But in fact, $\Lambda$ is only a $\mathbb{Z}_p$-module.

Prop. 3.7 of \cite{Car13} is a key proposition in the paper. Indeed, it is used to prove Prop. 3.8, Prop. 3.9 in \emph{loc. cit.}. These propositions are then  repeatedly used in later argument of that paper.

\section{Errata for \cite{Liu07} and \cite{Gao18limit}, by Hui Gao and Tong Liu} \label{secappB}

In this appendix, using the Breuil-Kisin  $G_K$-modules,  we fix a gap in the proof of the main theorems 
in \cite{Liu07} and \cite{Gao18limit} where we studied the limit of torsion semi-stable representations. Let us point out at the start that the gap arises because we   only recently realized that we \emph{do not know} if the ring $\hR$ (recalled below) is $p$-adically complete or not. (The gap is discussed in detail  in Step 2 of \S \ref{b04}.)

We first recall the main results of \cite{Liu07} and \cite{Gao18limit}.
Recall that a (finite) $p$-power torsion representation $T$ of $G_K$ is called \emph{torsion semi-stable (resp. crystalline) of weight $r$}, if there exists a $G_K$-stable $\Zp$-lattice $\wt L$ in a semi-stable (resp. crystalline) representation with Hodge-Tate weights  in $[0, r]$, and such that there exists a $G_K$-equivariant surjection $\wt L \twoheadrightarrow T$ (which is called a \emph{loose semi-stable (resp. crystalline) lift}).
The following is the main theorem of \cite{Gao18limit}:

\begin{thm}  \label{thmGao}
Let $T$ be a finite free $\Zp$-representation of $G_K$ of   rank $d$. For each $n \ge 1$, suppose $T_n:=T/p^nT$ is torsion semi-stable (resp. crystalline) of weight $h(n)$.
If \[h(n) < \frac{1}{2d}\log_{16} n, \forall n \gg 0,\]
 then $T\otimes_{\Zp}\Qp$ is semi-stable (resp. crystalline).
\end{thm}

When $h(n)$ is a constant, then this is precisely the main theorem of \cite{Liu07}, which confirms a conjecture of Fontaine. 

Unfortunately, there is a gap in the proof of Thm. \ref{thmGao}; there is a (practically the same) gap in \cite{Liu07} as well, cf. Rem. \ref{remliu}.
We now focus on discussing and fixing the gap in the proof of Thm. \ref{thmGao} (which utilizes a very similar strategy as in \cite{Liu07}).

\subsubsection{}
First, we quickly recall the theory of $(\varphi, \hat{G})$-modules, cf. \cite{Liu10} for more details.
Define a subring inside $\bcris^+$:
\[\mathcal R_{K_0}: = \{x = \sum_{i=0 }^\infty f_i t^{\{i\}}, f_i \in S_{K_0} \text{
and } f_i \to 0\text{ as }i \to +\infty  \}, \] where $t ^{\{i\}}= \frac{t^i}{p^{\tilde q(i)}\tilde q(i)!}$ and $\tilde q(i)$ satisfies $i = \tilde q(i)(p-1) +r(i)$ with $0 \leq r(i )< p-1$.
 Define \[\hR := W(R)\cap \mathcal{R}_{K_0}.\] The rings $\mathcal R_{K_0}$ and $\hR$ are stable under the $G_K$-action and the $G_K$-action factors through  $\hat G$.
Let $I_+ \hR = W(\mathfrak{m}_R) \cap \hR$, then one has
$\hR / I_+\hR\simeq \gs/ u \gs  = W(k)$. 

\begin{defn}\label{defptau}
Let $\textnormal{Mod}_{\gs, \hR}^{\varphi, \hat{G} }$
be the category consisting of triples $(\mathfrak{M} , \varphi_{\gm}, \hat G)$ which are called \emph{$(\varphi, \hat G)$-modules}, where
\begin{enumerate}
\item $(\mathfrak{M}, \varphi_\mathfrak{M})\in \text{Mod}_{\gs}^{\varphi}$;
\item $\hat G$ is a continuous $\hR$-semi-linear $\hat G$-action on $\widehat{\mathfrak{M}}: =\hR
\otimes_{\varphi, \gs} \mathfrak{M}$;
\item $\hat G$ commutes with $\varphi_{\widehat{\mathfrak{M}}}$ on $\widehat{\mathfrak{M}}$;
\item regarding $\mathfrak{M}$ as a $\varphi(\gs)$-submodule in $ \widehat{\mathfrak{M}} $, then $\mathfrak{M}
\subset \widehat{\mathfrak{M}} ^{\gal(L/\Kinfty)}$;
\item $\hat G$ acts on the $W(k)$-module $ \widehat{\mathfrak{M}}/I_+\hR\widehat{\mathfrak{M}}$ trivially.
\end{enumerate}
\end{defn}

Then (the co-variant version of) the main theorem of \cite{Liu10} says that the functor $T_{W(R)}(\hM):=(\hM \otimes_{\hR} W(\Fr R))^{\varphi=1}$ induces an equivalence between $\textnormal{Mod}_{\gs, \hR}^{\varphi, \hat{G} }$ and $\Rep_{\Zp}^{\st, \geq 0}(G_K)$.

\subsubsection{} \label{b04}
We now sketch the  proof of Thm. \ref{thmGao} in \cite{Gao18limit}  in two steps; it is in Step 2 that the gap arises.

\textbf{Step 1}:  we first show that $T$ is of finite $E(u)$-height. As we need the constructions in our Step 2 (where the gap arises), let us sketch the argument.

Since $T_n$ is torsion semi-stable, one can construct a (not necessarily unique) \emph{$p$-power torsion $(\varphi, \hat{G})$-module} (the definition of which is obvious) $\wh{\gm}_n$, simply by projecting down the finite free  $(\varphi, \hat{G})$-module associated to some loose semi-stable lift of $T_n$. 

Note that \emph{a priori}, these   $\wh{\gm}_n$'s for different $n$ have no direct relations.
The technical heart in the proof of 
\cite[Thm. 3.1]{Gao18limit} is that we can \emph{modify} the
 torsion Breuil-Kisin modules inside these torsion $(\varphi, \hat{G})$-modules  so that in the end we can obtain a \emph{compatible} system of $\gm_n'$ for $n \gg 0$, such that
 \begin{itemize}
\item $\gm_n'$ is finite free over $\gs_n:=\gs/p^n\gs$; 
\item  Let $\OE_n:=\OE/p^n\OE$, and let $M_n$ be the torsion \'etale $\varphi$-module associated to $T_n|_{G_\infty}$, which is a finite free $\OE_n$-module, then $\gm_n' \otimes_{\gs_n} \OE_n \simeq M_n$, i.e., $\gm_n'$ is a ``\emph{Breuil-Kisin model}" of $M_n$;
\item $\gm_{n+1}'/p^n =\gm_n'$.
\end{itemize}
Hence we can form the inverse limit $\wt{\gm}: =\projlim_{n \gg 0} \gm_n'$, which is a finite free $\gs$-module. Using some techniques related with Weirestrass preparation theorem, we can in fact show that $\wt{\gm}$ is of finite height (this is easy when $h(n)$ is constant, but more difficult in the general case). Furthermore, we obviously have $T_{\gs}(\gm) \simeq T|_{G_\infty}$.

\textbf{Step 2}: Now it remains to show that the $\hat{G}$-action on $\wt{\gm} \otimes_{\varphi, \gs}
\hR$ is stable, and hence $T$ indeed comes from a $(\varphi, \hat{G})$-module and hence is semi-stable.

This is where the \emph{gap} arises. Indeed, we know that $\hat{G}$-action on $\gm_n' \otimes_{\varphi, \gs}
\hR$ is stable, because it comes from projection from the $\hat{G}$-action of a finite free $(\varphi, \hat{G})$-module. However, unfortunately, we only recently realized that we \emph{do not know} if $\hR$ is $p$-adically complete or not! Hence we cannot directly conclude that $\wt{\gm} \otimes_{\varphi, \gs} \hR$ is $\hat{G}$-stable!
Indeed, recall that $\hR =W(R) \cap \mathcal R_{K_0}$. In fact, it is not even clear \emph{which} ``$p$-adic topology" we should use here: should it be the one induced from $\bcris^+$, or the one induced from $W(R)$? In either of these choices, it is very difficult to actually compute the $p$-adic valuations of elements in $\hR$.

\begin{rem}\label{remliu}
Note that in the proof of \cite[Prop. 6.1.1]{Liu07}, it is implicitly assumed that 
$\mathcal R_{K_0} \cap \Acris$ is $p$-adically complete; again, it is not clear how to actually prove this: the difficulty is the same as that for $\hR$.
\end{rem}

\subsubsection{Fixing the gap using Breuil-Kisin $G_K$-modules}  To fix the gap, we can simply replace all the mentioning of ``$(\varphi, \hat{G})$-modules" above by ``Breuil-Kisin $G_K$-modules",  all the $\hR$ by $W(R)$, and all the $\hat{G}$ by $G_K$. Since $W(R)$ is $p$-adically complete, we readily conclude that $\wt{\gm} \otimes_{\gs} W(R)$ is $G_K$-stable! Furthermore, in each torsion level for $n \gg 0$, we know $\gm_n'/u\gm_n'$ is fixed by $G_K$ (again because the $G_K$-action come from that on a finite free Breuil-Kisin $G_K$-module); hence $\wt{\gm}/u\wt{\gm}$ is also fixed by $G_K$, simply because $W(k)$ is $p$-adically complete. Hence we have shown that $T$ indeed comes from a finite free Breuil-Kisin $G_K$-module, and hence is semi-stable. If all the $T_n$ are furthermore torsion crystalline, then the torsion version of the condition  in Prop. \ref{propcryscrit} holds, and hence the condition also holds for $\wt{\gm}$, again because all the rings (and ideals) in \emph{loc. cit.} are $p$-adically complete.

\begin{rem}
As we can observe from the above paragraph, in order to fix the gap in \cite{Liu07, Gao18limit}, it suffices to use some ``integral $p$-adic linear-algebra category" where all the rings involved are $p$-adic complete. Thus, one can also fix the gap in \cite{Liu07, Gao18limit} using Thm. \ref{thmtwou} or Thm. \ref{thmf11}; we leave the details to the reader.
\end{rem}

 \bibliographystyle{alpha}

\end{document}